\documentclass[12pt]{article}
\usepackage[T2A]{fontenc}
\usepackage[utf8]{inputenc}
\usepackage[english]{babel}
\usepackage{amsfonts,amsmath,amsthm,amssymb,mathrsfs}
\usepackage[nottoc]{tocbibind}
\usepackage{babelbib}
\usepackage{color}

\numberwithin{equation}{section}

\newtheorem{rem}{Remark}[section]
\newtheorem{thm}{Theorem}[section]
\newtheorem{lem}{Lemma}[section]
\newtheorem{cor}{Corollary}[section]

\theoremstyle{definition}

\def\R{\mathbb{R}}
\def\dist{{\rm dist}\,}

\def\Xint#1{\mathchoice
    {\XXint\displaystyle\textstyle{#1}}%
     {\XXint\textstyle\scriptstyle{#1}}%
     {\XXint\scriptstyle\scriptscriptstyle{#1}}%
     {\XXint\scriptstyle\scriptscriptstyle{#1}}%
	\!\int}
\def\XXint#1#2#3{{\setbox0=\hbox{$#1{#2#3}{\int}$}	\vcenter{\hbox{$#2#3$}}\kern-.5\wd0}}

\title{Around the normal derivative lemma}
\author{Darya~E.~Apushkinskaya\footnote{RUDN University and St. Petersburg State University}, 
\setcounter{footnote}{6}
Alexander~I.~Nazarov\footnote{PDMI RAS and St. Petersburg State University}}
\date{}

\begin{document}



\maketitle

\begin{abstract}
This survey provides a description of the history and the state of the art of one of the most important fields in the qualitative theory of elliptic partial differential equations including the strong maximum principle, the boundary point principle (the normal derivative lemma) and related topics.\medskip

\medskip


\noindent
\end{abstract}

\tableofcontents

\section{Introduction}

The qualitative theory of partial differential equations has been intensively developed over the past century. Among the most important tools for studying solutions of elliptic and parabolic equations are, in particular, the normal derivative lemma (also known as the Hopf--Oleinik lemma or the boundary point principle) and the strong maximum principle. They play a key role in proving uniqueness theorems for boundary value problems. They are also used in studying the symmetry properties of solutions, the behavior of solutions in unbounded domains (Phragm\'en--Lindel\"of type theorems), and in other applications. 
\medskip\medskip

The first results in this area can be traced back to the works of C.F.~Gauss, who proved the strong maximum principle for harmonic functions in 1840, in Section 21 of the famous paper \cite{G1840E}, see also \cite{G1839E}, \cite{E1839}. In modern notation, the Gauss statement reads as follows:\medskip

\textit{Let $u$ be a non-constant harmonic function in a domain $\Omega \subset \R^3$, that is $\Delta u=0$ in $\Omega$. Then the function $u$ attains neither maximum nor minimum in the interior points of $\Omega$.}
\medskip

In what follows, by strong maximum principle for a second-order linear elliptic operator ${\mathbb L}$ we mean the following assertion:\medskip

\noindent
{\bf The strong maximum principle.}
\textit{Let $u$ be a super-elliptic function in a domain $\Omega \subset \R^n$, that is\footnote{We assume that the principal coefficients of the operator ${\mathbb L}$ form a {\bf\textit{non-positive}} matrix.} ${\mathbb L} u \geq 0$ в $\Omega$. If $u$ attains its minimum at an interior point of the domain then 
$u \equiv const$ and ${\mathbb L} u \equiv 0$.}\medskip

We also recall the formulation of the weak maximum principle: \medskip

\noindent
{\bf The weak maximum principle.}
\textit{Let $u$ be a super-elliptic function in a bounded domain $\Omega \subset \R^n$. If $u$ is non-negative on the boundary of $\Omega$ then $u$ is non-negative in $\Omega$.
}
\medskip

The boundary version of the strong maximum principle is the so-called normal derivative lemma, first formulated by S.~Zaremba in 1910 \cite{Z1910E} for harmonic functions in a (three-dimensional, bounded) domain satisfying the interior ball condition\footnote{Notice that Zaremba used this lemma to prove the uniqueness theorem for a mixed problem (the boundary of the domain is split into two parts, one of which is subject to the Dirichlet condition and the other to the Neumann condition). Nowadays, it is called the Zaremba problem, although Zaremba himself in \cite{Z1910E} points out that it was posed to him by W.~Wirtinger.}.  \medskip

\noindent
{\bf The normal derivative lemma.}
{\it Let $u$  be a non-constant super-elliptic function in a domain $\Omega \subset \R^n$. If $u$ attains its minimum at a boundary point $x^0\in \partial\Omega$ then the following inequality holds:
\begin{equation}
\liminf\limits_{\varepsilon\to+0}\,\frac  {u(x^0+\varepsilon{\bf n})-u(x^0)}\varepsilon\,>0,
\label{eq:bpp}
\end{equation}
where $\mathbf{n}$ is the interior normal vector to the boundary at the point $x^0$.

In particular, if the function $u$ has a derivative with respect to the direction $\mathbf{n}$ at the point $x^0$ then $\partial_{\mathbf{n}}u (x^0)>0$.}
\medskip

It is important to note that the strong maximum principle is a property of the operator ${\mathbb L}$ whereas the normal derivative lemma also depends on the behavior of $\partial\Omega$ in a neighborhood of $x^0$. 
\medskip

Closely adjacent to the main subjects of the survey is the Harnack in\-equal\-ity, which can be regarded as a quantitative version of the strong maximum principle. It was first proved by C.G.A. Harnack\footnote{The mathematician Carl Gustav Axel Harnack had a twin brother Carl Gustav Adolf von Harnack, historian and theologian, founding president of the Kaiser Wilhelm Gesellschaft (now the Max Planck Society for Scientific Research). The highest award of the Max Planck Society bears his name. } in 1887 \cite[\S\,19]{H1887} for the harmonic functions on the plane. The classical formulation of this inequality is as follows: 
\medskip

\noindent {\bf The Harnack inequality.}
\textit{Let ${\mathbb L}$ be an elliptic operator in a domain $\Omega$. If $u$ is a non-negative solution of the equation ${\mathbb L}u=0$ in $\Omega$ then in any bounded subdomain $\Omega'$ such that $\overline{\Omega}{}'\subset\Omega$, we have the inequality
\begin{equation}
\sup\limits_{\Omega'}u \leq C\inf\limits_{\Omega'}u,
\label{eq:Harnack}
\end{equation}
where $C$ is a constant independent of $u$.
}

\begin{rem}
It is clear from a compactness argument that it suffices to prove (\ref{eq:Harnack}) in the case where $\Omega$ and $\Omega'$ are concentric balls. In this case, it is important for applications that the constant $C$ does not depend on the radii of the balls (but only on their ratio) or, in the worst case, remains bounded as the radii tend to zero with a fixed ratio. 
\end{rem}

Some a priori estimates of solutions, in particular, the Aleksandrov--Bakelman maximum principle, can also be considered as a quantitative ver\-sion of the strong maximum principle. On the other hand, it has become clear relatively recently that the boundary gradient a priori estimate for solutions is a state\-ment dual to the normal derivative lemma.\medskip

The discussed topic is almost boundless, so in this paper we focus on the elliptic case\footnote{Moreover, we restrict ourselves to scalar equations. In this regard, we point to a recent survey \cite{KM2020} devoted to the maximum principle for elliptic  {\bf\textit {systems}}.} only. The main part of the article is split into three sections. Section~\ref{sec:nondiv} discusses the properties of {\bf\textit {classical}} and {\bf\textit {strong}} (sub/super)so\-lu\-tions of {\bf\textit {non-divergence}} type equations, whereas Section~\ref{sec:div} concerns the properties of {\bf\textit {weak}} (sub/super)so\-lu\-tions of {\bf\textit {divergence}} type equations. Finally, Section~\ref{sec:appl} is a ``patchwork'' of various generalizations and applications. Here we do not claim to be complete, and the choice of topics reflects the personal interests of the authors. 

Various aspects of the topic under discussion are reflected in monographs and survey papers \cite{PW1984}, \cite{Sp1981}, \cite{KoLa1988E}, \cite{Br1992}, \cite{Frn2000}, \cite{PS2004}, \cite{Na2005E}, \cite{Kas2007}, \cite{PS2007}, \cite{ABMMZR2011E}. In this paper, we have used some information from these sources, as well as from our articles \cite{AN2016}, \cite{AN2019}, but we tried to cover the history of the mentioned issues as deeply as possible. 
\medskip

We are deeply grateful to Nina Nikolaevna Uraltseva, our Teacher, who introduced us to this topic. We are grateful to N.D.~Filonov, A.I.~Ibragimov, M.~Kwa\'{s}nicki, \ V.G.~Maz'ya, \ R.~Musina, \ M.V.~Safonov, \ T.N.~Shilkin, B.~Sirakov, and M.D.~Surnachev for consultations and discussions. Special thanks to G.V. Rosenblum and N.S. Ustinov for the help in the literature search. \medskip

Authors' work was supported by RFBR, project 20-11-50059. 

\subsection*{Basic notation}
\addcontentsline{toc}{subsection}{Basic notation}

\noindent
$x=(x_1,\dots,x_{n-1},x_n)=(x',x_n)$ are points in $\R^n$, $n\geq2$.

\noindent
$|x|$, $|x'|$ are the Euclidean norms in the corresponding spaces.
\smallskip

\noindent
$\R_+=[0,+\infty)$ denotes the {\bf\textit {closed}} half-axis.
\smallskip

\noindent
$\Omega$ is a domain  (i.e., a connected open set) in $\R^n$ with a boundary $\partial\Omega$.

\noindent
$\Omega$ is assumed to be bounded, unless otherwise  (as in $\S\,\ref{ss:4.2}$) specified.

\noindent
$\overline{\Omega}$ denotes the closure of $\Omega$. 

\noindent
$|\Omega|$ is the Lebesgue measure of $\Omega$.

\noindent
${\rm diam}(\Omega)$ is the diameter of $\Omega$. 

\noindent
${\rm d}(x)=\dist(x, \partial\Omega)$ is the distance from the point $x$ to $\partial\Omega$.
\smallskip

\noindent
$B_r^n(x^0)=\{x\in\R^n\,\big|\, |x-x^0|<r\}$ is the open ball in $\R^n$ with center $x^0$ and radius $r$; $B_r^n=B_r^n(0)$. 
If the dimension of the Euclidean space is clear from the context, we will simply write   $B_r(x^0)$ and $B_r$.

\noindent
$Q_{r,h}=B_r^{n-1}\times(0,h)$.
\smallskip

The indices $i$ and $j$ run from $1$ to $n$.
$D_i$ denotes the operator of (weak) differentiation with respect to $x_i$.
We adopt the standard convention regarding summation with respect to repeated indices.
\smallskip

For a function $f$ we set $f_{\pm}=\max\{\pm f,0\}$ and
$$
\Xint{ -}\limits_{\Omega}f\,dx=\frac 1{|\Omega|}\int\limits_{\Omega}f\,dx.
$$

We use the letters  $C$ and $N$ (with or without indices) to denote various positive constants. To indicate that, say,
$C$ depends on some parameters, we list them in parentheses $C(\dots)$.

\subsection*{Function spaces and classes of domains}
\addcontentsline{toc}{subsection}{Function spaces and classes of domains}

${\cal C}^k(\overline{\Omega})$ is the space of functions defined on $\overline{\Omega}$ and having continuous derivatives up to the order $k$ ($k \geq 0$). For brevity we write ${\cal C}$ instead of ${\cal C}^0$.
\smallskip

We denote by $L_p(\Omega)$, $W^k_p(\Omega)$, and ${\stackrel{\circ}W\vphantom {W}\!\!^k_p}(\Omega)$  the standard Lebesgue and Sobolev spaces, see, e.g., \cite[\S\,4.2.1]{Tb1978}; $\|\cdot\|_{p,\Omega}$ stands for the norm in $L_p(\Omega)$. Further, we write $f\in L_{p,{\rm loc}}(\Omega)$ if $f\in L_p(\Omega')$ for arbitrary subdomain $\Omega'$ such that $\overline{\Omega}{}'\subset\Omega$. In a similar way we understand $f\in W^k_{p,{\rm loc}}(\Omega)$.
\smallskip

$L_{p,q}(\Omega)$ is the Lorentz space, see, e.g., \cite[\S\,1.18.6]{Tb1978}.

\medskip

We say that $\sigma : [0,1]\rightarrow \R_+$ is a function of the ${\cal D}$ class, if
\begin{itemize}
\item $\sigma$ is continuous, increasing, and $\sigma (0)=0$;
\item $\sigma (\tau)/\tau$ is decreasing and integrable.
\end{itemize}

\begin{rem} \label{rem:A}
Notice that the monotonicity assumption for $\sigma(\tau)/\tau$ is not re\-stric\-tive. Indeed, for an increasing function $\sigma :[0,1]\rightarrow \R_+$ such that $\sigma (0)=0$ and $\sigma (\tau)/\tau$ is integrable, we define $$
\widetilde{\sigma} (t)=t \sup\limits_{\tau \in [t,1]}\frac{\sigma (\tau)}{\tau}, \qquad t\in (0,1).
$$
Obviously, $\widetilde{\sigma} (t)/t$ decreases on $[0,1]$ and $\sigma (t) \leq \widetilde{\sigma} (t)$ on $(0,1 ]$ (the latter inequality allows us to put $\widetilde{\sigma}$ instead of $\sigma$ in all estimates). Further, the set of points where $\sigma (t)<\widetilde{\sigma} (t)$ is at most a countable union of the intervals $(t_{1j},t_{2j})$. Evidently, $\widetilde{\sigma}$ is increasing on each of these intervals and therefore it is increasing on $[0,1]$. 

Now we consider the integral
$$
\int\limits_0^1 \frac{\widetilde\sigma(\tau)}{\tau}\,d\tau =\int\limits_{\{\widetilde\sigma=\sigma\}} \frac{\sigma(\tau)}{\tau}\,d\tau + \sum\limits_j\int\limits_{t_{1j}}^{t_{2j}} \frac{\widetilde\sigma(\tau)}{\tau}\,d\tau.
$$
However, on $(t_{1j},t_{2j})$ we have
$$
\frac{\widetilde\sigma(t)}{t}\equiv \frac{\sigma(t_{1j})}{t_{1j}}=\frac{\sigma(t_{2j})}{t_{2j}},
$$
whence, taking into account the monotonicity of $\sigma$ we arrive at
$$
\int\limits_0^1 \frac{\widetilde\sigma(\tau)}{\tau}\,d\tau =\int\limits_{\{\widetilde\sigma=\sigma\}} \frac{\sigma(\tau)}{\tau}\,d\tau + \sum\limits_j\big(\sigma(t_{2j})-\sigma(t_{1j})\big)<\infty.
$$
Thus, $\widetilde{\sigma} \in {\cal D}$.
\end{rem}

\begin{rem} \label{rem:B}
Without loss of generality, we can also assume that $\sigma$ is con\-tin\-u\-ous\-ly differentiable on $(0;1]$. Indeed, for any $\sigma \in {\cal D}$ we can define
\begin{equation} \label{eq:hat-sigma}
\hat{\sigma} (r):=2\int\limits_{r/2}^r \frac{\sigma (\tau)}{\tau} \,d\tau
= 2\int\limits_{1/2}^1 \frac{\sigma (r\tau)}{\tau}\,d\tau, \qquad r \in (0;1].
\end{equation}
By the monotonicity of the functions $\sigma$ and $\dfrac{\sigma (\tau)}{\tau}$, we conclude from the second equality in (\ref{eq:hat-sigma}) that $\hat{\sigma}$ also increases whereas $\dfrac{\hat{\sigma}(r)}{r}$ decreases on $(0;1]$.
Further, the first equality in (\ref{eq:hat-sigma}) easily implies that $\hat{\sigma} \in {\cal C}^1(0;1]$, and the following inequalities hold:
\begin{equation} \label{eq:double-sigma}
\sigma (r) \leq \hat{\sigma} (r) \leq 2\sigma (r/2), \qquad r \in (0;1].
\end{equation}
The second inequality in (\ref{eq:double-sigma}) provides $\hat{\sigma} \in {\cal D}$. Finally, the first inequality in (\ref{eq:double-sigma}) allows us to put $\hat{\sigma}$ instead of $\sigma$ in all estimates.
\end{rem}

We say that a function $\zeta : {\Omega} \to \R$ satisfies: 
\begin{itemize}
\item the H\"{o}lder condition with exponent $\alpha\in(0,1]$, if
$$
|\zeta(x)-\zeta (y)|\leq C|x-y|^{\alpha} \quad \text{ for all}\quad x,y \in \Omega;
$$
\item 
the Dini condition, if there is a function $\sigma\in{\cal D}$ such that
$$
|\zeta(x)-\zeta (y)|\leq \sigma (|x-y|)\quad \text{ for all}\quad x,y \in \Omega.
$$ 
\end{itemize}
Further, ${\cal C}^{k,\alpha}(\Omega)$ and ${\cal C}^{k, {\cal D}}(\Omega)$ for $k\geq0$ are the spaces of functions which have derivatives of order $k$ satisfying the H\"{o}lder condition with exponent $\alpha~\in~(0,1]$ (respectively, the Dini condition). Functions in ${\cal C}^{0,1}(\Omega)$ are called Lipschitz.
\medskip

We say that a domain $\Omega \subset \R^n$ belongs to the {\bf class} ${\cal C}^k$ or is ${\cal C}^k$-{\bf smooth} for some $k\geq0$, if there is an $r>0$ such that for every point $x^0\in\partial \Omega$ the set  $B_r(x^0)\cap \partial\Omega $ (in an appropriate Cartesian coordinate system) is the graph\footnote{The set $B_r(x^0)\cap \Omega $ lies on one side of the graph.} of a function $x_n=f(x')$, $f\in {\cal C}^k(G)$ (here $G$ is a domain in $\R^{n-1}$). In a similar way we define domains of classes ${\cal C}^{k,\alpha }$ and ${\cal C}^{k, {\cal D}}$. 

Domains of ${\cal C}^{0,1}$ class are called strongly Lipschitz.

\medskip

Recall that the {\bf interior ball condition} means that one can touch any point of the boundary $\partial\Omega$ with a ball of fixed radius lying in $\Omega$.

In a similar way, denote by ${\mathfrak T}(\phi,h)$ (here $\phi : [0,+\infty) \mapsto [0,+\infty)$ is a convex function, $\phi(0)=0$, and $h>0$) the domain (body)
$$
{\mathfrak T}(\phi,h)=\big\{x\in\R^n\,\big| \, \phi(|x'|)<x_n<h\big\}.
$$
Assume that one can touch any point of the boundary $x^0\in\partial\Omega$ with a body congruent to ${\mathfrak T}(\phi,h)$ with vertex at the point $x^0$, and this body lies in $\Omega$. Suppose also that $\phi,h$ do not depend on $x^0$. Then we say that $\Omega$ satisfies \begin{itemize}
    \item  the {\bf interior ${\cal C}^{1,\alpha}$-paraboloid condition}, $\alpha\in(0,1]$, if $\phi(s)=Cs^{1+\alpha}$ (for $\alpha=1$ this condition coincides with the interior ball condition);
    \item the {\bf interior ${\cal C}^{1,{\cal D}}$-paraboloid condition} if $\phi'(0+)=0$, and $\phi'$ satisfies the Dini condition;
    \item the {\bf interior cone condition} if $\phi(s)=Cs$.
\end{itemize}
In a similar way we define conditions of {\bf exterior ball}, {\bf exterior paraboloid} and {\bf exterior cone}.

It is easy to see that any domain of ${\cal C}^{1,1}$ class satisfies the interior and exterior ball conditions. Moreover, these conditions together are equivalent to the ${\cal C}^{1,1}$-smoothness of the domain, see, e.g., \cite[Lemma 2]{NU1985E}. In a similar way, the ${\cal C}^{1,\alpha}$-smooth domains are exactly domains satisfying the interior and exterior ${\cal C}^{1,\alpha}$-paraboloid conditions; the ${\cal C}^{1,{\cal D}}$-smooth domains are exactly domains satisfying the interior and exterior ${\cal C}^{1,{\cal D}}$-paraboloid conditions\footnote{Without a priori assumption ``the boundary is locally the graph of a function'' these equivalences were proved in \cite{ABMMZR2011E}.}; and strongly Lipschitz domains satisfy the interior and exterior cone conditions\footnote{Here, contrary to the assertion made in \cite{ABMMZR2011E}, there is no longer any equivalence. A counterexample is a Lipschitz, but not strongly Lipschitz, domain composed of two ``bricks'', see, e.g., \cite[p.39]{Mn2003}.}.

\section{Non-divergence type operators}\label{sec:nondiv}

In this Section, we consider operators with the following structure:
\begin{equation}
{\cal L} \equiv -a^{ij}(x)D_i D_j+b^i(x)D_i.
\label{eq:nondiv_operator}
\end{equation}
We introduce the notation ${\cal A}=(a^{ij})$, ${\bf b}=(b^i)$. If ${\bf b}\equiv 0$ then we write ${\cal L}_0$ instead of ${\cal L}$.

The matrix of principal coefficients ${\cal A}$ is symmetric and satisfies either the {\bf degenerate ellipticity condition}
\begin{equation}
a^{ij}(x)\xi_i\xi_j\geq0 \quad\mbox{for all}\ \  \xi\in\R^n,
\label{eq:ellipt}
\end{equation}
or the {\bf uniform ellipticity condition}
\begin{equation}
\nu|\xi|^2\leq a^{ij}(x)\xi_i\xi_j\leq\nu^{-1}|\xi|^2 \quad\mbox{for all}\ \  \xi\in\R^n
\label{eq:uniell}
\end{equation}
(here $\nu \in (0,1]$ is the so-called {\bf ellipticity constant}).

In Subsections \ref{ss:2.1}--\ref{ss:2.2} we assume that the condition (\ref{eq:ellipt}) or (\ref{eq:uniell}) is satisfied for all $x\in\Omega$. Starting from Subsection \ref{ss:2.3}, it is assumed that the entries of the matrix ${\cal A}$ are measurable functions, and the condition  (\ref{eq:ellipt}) or (\ref{eq:uniell}) holds for almost all $x\in\Omega$.

\begin{rem}
\label{rem:1}
For operators in the more general form ${\cal L}+c(x)$, both the strong maximum principle and the normal derivative lemma obviously do not hold in the formulation given in the Introduction. Indeed, the first eigenfunction of the Dirichlet problem for the Laplacian is a counterexample even for the weak maximum principle. In this case, one usually imposes a condition on the sign of the coefficient $c(x)$ in a neighborhood of the minimum point. We provide two pairs of simple assertions. 
\begin{enumerate}
    \item Assume that the strong maximum principle holds for the operator ${\cal L}$. 
    \begin{enumerate}
        \item[(a)] Let $c\geq0$, $c\not\equiv0$. If ${\cal L}u+cu\geq0$ in $\Omega$ then $u$ cannot attain its {\bf\textit {negative}} minimum in $\Omega$.
        
        \item[(b)] Let $c\leq0$, $c\not\equiv0$. If ${\cal L}u+cu\geq0$ in $\Omega$ then $u$ cannot attain its {\bf\textit{non-negative}} minimum in $\Omega$ unless $u\equiv0$.
    \end{enumerate}
    
    \item Assume that the normal derivative lemma holds for the operator ${\cal L}$ in the domain $\Omega$.
    \begin{enumerate}
        \item[(a)] 
        Let ${\cal L}u+cu\geq0$ in $\Omega$, $c\geq0$, $c\not\equiv0$. If $u$ attains its {\bf\textit {negative}} minimum at a point $x^0\in \partial\Omega$ then the inequality $\partial_{\mathbf{n}}u (x^0)>0$ holds.
        
        \item[(b)] Let ${\cal L}u+cu\geq0$ in $\Omega$, $c\leq0$, $c\not\equiv0$. If $u$ attains its {\bf\textit {non-negative}} minimum at a point $x^0\in \partial\Omega$ then the inequality $\partial_{\mathbf{n}}u (x^0)>0$ holds unless $u\equiv0$.
    \end{enumerate}
\end{enumerate}
All four assertions follow from the fact that the inequality ${\cal L}u+cu\geq0$ implies ${\cal L}u\geq0$ in some neighborhood of the minimum point. 
\end{rem}

\subsection{Classical results: from Gauss and Neumann\\ to Hopf and Oleinik}\label{ss:2.1}

Recall that the strong maximum principle for harmonic functions in a three-dimensional domain was obtained by C.F.~Gauss \cite{G1840E} on the basis of his mean value theorem\footnote{An extensive survey of mean value theorems for various classes of functions is contained in \cite{Kz2019E}, see also \cite{NV1994}.}. Since this theorem is valid for harmonic functions in $\R^n$ for any $n$, the Gauss proof is obviously valid in any dimension and, moreover, can be easily extended to superharmonic functions. 

Proof of the strong maximum principle for uniformly elliptic operators of the more general form ${\cal L} +c(x)$ with ${\cal C}^2$-smooth coefficients (in the form given in the item~1(a) of Remark \ref{rem:1}) was given: 

\begin{itemize}
\item 
    in 1892, for $c(x)>0$ in the two-dimensional case \cite{P1892};
\item 
    in 1894, for $c(x)>0$ in the multidimensional case \cite{M1894};
\item 
    in 1905, for $c(x)\geq0$ in the two-dimensional case \cite{P1905}, see also \cite{L1912}.
\end{itemize}

The most important step was taken in 1927 by E.~Hopf \cite{H1927}\footnote{A similar idea is contained in \cite{Pi1927}, but the strong maximum principle is not established in this paper.}, see also \cite{H2002}. Although in this paper the validity of the strong maximum principle is established for uniformly elliptic operators of the form (\ref{eq:nondiv_operator}) with continuous coefficients, actually the Hopf proof runs without changes for operators with {\bf\textit {bounded}} coefficients. 

Another important observation was made in \cite{H1927} for operators of the form ${\cal L}+c(x)$. In addition to the obvious assertion of item~1(a) of Remark \ref{rem:1}, Hopf showed\footnote{In 1954, A.D. Aleksandrov \cite{Al1954E} gave another (purely geometric) proof of this statement.} that if ${\cal L}u+cu\geq0$ in $\Omega$, then without any conditions on the sign of the coefficient $c(x)$, the function $u$ cannot attain a {\bf\textit{zero}} minimum in $\Omega$ unless $u\equiv0$. 
\medskip

As mentioned in the Introduction, the normal derivative lemma was first established for harmonic functions by S.~Zaremba \cite{Z1910E} under the interior ball condition on the boundary of a three-dimensional domain. The Zaremba proof uses only the weak maximum principle and the Green's function of the Dirichlet problem for the Laplacian in a ball. So it is valid in any dimension, and also runs for superharmonic functions. 

It should be noted that for the Laplace operator there is an alternative (and equivalent) formulation of the normal derivative lemma: 
\medskip

\textit{Let ${\cal G}$ be the Green's function of the Dirichlet problem for the Laplacian in $\Omega$. Then the inequality $\partial_{\mathbf{n}}{\cal G}(x,x^0)>0$ holds for $x\in\Omega$ and $x^0\in \partial\Omega$.}\medskip

This assertion was proved by C.~Neumann \cite{N1888} back in 1888 for a two-dimensional ${\cal C}^2$-smooth convex domain. Later it was generalized: 
\begin{itemize}
\item 
    in 1901, for a two-dimensional ${\cal C}^2$-smooth domain, strictly star-shaped with respect to a point \cite{K1901};
\item 
    in 1909, for a general two-dimensional ${\cal C}^2$-smooth domain \cite{L1909}; 
\item 
    in 1912, for a two-dimensional ${\cal C}^{1,\alpha}$-smooth domain, $\alpha\in(0,1)$ \cite{K1912};
\item 
    in 1918, for a three-dimensional ${\cal C}^{1,1}$-smooth\footnote{In \cite{L1918} and \cite{L1921}, the author claims the statement for a domain of class ${\cal C}^{1,\alpha}$, $\alpha\in(0,1)$. However, the proof relies on the following fact: for any point $x^0\in\partial\Omega$ one can choose a point $x\in\Omega$ so that $x^0$ is the boundary point closest to $x$. This fact is not true for domains of class ${\cal C}^{1,\alpha}$ with $\alpha<1$.} domain \cite{L1918}, see also \cite{L1921}.
    \end{itemize}

For the operator $-\Delta +b^i(x)D_i+c(x)$ with $c(x)\geq0$ in a two-dimensional ${\cal C}^{2,\alpha}$-smooth domain, $\alpha\in(0,1)$, this statement was established in 1924 \cite{L1924}. Later, however, almost all the results we know were formulated in the form of the conventional normal derivative lemma\footnote{Probably, this is due to the fact that for operators with variable principal coefficients, the proof of the alternative formulation is essentially more difficult, and in the general case of measurable principal coefficients, the Green's function is not defined.}. 
\medskip

In 1931, it was first noted \cite{B1931} (for the operator $-\Delta+c(x)$ with $c(x)\geq0$ in a two-dimensional ${\cal C}^{2}$-smooth domain) that the normal derivative lemma is actually true for a derivative along any strictly interior direction $\boldsymbol{\ell}$ (i.e., along a direction that forms an acute angle with the interior normal). 

In 1932, G.~Giraud \cite[Ch. V]{G1932} proved the normal derivative lemma\footnote{Instead of the normal $\bf n$, Giraud uses the conormal ${\bf n}^{\cal L}$ with coordinates ${\bf n}^{\cal L} _i=a^{ij}{\bf n}_j$. This gives an equivalent statement. It is essential that he also considers the case $u(x^0)=0$ without conditions on the sign of $c(x)$; cf. item~2(a) of Remark~\ref{rem:1}.} for uniformly elliptic operators ${\cal L}+c(x)$, $c(x)\geq0$, with coefficients of class ${\cal C}^{0,\alpha}$, $\alpha\in(0,1)$, in an $n$-dimensional domain of class ${\cal C}^{1,1}$. In \cite{G1933} this result was extended to the case where the lower-order coefficients can have singularities on a set $\mathfrak M$, that is the union of a finite number of ${\cal C}^{1,\alpha}$-smooth, codimension~$1$ manifolds, and 
$$
|b^i(x)|, |c(x)|\leq C\cdot \dist^{\gamma-1}(x,\mathfrak M),\qquad \gamma\in(0,1).
$$

In 1937, for the first time, the condition on the boundary was significantly weakened: in the paper \cite{KL1937E}, the normal derivative lemma was proved for the Laplacian in a (three-dimensional) domain satisfying the interior ${\cal C}^{1,\alpha}$-paraboloid  condition\footnote{In some sources (for example, \cite{KH1972E} and \cite{ABMMZR2011E}), it is stated that a similar condition on the domain was considered already by Giraud. Indeed, in \cite{G1932} and \cite{G1933}, some theorems are proved for domains of class ${\cal C}^{1,\alpha}$, but the normal derivative lemma requires $\alpha=1$.}. 

Finally, a key step was taken by E.~Hopf \cite{H1952} and O.A.~Oleinik \cite{O1952E}, who simultaneously and independently proved the normal derivative lemma for uniformly elliptic operators with continuous coefficients under the interior ball condition on the boundary of a domain. The proofs in \cite{O1952E} and \cite{H1952} are based on the same idea and, like in \cite{H1927}, run without changes for operators with bounded coefficients\footnote{Hopf considers operators of the form (\ref{eq:nondiv_operator}), Oleinik deals with operators ${\cal L}+c(x)$ under the condition $c(x)\geq0$, $u(x^0)\leq0$. Moreover, in \cite{O1952E}, instead of the normal, she takes an arbitrary direction that forms an acute angle with $\bf n$.}. 
\medskip

Now we give a complete proof of the classical results \cite{H1927} and \cite{H1952}, \cite{O1952E}.

\begin{thm} \label{thm:HO}
{\bf A}. Let ${\cal L}$ be an operator of the form (\ref{eq:nondiv_operator}), let the functions $a^{ij}$, $b^i$ and $c$ be bounded in $\Omega$, and let the assumption (\ref{eq:uniell}) hold. Suppose that $u\in{\cal C}^2(\Omega)$, and ${\cal L}u+cu\geq0$ in $\Omega$. Then

\begin{itemize}
\item[A1.]\ The function $u$ cannot attain its zero minimum in $\Omega$ unless $u\equiv0$.
\item[A2.]\ 
If $c\geq0$ then $u$ can attain no negative minimum in $\Omega$ unless $u\equiv const$ and $c\equiv0$.
\item[A3.]\ 
If $c\leq0$ then $u$ can attain no positive minimum in $\Omega$ unless $u\equiv const$ and $c\equiv0$.
\end{itemize}

{\bf B}. In addition, let the domain $\Omega$ satisfy the interior ball condition, and let the function $u\not\equiv const$ be continuous in $\overline{\Omega}$. Denote by $x^0$ the point $\partial\Omega$ at which $u$ attains its minimum. Then the inequality (\ref{eq:bpp}) holds true under any of the following conditions\footnote{The assertion B1, without the sign condition for $c(x)$, was apparently first highlighted in \cite{GNN1979}, see also \cite{S1971}.}: 
\begin{itemize}
\item[B1.]\ 
$u(x^0)=0$;
\item[B2.]\ $u(x^0)<0$ and $c\geq0$;
\item[B3.]\ $u(x^0)>0$ and $c\leq0$.

\end{itemize}

\noindent Moreover, in (\ref{eq:bpp}) one can replace the normal $\bf n$ by any strictly interior direction $\boldsymbol{\ell}$. 
\end{thm}
 
\begin{proof}
1. We begin with the case $c\equiv0$.
First of all, we establish the weak maximum principle for the operator ${\cal L}$ in a domain $\pi$ with sufficiently small diameter $d$. 

Assume the contrary: let ${\cal L}u\geq0$ in $\pi$, and let $u|_{\partial \pi}\geq0$, but for some $x^0\in \pi$ we have $u(x^0)=-A<0$. Consider the function
$$
u^\varepsilon(x)=u(x)-\varepsilon |x-x^0|^2.
$$
Obviously, for all sufficiently small $\varepsilon$ we have
$$
u^\varepsilon|_{\partial \pi}\geq-\varepsilon d^2>-A=u^\varepsilon(x^0). 
$$
Therefore $u^\varepsilon$ attains its minimum at some point $x^1\in \pi$. At this point we have $Du^\varepsilon(x^1)=0$, and the matrix $D^2u^\varepsilon(x^1)$ is non-negative definite. Thus ${\cal L}u^\varepsilon(x^ 1)\leq0$. 

However, the assumption ${\cal L}u\geq0$ implies
$$
{\cal L}u^\varepsilon\ge 2\varepsilon\big(a^{ij}\delta_{ij}-b^i(x_i-x^0_i)\big)\geq 2\varepsilon(n\nu-d\sup|{\bf b}(x)|)>0 \quad \text{in}\quad \pi, 
$$
provided $d< d_0:=\frac {n\nu}{\sup|{\bf b}(x)|}$. This contradiction proves the statement.\medskip

2. Now we prove the strong maximum principle for the operator ${\cal L}$. Assume the contrary: let ${\cal L}u\geq0$ in $\Omega$, and let $u\not\equiv const$, but let the set
\begin{equation}
M=\{x\in\Omega\,\big| \, u(x)=\inf\limits_{\Omega} u\} 
\label{eq:min-set}
\end{equation}
be not empty. The complement $\Omega\setminus M$ is open, and therefore there is a ball lying in it whose boundary contains a point in $M$. We place the origin at the center of this ball and denote by $r$ the radius of the ball, and by $x^0$ a point in $\partial B_r\cap M$. Without loss of generality, we can assume that $r<\frac{d_0}2$. 

In the annulus $\pi=B_r\setminus \overline{B}_{ \frac r2}$, we consider the {\bf barrier function}\footnote{Hopf and Oleinik used different barrier functions. Apparently the function (\ref{eq:barrier}) was first introduced for this purpose in \cite{La1968E}, see also \cite{La1968bE} and \cite[Ch. 1]{La1971E}.} 
\begin{equation}
v_s(x)=|x|^{-s}-r^{-s}.
\label{eq:barrier}
\end{equation}
We estimate ${\cal L}v_s$ taking into account the ellipticity condition (\ref{eq:uniell}):
$$
D_iv_s(x)=-sx_i|x|^{-s-2};\qquad D_iD_jv_s(x)=s(s+2)x_ix_j|x|^{-s-4}-s\delta_{ij}|x|^{-s-2};
$$
$$
\aligned
{\cal L}v_s(x)= &\, |x|^{-s-2}\cdot\Big(-s(s+2)a^{ij}\frac{x_i}{|x|}\frac{x_j}{|x|}+sa^{ij}\delta_{ij}-sb^ix_i\Big)\\
\leq &\, s|x|^{-s-2}\cdot\big[-(s+2)\nu +n\nu^{-1}+r\sup\limits_{\Omega}|{\bf b}(x)|\big].
\endaligned
$$
We choose $s$ so large that the expression in the square brackets is negative. Then for any $\varepsilon>0$ the function $w^\varepsilon=u-\inf\limits_{\Omega} u-\varepsilon v_s$ satisfies the inequality
${\cal L}w^\varepsilon\geq0$ in $\pi$.

Further, $\partial\pi=\partial B_r\cup\partial B_{\frac r2}$. Obviously, $w^\varepsilon|_{\partial B_r}\geq0$. Since $B_r\subset\Omega\setminus M$ by construction, the function $u-\inf\limits_{\Omega}u$ is bounded away from zero on the set $\partial B_{\frac r2}$, and for sufficiently small $\varepsilon>0$ we have $w^\varepsilon|_{\partial B_{\frac r2}}\geq0$. Therefore, we can apply the weak maximum principle to $w^\varepsilon$ in $\pi$, that gives $w^\varepsilon\geq0$ в $\pi$.

However, $w^\varepsilon(x^0)=0$. Therefore, for any vector $\boldsymbol{\ell}$ directed into $\pi$ we have $\partial_{\boldsymbol{\ell}}w^\varepsilon(x^0)\geq0$, that is 
$$
\partial_{\boldsymbol{\ell}}u(x^0)\geq\varepsilon \partial_{\boldsymbol{\ell}}v_s(x^0)>0.
$$
This gives a contradiction, since at the minimum point $Du(x^0)=0$, and the statement follows.
\medskip

3. Next, we prove the normal derivative lemma for ${\cal L}$. By assumption, one can choose a ball of radius $r$ touching $\partial\Omega$ at the point $x^0$. We place the origin at the center of this ball. According to the strong maximum principle, $u>u(x^0)$ in $B_r$. Further, verbatim repetition of part 2 of the proof gives the inequality (\ref{eq:bpp}), where $\bf n$ can be replaced by $\boldsymbol{\ell}$.\medskip

4. Finally, we drop the assumption $c\equiv0$. Statements {\it A2, A3, B2} and {\it B3} follow immediately from Remark \ref{rem:1}.

To prove {\it A1} и {\it B1}, we represent $u$ in the form $u=\psi v$, with $\psi>0$ and $v\geq0$ in $\overline{\Omega}$. Direct computation gives
\begin{equation}
\label{eq:tilde-L}
0\leq \frac {{\cal L}u+cu}{\psi} =\widetilde{\cal L}v:=-a^{ij}D_iD_jv+\widetilde{b}^iD_iv+\widetilde{c}v,
\end{equation}
where 
$$
\widetilde{b}^i=b^i-\frac{2a^{ij}D_j\psi}{\psi}, \qquad \widetilde{c}=\frac{{\cal L}\psi+c\psi}{\psi}.
$$
Now we put $\psi(x)=\exp(\lambda x_1)$. Then
$$
{\cal L}\psi+c\psi=\psi (-a^{11}\lambda^2+b^1\lambda+c)\leq\psi \big[-\nu\lambda^2 +\sup\limits_{\Omega}b^1(x)\lambda+\sup\limits_{\Omega}c(x)\big].
$$
We choose $\lambda$ so large that the expression in the square brackets is negative. Then for the operator $\widetilde{\cal L}$ defined in (\ref{eq:tilde-L}), the statements of items~1(b) and 2(b) in Remark \ref{rem:1} hold true. In particular, $v$ cannot vanish inside the domain, which gives {\it A1}. Since $u(x^0)=0$ implies $Du(x^0)=\psi(x^0)Dv(x^0)$, item~2(b) for $v$ provides {\it B1} for $u$.
\end{proof}

\subsection{Generalizing of classical results and refining \\the conditions on $\partial\Omega$}\label{ss:2.2}

After the basic results of \cite{H1952}, \cite{O1952E}, through the efforts of many authors, the topic was developed in several directions:
\begin{enumerate}
     \item extension of the class of differential operators, that is, weakening the requirements for the principal and lower-order coefficients;
     \item extension of the class of domains, that is, reduction of requirements on the boundary (for the normal derivative lemma);
     \item refinement of the applicability limits for the corresponding statements by constructing various counterexamples.
\end{enumerate} 

We begin the description of the results with the article of C.~Pucci \cite{Pu1957}--\cite{Pu1958}, in which the normal derivative lemma was established in the domain $\Omega=B_r$ for a wider class of operators than in \cite{H1952}, \cite{O1952E}. Namely, the ellipticity condition is allowed to degenerate in the directions tangent to $\partial\Omega$, and the lower-order coefficients satisfy the conditions 
\begin{equation}
\label{eq:Pucci}
|b^i(x)|\leq \frac {\sigma({\rm d}(x))} {{\rm d}(x)}, \qquad 0\leq c(x)\leq \frac {\sigma({\rm d}(x))} {{\rm d}^2(x)}, \qquad \sigma\in{\cal D}.
\end{equation}
The Pucci proof is based on the barrier function
$$
\mathfrak{v}(x)=\!\!\int\limits_0^{{\rm d}(x)}\int\limits_0^\tau \frac {\sigma(t)}t\,dtd\tau\  + \ \kappa{\rm d}(x)
$$
with an appropriate choice of the constant $\kappa$. This function and its variations were used later in many papers. 
\medskip

If the ellipticity condition degenerates even more, then the strong max\-i\-mum principle in its classical form does not hold. A.D.~Aleksandrov in a series of papers \cite{Al1958E}, \cite{Al1959E}, \cite{Al1959aE}, \cite{Al1960E}, \cite{Al1960aE} gave for such operators a description of the zero set structure of a non-negative function $u$, satisfying the inequality ${\cal L}u+cu\geq0$ in $\Omega$.\footnote{This problem is also discussed in \cite[Ch. III]{ORa1971E} and in \cite{Ra2008E}--\cite{Ra2008aE}; some operators with unbounded coefficients are considered in \cite{PT2009}; see also \cite{Fh2020}.} 

In the papers \cite{V1963}--\cite{V1964} by R.~V\'{y}born\'{y} the normal derivative lemma was proved for the operator ${\cal L}+c(x)$ in a ${\cal C}^{1,{\cal D}}$-smooth domain\footnote{More precisely, V\'{y}born\'{y} assumes that there exists a function $\rho\in{\cal C}^2(\Omega)\cap{\cal C}^1(\overline{\Omega})$ such that
$\rho(x)=0$ and $D\rho\ne0$ on $\partial\Omega$, $\rho>0$ and $|D^2\rho(x)|\leq \frac {\sigma (\rho(x))}{\rho(x)}$ in $\Omega$, where $\sigma\in{\cal D}$. The (local) existence of such a function for a domain of ${\cal C}^{1,{\cal D}}$ class was proved in \cite{Li1985}. }. The conditions imposed on the coefficients of the operator were the same\footnote{V\'{y}born\'{y} proves the assertion of item~2(a) of Remark \ref{rem:1}, in this case the upper bound for the coefficient $c(x)$ in (\ref{eq:Pucci}) is redundant.} as in \cite{Pu1957}. 

Unfortunately, the results of \cite{V1963}--\cite{V1964} are not widely known.\medskip 

In \cite{W1967}, sharp estimates were obtained for the derivatives of the Green's function of the Dirichlet Laplacian in a ${\cal C}^{1,{\cal D}}$-smooth domain\footnote{Under more restrictive conditions on the domain, some of these estimates were established earlier in \cite{E1956E} and \cite{S1958E}.}. In particular, the normal derivative lemma was proved in the Neumann form (the normal derivative of the Green's function on $\partial\Omega$ is positive). Also there was given a counterexample showing that the condition ${\cal C}^{1,{\cal D}}$ on the boundary cannot be relaxed to ${\cal C}^{1}$. Namely, if $\phi'$ does not satisfy the Dini condition at zero, then the relation $\partial_{\mathbf{n}}{\cal G}(x ,0)=0$ holds in the paraboloid ${\mathfrak T}(\phi,h)$.

The note \cite{VM1967E} was published at the same time as \cite{W1967}. Subtle asymptotics of harmonic functions in a neighborhood of non-smooth boundary points were derived in this note. As a corollary, the following statement was proved. Let a function $u$ be harmonic in the paraboloid ${\mathfrak T}(\phi,h)$, and let $u$ attain its minimum at the vertex $x^0=0$. Then the necessary and sufficient condition for the relation $\partial_{\boldsymbol{\ell}}u(0)>0$ (for any strictly interior direction $\boldsymbol{\ell}$) is the Dini condition for the function $\phi'$ at zero. This statement is equivalent to that obtained in \cite{W1967}.
\medskip 

The behavior of solutions to the equation ${\cal L}u=0$ in a neighborhood of the point $x^0\in\partial\Omega$ under the assumption $b^i(x)=o(|x-x^0|^{-1})$ was studied in  \cite{Od1967a}--\cite{Od1967} and \cite{Mi1967} in the cases where $\partial\Omega$ satisfies, respectively, the interior/exterior cone condition at the point $x^0$.
\medskip

A large series of papers generalizing the normal derivative lemma is due to B.N. Himchenko and L.I. Kamynin.

In the article \cite{Hi1969E} (see also \cite{Hi1970aE}), the normal derivative lemma for the Laplacian was established for domains satisfying the interior ${\cal C}^{1,{\cal D}}$- paraboloid condition. Further, in this paper (see also \cite{Hi1970bE}), the normal derivative es\-ti\-mate on $\partial\Omega$ was obtained for solutions of the problem 
$$
-\Delta u=f \quad \mbox{in} \quad\Omega,\qquad u\big|_{\partial\Omega}=0,
$$
provided that $\Omega$ satisfies the exterior ${\cal C}^{1,{\cal D}}$-paraboloid condition\footnote{Here (apparently, for the first time in the literature) one can see the duality of the gradient estimate for the solution on $\partial\Omega$ and the normal derivative lemma.}, while the right-hand side is subject to $|f(x)|\leq C {\rm d}^{\gamma-1}(x)$, $\gamma\in(0,1 )$. Finally, \cite{Hi1969E} gives examples showing that the conditions on the boundary cannot be noticeably improved (these examples repeat in essence the corresponding counterexamples from \cite{W1967} and \cite{VM1967E}). 

In the paper \cite{Hi1970E}, the results of \cite{Hi1969E} were extended to uniformly elliptic operators of the form ${\cal L}+c(x)$ with bounded coefficients $b^i(x)$. The normal derivative lemma is stated there (for any strictly interior direction) ``under the assumption that the maximum principle holds'' (apparently this means $c(x)\geq0$), and the gradient estimate on $\partial \Omega$ for the solution of the problem 
$$
{\cal L}u+cu=f \quad \mbox{in} \quad\Omega,\qquad u\big|_{\partial\Omega}=g
$$
is stated under the assumptions
$$
|c(x)|, |f(x)|\leq C {\rm d}^{\gamma-1}(x), \quad \gamma\in(0,1);\qquad g\in {\cal C}^{1,{\cal D}}(\partial\Omega).
$$

In the article \cite{KH1972E} (see also \cite{KH1971E}), the normal derivative lemma is ex\-tend\-ed to elliptic-parabolic operators 
$$
-a^{ij}(x,y)D_{x_i} D_{x_j}-\tilde a^{kl}(x,y)D_{y_k} D_{y_l}+b^i(x,y)D_{x_i}+\tilde b^k(x,y)D_{y_k}+c(x,y),
$$
with bounded coefficients under the following conditions: the matrix ${\cal A}$ sat\-is\-fies the uniform ellipticity condition, the matrix $\widetilde{\cal A}$ is non-negative definite, $c(x)\geq0$, and the domain $\Omega$ satisfies the interior ${\cal C}^{1,{\cal D}}$-paraboloid condition, and the paraboloid axis is not perpendicular to the plane $y=0$. 

the In \cite{KH1977E} and \cite{KH1978aE} (see also \cite{KH1975E}), the results of \cite{Hi1970E} are generalized to the class of weakly degenerate operators whose principal coefficients satisfy conditions similar to \cite{Pu1957}, \cite{V1963} (lower-order coefficients are bounded)\footnote{Further development of this topic can be found, for example, in \cite{CCK2013}.}.

Finally, in the series of papers \cite{KH1978E}--\nocite{KH1979E}\nocite{KH1979aE}\nocite{KH1979bE}\nocite{KH1979cE}\nocite{KH1980E}\cite{KH1980aE}
some subtle generalizations of the results from \cite{Al1958E}--\cite{Al1960aE} are given.
\medskip

A very interesting ``weakened'' form of the normal derivative lemma was established by N.S.~Nadirashvili \cite{Nd1983E} (see also \cite{Nd1981E}) in a domain $\Omega$ satisfying the interior cone condition. Namely, let ${\cal L}$ be a uniformly elliptic operator of the form (\ref{eq:nondiv_operator}), and $c(x)\geq0$.
Suppose that a non-constant function $u$ satisfying the condition ${\cal L}u+cu\geq0$ attains its non-positive minimum at the point $x^0\in \partial\Omega$. Then {\bf\textit{in any neighborhood of $x^0$ there is a point}} $x^*\in \partial\Omega$ such that for any strictly interior direction $\boldsymbol{\ell}$ the inequality 
$$
\liminf\limits_{\varepsilon\to+0}\,\frac  {u(x^*+\varepsilon\boldsymbol{\ell})-u(x^*)}\varepsilon\,>0.
$$
holds true. In \cite{Ka1985E} this result was generalized to a certain class of domains with outer ``peaks'' and to weakly degenerate (in the spirit of \cite{KH1977E}) non-divergence type operators. 
\medskip

In the article \cite{Li1985} by G.~Lieberman, the important notion of {\bf regularized distance}\footnote{In particular cases, this construction has been used earlier, see, e.g., \cite{Ne1962}, \cite{V1963}, \cite{V1964}.} was introduced. In particular, it was shown that in any domain $\Omega$ of class ${\cal C}^1$ there exists a function $\rho\in{\cal C}^2(\R^n\setminus\partial\Omega) \cap {\cal C}^1(\R^n)$ for which the estimates 
$$
C^{-1}{\rm d}(x)\leq \pm\rho(x)\leq C\,{\rm d}(x);
$$
$$
|D\rho(x)-D\rho(y)|\leq C\sigma(|x-y|);
$$
$$
|D^2\rho(x)|\leq C\,\frac {\sigma(|\rho(x)|)}{|\rho(x)|}
$$
hold true (the $+$ and $-$ signs are related to the points $x\in\overline{\Omega}$ and $x\in\mathbb R^n\setminus\Omega$, respectively). Here $\sigma$ stands for the common modulus of continuity for the derivatives of the functions describing $\partial\Omega$ in local coordinates.

As a corollary, the normal derivative lemma in a domain of class ${\cal C}^{1,{\cal D}}$ was obtained in \cite{Li1985} under conditions on the coefficients (both principal and lower-order) close to \cite{Pu1957}, \cite{V1963}. Later, in \cite{Li1986}, the gradient estimates on $\partial\Omega$ for solutions to the Dirichlet problem were established in a domain of class ${\cal C}^{1,{\cal D}}$ with boundary data $g \in {\cal C}^{1,{\cal D}}(\partial\Omega)$. Also the boundary smoothness of the solution was analyzed in \cite{Li1986} in the case where $Dg\in{\cal C}(\partial\Omega)$ does not satisfy the Dini condition. 

Finally, we mention the monumental work \cite{ABMMZR2011E}. In this paper, the as\-sump\-tions on the coefficients providing the validity of the normal derivative lemma and the strong maximum principle are somewhat weakened compared to the works listed earlier, although it is much more difficult to verify these conditions. Also in \cite{ABMMZR2011E} some new counterexamples are given, showing the sharpness of the introduced conditions.

\subsection{The Aleksandrov--Bakelman maximum principle}\label{ss:2.3}

This subsection is devoted to one of the most beautiful geometric ideas in the theory of PDEs, the maximum principle of A.D. Aleksandrov and I.Ya. Bakelman. This name is given to a priori maximum estimates for solutions of non-divergence type equations. These estimates have a huge number of applications and, in particular, play a key role in proving the strong maximum principle and the normal derivative lemma for equations with unbounded lower-order coefficients in Lebesgue spaces.

The first estimates of this type were published in the papers\footnote{The history of this result is complicated. The article \cite{B1961E} was published later than the short communication \cite{Al1960bE} but was submitted somewhat earlier.  In \cite[\S\,28.1]{B1994}, it is written: ``The first version of these maximum principles was obtained by Bakelman \cite{B1959E}, \cite{B-disserE} in 1959''. In fact, these papers do not yet contain the estimates under consideration, although the idea of studying normal images for estimating solutions was developed earlier by Aleksandrov in \cite{Al1958aE} as well as by Bakelman in \cite{B1958E}--\nocite{B1959E}\cite{B-disserE}. On the other hand, the survey \cite{Na2005E} does not describe the importance of \cite{B1961E} correctly.} \cite{Al1960bE} and \cite{B1961E}. An estimate for solutions of the Dirichlet problem in the general case was ob\-tained in \cite{Al1963E}. In this work the sharpness of the obtained es\-ti\-mates was proved as well\footnote{The results of \cite{Al1963E} were later rediscovered in \cite{Pu1966a}, \cite{Pu1966}. In this regard, the name ``Aleksandrov--Bakelman--Pucci (ABP) maximum principle'' often occurs in literature.}. In 1963 Aleksandrov gave in Italy a series of lectures about his method. These lectures were published in Rome \cite{Al1965}. 
\medskip

To prove the Aleksandrov--Bakelman estimate, we need some de\-fi\-ni\-tions. 

Let a function $u$ be continuous in $\Omega$, and let $u\bigr|_{\partial \Omega}<0$. We denote by $\widetilde\Omega={\rm conv}(\Omega)$ the convex hull of $\Omega$. In what follows, we assume that the function $u_+$ is extended by zero to $\widetilde\Omega\setminus\Omega$.
\medskip

{\bf The convex hull} of $u_+$ is the minimal upward convex function that majorizes $u_+$ in $\widetilde\Omega$. We denote this function by $z$. It is obvious that $z\bigr|_{\partial\widetilde\Omega}=0$, and the subgraph of the function $z$ is a convex set (the convex hull of the subgraph of $u_+$). It can also be shown (see \cite{NU1985E}) that if $\Omega$ is a ${\cal C}^{1,1}$-smooth domain and $u\in {\cal C}^{1, 1}(\Omega)$, then\footnote{Note that this statement is false if the condition $u\bigr|_{\partial \Omega}<0$ is relaxed to $u\bigr|_{\partial \Omega}\leq0$.}
$z\in {\cal C}^{1,1}(\widetilde\Omega)$. We also introduce the so-called {\bf contact set} 
$$
{\cal Z}=\{x\in\Omega\,\big| \, z(x)=u(x)\}.
$$

Now we define the (in general, multi-valued) {\bf normal mapping (hodograph mapping)} $\Phi:\widetilde\Omega\to\R^n$ generated by the
function $z$. This mapping assigns to any point $x^0\in\widetilde\Omega$ all possible vectors $p\in\R^n$ such that the graph of the function $\pi(x)=p\cdot(x-x^0)+z (x^0)$ is the supporting plane to the subgraph of the function $z$ at the point $x^0$. Obviously, if $z\in C^1(\widetilde\Omega)$, then the mapping $\Phi$ is single-valued in $\widetilde\Omega$ (but not in its closure!) and is given by the formula $\Phi(x) =Dz(x)$. 
\medskip

First, we consider an operator ${\cal L}_0$ with measurable coefficients.

\begin{lem} \label{lem:pmAB-1} Let $\Omega$ be a ${\cal C}^{1,1}$ class domain, let $u\in {\cal C}^{1,1}(\Omega)$ and $u\bigr|_{\partial \Omega}<0$. Suppose that the uniform ellipticity condition (\ref{eq:uniell}) is satisfied. Then the inequality \begin{equation} \label{eq:1.1}
\int\limits_{\Phi(\widetilde\Omega)}\mathfrak{g}(p)\,dp\leq \frac 1{n^n}\int\limits_
{\cal Z}\mathfrak{g}(Du)\cdot\frac {({\cal L}_0u)^n}{\det({\cal A})}\,dx
\end{equation}
holds true for any non-negative function $\mathfrak{g}$.
\end{lem}

\begin{proof} Note that under the assumptions of Lemma the mapping $\Phi$ satisfies the Lipschitz condition. Changing variables in the integral we obtain
\begin{equation} \label{eq:1.2}
\int\limits_{\Phi(\widetilde\Omega)}\mathfrak{g}(p)\,dp=\int\limits_
{\widetilde\Omega}\mathfrak{g}(Dz)|\det(D^2z)|\,dx=\int\limits_{\widetilde\Omega}\mathfrak{g}(Dz)
\det(-D^2z)\,dx
\end{equation}

\noindent (the latter equality follows from the fact that $(-D^2z)$ is a non-negative definite matrix).\smallskip

If $x\notin {\cal Z}$ then we apply the Caratheodory theorem, see, e.g., \cite[\S\,17]{R1970}, to claim that $(x,z(x))$ is an interior point of a simplex\footnote{In this case the dimension of the simplex can be any number from $1$ to $n$.} that completely belongs to the graph of $z$. Therefore, the second derivative of $z$ in some direction vanishes. But since $D^2z(x)$ is non-positive definite matrix, this direction is a principal one, and thus $\det(-D^2z(x))=0$. 

Otherwise, if $x\in {\cal Z}$ then the tangency condition at the point $x$ gives
$$
Dz(x)=Du(x);\qquad -D^2z(x)\leq-D^2u(x)
$$
(the second relation is understood in the sense of quadratic forms and holds for almost all $x$). Therefore, (\ref{eq:1.2}) implies 
$$
\int\limits_{\Phi(\widetilde\Omega)}\mathfrak{g}(p)\,dp\leq\int\limits_{\cal Z}\mathfrak{g}(Du)\det(-D^2u)\,dx.
$$

Further, since ${\cal A}$ and $-D^2u$ are non-negative definite matrices on the set $\cal Z$, the matrix $-{\cal A}\cdot\!D^2u$ has non-negative eigenvalues. By the  inequality of arithmetic and geometric means, we have (here and below, ${\bf Tr}$ is the matrix trace) 
$$
\det(-D^2u)=\frac {\det(-{\cal A}\cdot\!D^2u)}{\det({\cal A})}\leq \frac 1{n^n}\cdot \frac {({\bf Tr}(-{\cal A}\cdot\!D^2u))^n}{\det({\cal A})}=\frac 1{n^n}\cdot \frac {({\cal L}_0u)^n}{\det({\cal A})},
$$
and (\ref{eq:1.1}) follows.
\end{proof}

\begin{rem} \label{rem:pmAB-1} Since the inequalities $u>0$ и ${\cal L}_0u\geq0$ hold on the set ${\cal Z}$, one often uses the more convenient estimate
\begin{equation} \label{eq:1.3}
\int\limits_{\Phi(\widetilde\Omega)}\mathfrak{g}(p)\,dp\leq \frac 1{n^n}\int\limits_
{\{u>0\}}\mathfrak{g}(Du)\cdot\frac {({\cal L}_0u)_+^n}{\det({\cal A})}\,dx
\end{equation}
instead of (\ref{eq:1.1}).
\end{rem}

\begin{thm} \label{thm:pmAB-1} Let the condition (\ref{eq:ellipt}) hold, and let ${\bf Tr}({\cal A})>0$ almost everywhere in $\Omega$. Then any function $u\in W^2_{n,{\rm loc}}(\Omega)$ such that\footnote{This means that for any $\varepsilon>0$ the inequality $u-\varepsilon<0$ holds in some neighborhood of $\partial\Omega$. 
\label{foot:28}} $u\bigr|_{\partial\Omega}\leq0$ satisfies the estimate
\begin{equation} \label{eq:1.4}
(\max\limits_{\overline\Omega}u_+)^n\leq \frac {\mbox {\rm diam}^n(\Omega)}
{n^n|B_1|}\int\limits_{\cal Z}\frac {({\cal L}_0u)^n}{\det({\cal A})}\,dx
\end{equation}
(here and below, we set $0/0=0$ if such uncertainty arises).
\end{thm}

\begin{proof}  Let us first assume that the matrix ${\cal A}$, the function $u$, and the domain $\Omega$ satisfy the conditions of Lemma \ref{lem:pmAB-1}. It suffices to consider the case when $M=\max\limits_{\overline\Omega}u=\max\limits_{\widetilde\Omega}z>0$.

We set $d=\mbox {\rm diam}(\Omega)=\mbox {\rm diam}(\widetilde\Omega)$ and claim that the set $\Phi(\widetilde\Omega)$ contains the ball $B_{M/d}$. Indeed, let $p\in B_{M/d}$. Consider the graph of the function $\pi(x)=p\cdot x+h$. By choosing an appropriate $h$, we can ensure that this is a supporting plane to the subgraph of the function $z$ at some point $x^0$, and write $\pi(x)=p\cdot(x-x^0)+z(x^0)$. 

If $x^0\in\partial\widetilde\Omega$ then $z(x^0)=0$, and the maximum point of $z$ satisfies
$$
M=z(x)\leq p\cdot(x-x^0)\leq |p|\cdot d<M,
$$
which is a contradiction. Therefore, $x^0\in\widetilde\Omega$, whence 
$$
p=Dz(x^0)=\Phi(x^0)\in\Phi(\widetilde\Omega),
$$
and the claim follows.
 
We use the estimate (\ref{eq:1.1}) with $\mathfrak{g}\equiv1$ and obtain
$$
|B_1|\cdot \left(\frac Md\right)^n=|B_{M/d}|\leq |\Phi(\widetilde\Omega)|\leq \frac 1{n^n}\int\limits_{\cal Z}\frac {({\cal L}_0u)^n}{\det({\cal A})}\,dx,
$$
which immediately implies (\ref{eq:1.4}).\medskip

Consider now the general case. The integrand in (\ref{eq:1.4}) does not change if the matrix ${\cal A}$ is multiplied by a positive function. Therefore, without loss of generality, we can assume that ${\bf Tr}({\cal A})\equiv1$. Let us take the function $u^{\varepsilon}=u-\varepsilon$ and approximate $\Omega$ from the inside by domains with smooth boundaries. Further, since the estimate (\ref{eq:1.4}) keeps under a passage to the limit in $W^2_n$, we can assume that $u^{\varepsilon}$ is a smooth function. We apply the estimate (\ref{eq:1.4}) to the function $u^{\varepsilon}$ and the uniformly elliptic operator ${\cal L}_0-\nu\Delta$. Then we can push $\nu\to0$ and then $\varepsilon\to0$. 
\end{proof}

\begin{thm} \label{thm:pmAB-3} Let ${\cal L}$ be an operator of the form (\ref{eq:nondiv_operator}), let the assumption (\ref{eq:ellipt}) hold, and let ${\bf Tr}({\cal A})>0$ almost everywhere in $\Omega$. Assume that
\begin{equation} \label{eq:frak-h}
\mathfrak{h}\equiv \frac {|{\bf b}|}{\det^{\frac 1n}({\cal A})}\in L_n(\Omega).
\end{equation}
Then the estimate
\begin{equation} \label{eq:1.5}
\max\limits_{\overline\Omega}u_+\leq N\left(n, \|\mathfrak{h}\|_{n,\{u>0\}}\right)\cdot\mbox {\rm diam}(\Omega)\left\|\frac {({\cal L}u)_+}{\det^{\frac 1n}({\cal A})}\right\|_{n,\{u>0\}}
\end{equation}
holds true for any function $u$ satisfying the assumptions of Theorem~\ref{thm:pmAB-1}.
\end{thm}

\begin{proof}
We can assume that the matrix ${\cal A}$, the function $u$, and the domain $\Omega$ satisfy the conditions of Lemma~\ref{lem:pmAB-1}. The general case is obtained from this particular one analogously to the second part of the proof of Theorem~\ref{thm:pmAB-1}. 
 \medskip

Let $\mathfrak{g}=\mathfrak{g}(|p|)$. Taking into account the inclusion $B_{M/d}\subset\Phi(\widetilde\Omega)$, we obtain from (\ref{eq:1.3})
\begin{equation} \label{eq:1.7}
n|B_1|\cdot\int\limits_0^{M/d}\mathfrak{g}(\rho)\rho^{n-1}\,d\rho
\leq\frac 1{n^n}\int\limits_{\{u>0\}}\mathfrak{g}(|Du|)\cdot \frac {({\cal L}u-b^iD_iu)_+^n}{\det({\cal A})}\,dx.
\end{equation}

We introduce the notation
$$
F=\left\|\frac {({\cal L}u)_+}{\det^{\frac 1n}({\cal A})}\right\|_{n,\{u>0\}}+\varepsilon,\quad \varepsilon>0.
$$
Then the quotient in the right-hand side of (\ref{eq:1.7}) can be estimated by the H\"older inequality:
$$
\frac {({\cal L}u-b^iD_iu)_+^n}{\det({\cal A})}
\leq(F^{\frac n{n-1}}+|Du|^{\frac n{n-1}})^{n-1}\cdot \left(\frac {({\cal L}u)_+^n}{\det({\cal A})F^n}+\mathfrak{h}^n\right).
$$
Now we put $\mathfrak{g}(\rho)=(F^n+\rho^n)^{-1}$. Then (\ref{eq:1.7}) implies
$$
n|B_1|\int\limits_0^{M/d}\frac {\rho^{n-1}}{F^n+\rho^n}\,d\rho
\leq\frac 1{n^n}\int\limits_{\{u>0\}}\!\!\frac {(F^{\frac n{n-1}}+|Du|^{\frac n{n-1}})^{n-1}}{F^n+|Du|^n}\cdot \left(\frac {({\cal L}u)_+^n}{\det({\cal A})F^n}+\mathfrak{h}^n\right)dx.
$$
By the elementary inequality
$(x+y)^{n-1}\leq 2^{n-2}(x^{n-1}+y^{n-1})$ we deduce
$$
\ln\left(1+\tfrac {M^n}{d^nF^n}\right)\leq \tfrac {2^{n-2}}{n^n|B_1|}\left(1+\|\mathfrak{h}\|^n_{n,\{u>0\}}\right),
$$
and therefore,
$$
M\leq d\cdot F\,\left(\exp\left(\tfrac {2^{n-2}}{n^n|B_1|}\left(1+
\|\mathfrak{h}\|^n_{n,\{u>0\}}\right)\right)-1\right)^{\frac 1n}.
$$
Pushing $\varepsilon\to0$ in the expression for $F$, we arrive at (\ref{eq:1.5}).
\end{proof}

\begin{rem} \label{rem:pmAB-2} If the uniform ellipticity condition (\ref{eq:uniell}) is fulfilled then, in view of Remark \ref{rem:pmAB-1}, (\ref{eq:1.5}) implies the simpler estimate:
\begin{equation} \label{eq:1.6}
\max\limits_{\overline\Omega}u_+\leq N\left(n,\tfrac {\|{\bf b}\|_{n,\{u>0\}}}{\nu}\right)\cdot\frac {\mbox {\rm diam}(\Omega)}{\nu}\cdot\|({\cal L}u)_+\|_{n,\{u>0\}}.
\end{equation}
\end{rem}


\begin{rem} \label{rem:pmAB-3} For uniformly elliptic operators of the form (\ref{eq:nondiv_operator}), the following Hopf's maximum estimate is well known (see, e.g., \cite[Theorem 3.7]{GTr1983}): 
$$
\max\limits_{\overline\Omega}u_+\leq C\left(\mbox {\rm diam}(\Omega),\tfrac {\|{\bf b}\|_{\infty,\{u>0\}}}{\nu}\right)\cdot\frac 
{\|({\cal L}u)_+\|_{\infty,\{u>0\}}}{\nu}.
$$
Here, the maximum of the solution is estimated in terms of the $L_{\infty}$-norm of the right-hand side, which turns out to be insufficient in some applications. 

On the other hand, coercive estimates in $L_r$ (\cite[Theorem 9.13]{GTr1983}) together with the Sobolev embedding theorem imply
\begin{equation}
\label{eq:1.9}
\max\limits_{\overline\Omega}u_+\leq C\cdot\|({\cal L}_0u)_+\|_{r,\Omega}, \qquad r>n/2.
\end{equation}
However, in this estimate the constant $C$ {\bf\textit{depends on the continuity moduli of the coefficients}} $a^{ij}$. Therefore, for example, for quasilinear equations, where the coefficients $a^{ij}$ depend on the solution $u$ itself and on its derivatives, the estimate (\ref{eq:1.9}) is of little use. 

The Aleksandrov--Bakelman estimate differs in that it requires neither the continuity of the principal coefficients nor the boundedness of the lower-order coefficients and the right-hand side of the equation. 
\end{rem}

In connection with Theorem \ref{thm:pmAB-3}, we mention the so-called {\bf Bony type maximum principle}. 
\medskip

\noindent\textit{Let ${\cal L}$ be an operator of the form (\ref{eq:nondiv_operator}), and let the assumption (\ref{eq:ellipt}) be satisfied. If a function $u$ attains its minimum at the point $x^0\in\Omega$, then the inequality ${\rm ess} \liminf\limits_{x\to x^0} {\cal L}u\leq 0$ holds true.
}
\medskip 

This statement was proved for operators with bounded coefficients in \cite{Bo1967} (for $u\in W^2_q(\Omega)$ with any $q>n$) and
in \cite{Ls1983} (for $u\in W^2_n(\Omega)$).\footnote{In \cite{Ls1983}, a stronger property was proved: 
$$
{\rm ess} \liminf\limits_{x\to x^0} |Du|=0; \qquad {\rm ess} \liminf\limits_{x\to x^0} D^2u\geq 0
$$
 (the second relation is understood in the sense of quadratic forms).
However, for operators with unbounded coefficients, the relation (\ref{eq:1.8}) does not follow directly from this.} We prove its variant for operators with unbounded lower-order coefficients. 

\begin{cor}
Assume that the coefficients of the operator ${\cal L}$ satisfy the conditions of Theorem \ref{thm:pmAB-3}. If a function $u\in W^2_{n,{\rm loc}}(\Omega)$ attains its minimum at the point $x^0\in\Omega$, then 
\begin{equation} \label{eq:1.8}
{\rm ess} \liminf\limits_{x\to x^0} \frac {{\cal L}u}{{\bf Tr}({\cal A})}\leq 0.
\end{equation}
\end{cor}

\begin{proof} 
As in Theorem \ref{thm:pmAB-1}, we can assume without loss of gen\-er\-al\-i\-ty that ${\bf Tr}({\cal A})\equiv1$. Let us place the origin at the point $x^0$.

We proceed by contradiction. Suppose that in some neighborhood of the origin the inequality ${\cal L}u\geq\delta>0$ holds almost everywhere. Consider the function 
$$
w^\varepsilon(x)=\varepsilon\Big(1-\frac {|x|^2}{r^2}\Big)-u(x)+u(0)
$$
in the ball $B_r$. Then $w^\varepsilon(0)=\varepsilon$, and we have $w^\varepsilon|_{\partial B_r}\leq0$ for sufficiently small $r$. Applying the estimate (\ref{eq:1.6}) to $w^\varepsilon$ in $B_r$ we obtain
$$
\varepsilon\leq N\left(n,\|\mathfrak{h}\|_{n,B_r}\right)\cdot 2r\cdot \left\|\frac {({\cal L}w^\varepsilon)_+}{\det^{\frac 1n}({\cal A})}\right\|_{n,B_r}.
$$
Since 
$${\cal L}w^\varepsilon=\frac {2\varepsilon}{r^2}\big({\bf Tr}({\cal A})+b^ix_i\big)-{\cal L}u\leq \frac {2\varepsilon}{r^2}\big(1+r|{\bf b}|\big)-\delta,
$$
we have for $\varepsilon<\frac {\delta r^2}4$
$$
\varepsilon\leq N \left\|\frac {(4\varepsilon|{\bf b}|-r\delta)_+}{\det^{\frac 1n}({\cal A})}\right\|_{n,B_r}\stackrel{(*)}{\leq} 4\varepsilon N \left\|\Big(\mathfrak{h}-\frac {r\delta}{4\varepsilon
}\Big)_+\right\|_{n,B_r}=o(\varepsilon)\quad \mbox{as} \quad\varepsilon\to 0
$$
(here the inequality ($*$) follows from the relation $\det^{\frac 1n}({\cal A})\leq{\bf Tr}({\cal A})=1$). This contradiction proves (\ref{eq:1.8}).
\end{proof}

A.D.~Aleksandrov repeatedly developed and improved the results of \cite{Al1963E}. In \cite{Al1966E}, {\bf\textit{pointwise}} estimates of solutions to the Dirichlet problems are ob\-tained in terms of the distance to the  boundary of the domain; in \cite{Al1966aE} these estimates are extended to a wider class of equations. The paper \cite{Al1966bE} is devoted to proving the attainability of the obtained estimates, while a short note \cite{Al1966dE} shows that {\bf\textit{in the general case}} it is impossible to relax the requirements on the right-hand side of the equation. Finally, in \cite{Al1967E}, pointwise estimates for the solution in terms of fine characteristics of the domain $\Omega$ are obtained. Based on this result, a gradient estimate for the solution on $\partial\Omega$ is obtained for some special cases (a summary of these results is given in \cite{Al1967bE}, \cite{Al1967aE}). 
\medskip

In the mid-1970s, N.V.~Krylov (\cite{Kr1974E}--\nocite{Kr1976aE}\cite{Kr1976E}) first obtained an Alek\-san\-drov--Bakelman type estimate for parabolic operators. After that, the study of elliptic and parabolic problems proceeded almost in parallel, but the discussion of the results for nonstationary equations is beyond the scope of our survey.\medskip

Later, the techniques based on the normal image was applied to other boundary value problems. For the {\bf oblique derivative problem} where a non-tangential directional derivative is given on the boundary, a local Aleksandrov type maximum estimate was established in \cite{Nd1988E} for bounded coefficients $b^i$ and in \cite{Na1990E} for the general case (see also \cite{Ch1997} and \cite{Li2000}). 

For the {\bf Venttsel problem} where a second order operator in tangential variables 
$$
{\cal L}' \equiv -\alpha^{ij}(x)\mathfrak{d}_i \mathfrak{d}_j+\beta^i(x)D_i,\qquad \mathfrak{d}_i\equiv D_i-{\bf n}_i{\bf n}_kD_k,\qquad \beta^i(x){\bf n}_i\leq0,
$$
is given on the boundary, the corresponding estimates were obtained in \cite{Lu1993}, \cite{LuTr1991} in the non-degenerate case (the operator ${\cal L }'$ is uniformly elliptic with respect to tangential variables) and in the degenerate case (the second-order terms in the boundary operator can vanish on a set of positive measure but the vector field $(\beta^i)$ is non-tangential to $\partial\Omega$). Later these estimates were generalized to the case of operators ${\cal L}$ and ${\cal L}'$ with unbounded lower-order coefficients \cite{AN1995}, \cite{AN1997E}. In the paper \cite{AN2001}, local Aleksandrov type estimates were established for solutions to the so-called {\bf two-phase} Venttsel problem. In all these cases, these estimates served as a ``launching pad'' for obtaining a series of a priori estimates required to prove existence theorems for solutions of quasilinear and fully nonlinear boundary value problems. 
\medskip

Another direction in the development of Aleksandrov's ideas is the trans\-fer of the maximum estimates to equations with lower-order coefficients and right-hand sides from other functional classes. The papers \cite{AN1995aE}, \cite{Li2000}, and \cite{Na2002E} dealt with various classes of operators with ``composite'' coefficients. The article \cite{Na2001E} is devoted to an Aleksandrov type estimate in terms of the norms of the right-hand side in the weighted Lebesgue spaces. Each of these results led to a corresponding extension of the class of nonlinear equations for which one can prove the solvability of the basic boundary value problems.\medskip

L.~Caffarelli \cite{Cf1988} established the Aleksandrov--Bakelman estimate for the so-called {\bf viscosity solutions} of elliptic equations. Later this idea was ac\-tive\-ly applied to various classes of nonlinear equations (see, e.g., \cite{Cf1989}, \cite[Ch.3]{CfCb1995}, \cite{DFQ2009}, \cite{I2011}, \cite{ACP2011}, \cite{CPCM2013}).
\medskip

A further group of papers is devoted to weakening the conditions on the right-hand side of the equation for {\bf\textit {certain classes}} of operators. In 1984, E.~Fabes and D.~Stroock \cite{FS1984} obtained the estimate (\ref{eq:1.9}) for operators with measurable principal coefficients under the assumption $r>r_0$, where $r_0<n$ is an exponent depending on the ellipticity constant of the operator (see also \cite{Fr1989}). In \cite{KNd2001} and \cite{Li2003} this estimate was established for the oblique derivative problem. On the other hand, C.~Pucci \cite{Pu1966b} introduced the concept of {\bf maximal and minimal operators}. Using this concept he established a lower bound for the values of $r_0$ allowing such an estimate (in this connection see \cite{ Pu1986E} and references therein). Necessary and sufficient conditions for the fulfillment of (\ref{eq:1.9}) are obtained only in the two-dimensional case \cite{AIM2006}. In some papers (see \cite{Ko2013} and references therein), the results of \cite{FS1984} have been extended to viscosity solutions of nonlinear equations. 

In \cite{KuTr2007}, a series of maximum estimates for the solution was established in terms of the $L_m$-norm of the right-hand side (here $m\in(n/2,n]$ is an integer) under the assumption that the matrix of principal coefficients belongs to some special convex cone in the space of matrices (for almost all $x\in\Omega$). Among recent advances in this direction, we mention the paper \cite{Tr2020} by N.S.~Trudinger. Undoubtedly, these studies are still far from complete. \medskip

We should also quote the paper \cite{Cb1995}, which studies the dependence of the maximum estimate on the domain characteristics. In particular, the author managed to obtain an estimate in terms of $|\Omega|^{\frac 1n}$ instead of the domain diameter (notice that for {\it convex} domains this was already done in \cite{Al1963E}).

We also mention the paper \cite{KuTr2000}, in which a discrete analogue of the Aleksandrov--Bakelman estimate for difference operators was obtained.

\subsection{Results for operators with coefficients $b^i(x)$ \\in Lebesgue spaces }\label{ss:2.4}

The simplest consequence of the Aleksandrov--Bakelman estimate is the weak maximum principle for functions $u\in W^2_n(\Omega )$ and operators of the form (\ref{eq:nondiv_operator}) with $b^i\in L_n(\Omega)$. Moreover, as pointed out already in \cite{Al1963E}, this estimate allows us to consider the operator ${\cal L}+c(x)$ with coefficient~$c(x)$ of ``bad sign''. 

\begin{cor}
Assume that the coefficients of the operator ${\cal L}$ satisfy the assumptions of Theorem~\ref{thm:pmAB-3}. Then there exists $\delta>0$ depending only on $n$, ${\rm diam}(\Omega)$ and $\|\mathfrak{h}\|_{n,\Omega}$ (the function $\mathfrak{h}$ is introduced in (\ref{eq:frak-h})) such that if
$$
h\equiv \frac {c_-}{\det^{\frac 1n}({\cal A})}\in L_n(\Omega), \qquad \|h\|_{n,\Omega}<\delta
$$
(recall that we set $0/0=0$ if such uncertainty arises), then the weak maximum principle holds for the operator ${\cal L}+c(x)$ and functions $u\in W^2_{n,{\rm loc}}(\Omega)$. 
\end{cor}

\begin{proof}
We proceed by contradiction. Suppose that ${\cal L}u+cu\geq0$ in $\Omega$ and $u\geq0$ on $\partial\Omega$, but $\min\limits_{\Omega} u=-A <0$. Apply the estimate (\ref{eq:1.5}) to the function $u^\varepsilon=-u-\varepsilon$. Since the inequality ${\cal L}u^\varepsilon=-{\cal L}u\leq cu\leq Ac_-$ holds on the set $\{u^\varepsilon>0\}$, we obtain 
\begin{equation*}
(A-\varepsilon)_+\leq N\left(n, \|\mathfrak{h}\|_{n,\Omega}\right)\cdot\mbox {\rm diam}(\Omega)\cdot\|h\|_{n,\Omega}\cdot A.
\end{equation*}
This is impossible if $N(n, \|\mathfrak{h}\|_{n,\Omega})\cdot\mbox {\rm diam}(\Omega)\cdot\|h\|_{n,\Omega}<1$ and $\varepsilon>0$ is small enough.
\end{proof}

It is easy to see that now the proof of Theorem~\ref{thm:HO} runs without changes for the so-called {\bf strong supersolutions}: $u\in W^2_n(\Omega)$, and ${\cal L}u+cu \geq0$ almost everywhere in $\Omega$ (the coefficients of the operator ${\cal L}$ are assumed to be measurable and bounded). However, in order to consider lower-order co\-ef\-fi\-cients in Lebesgue spaces, new ideas were needed. 

Note that one cannot reduce the assumption $b^i\in L_n(\Omega)$ to $b^i\in L_p(\Omega)$ with $p<n$. Indeed, the function $u(x)=|x|^2 $ satisfies the equation
$$
-\Delta u+\frac {nx_i}{|x|^2}\,D_iu=0 \quad \mbox{in} \quad B_1,
$$
but does not satisfy the maximum principle; the coefficients $b^i(x)=\frac {nx_i}{|x|^2}$ belong to the space $L_p(B_1)$ with any $p<n$ and even to the weak-$L_n$ space (the Lorentz space $L_{n,\infty}(B_1)$) but do not belong to $L_n(B_1)$. 
\medskip

The strong maximum principle for operators with $b^i\in L_n(\Omega)$ was es\-tab\-lished in \cite{Al1961E}. We prove the simplest version of this result\footnote{In \cite{Al1961E}, operators of the form ${\cal L}+c(x)$ are considered under the condition $c(x)\leq \frac {h(x)}{|x-x^0|}$, where $x^0$ is a (zero) minimum point of the function $u$, and $h\in L_n(\Omega)$. In addition, restrictions on the coefficients in this work can depend on the direction.}. 

\begin{thm}\label{thm:SMP-Al}
Let ${\cal L}$ be an operator of the form (\ref{eq:nondiv_operator}), let the condition (\ref{eq:uniell}) be satisfied, and let $b^i\in L_{n,{\rm loc} }(\Omega)$. Assume that $u\in W^2_{n,{\rm loc}}(\Omega)$, and ${\cal L}u\geq0$ almost everywhere in $\Omega$. If $u$ attains its minimum at an interior point of the domain, then $u \equiv const$ and ${\cal L} u \equiv 0$. 
\end{thm}

\begin{proof}
Suppose that $u\not\equiv const$, but the set (\ref{eq:min-set}) is not empty. As in the proof of Theorem \ref{thm:HO}, one can see that there exists a ball in the set $\Omega\setminus M$ whose boundary contains a point $x^0\in M$.
We denote the radius of the ball by $\frac r2$ and assume, without loss of generality, that $B_r\subset\Omega$. Denote $\pi= B_r\setminus \overline{B}_{\frac r4}$ and consider the barrier function (\ref{eq:barrier}) in $\pi$. 

We have
$$
{\cal L}v_s(x) \leq s|x|^{-s-2}\cdot\big(-(s+2)\nu +n\nu^{-1}+r|{\bf b}(x)|\big).
$$ 
In contrast to Theorem~\ref{thm:HO}, here we cannot achieve the inequality
${\cal L}v_s\leq0$. However, choosing $s=n\nu^{-2}$, we obtain
$$
{\cal L}v_s(x) \leq sr|x|^{-s-2}|{\bf b}(x)|\leq 4^{s+2}sr^{-s-1}|{\bf b}(x)|\quad\mbox{in}\quad\pi.
$$ 
By construction, the inequality $u(x)-u(x^0)>0$ holds on $\partial B_{\frac r4}$. Therefore, for sufficiently small $\varepsilon>0$, the function $w^\varepsilon(x)=\varepsilon v_s(x)-u(x)+u(x_0)$ is non-positive on both parts of the boundary of $\pi$. 

Application of the estimate (\ref{eq:1.6}) to $w^\varepsilon$ in $\pi$ gives
$$
w^\varepsilon(x)\leq C\left(n,\nu,\|{\bf b}\|_{n,\pi}\right)\cdot r\cdot \varepsilon \|({\cal L}v_s(x))_+\|_{n,\pi}\leq C\left(n,\nu,s,\|{\bf b}\|_{n,\pi}\right)\cdot \varepsilon r^{-s}\,\|{\bf b}\|_{n,\pi},
$$
and therefore,
\begin{equation}
\label{eq:1.10}
u(x)-u(x^0)\geq \varepsilon \Big(|x|^{-s}-r^{-s}-C\left(n,\nu,s,\|{\bf b}\|_{n,\pi}\right)\|{\bf b}\|_{n,\pi}r^{-s} 
\Big).
\end{equation}
By the Lebesgue theorem, for any $\delta>0$ one can choose $r$ so small that $\|{\bf b}\|_{n,\pi}\leq\delta$. Then the inequality (\ref{eq:1.10}) taken at the point $x^0$ becomes
$$
0\geq \varepsilon r^{-s} \Big(2^s-1-C\left(n,\nu,s,\delta\right)\delta\Big).
$$
The latter is impossible if $\delta$ is small enough.
\end{proof}

As a corollary, the following statement was proved in \cite{Al1961E}\footnote{This result is also given here in a simplified version.}. 

\begin{cor}\label{cor:Al}
Let the operator ${\cal L}$ and the function $u$ satisfy the assumptions of Theorem~\ref{thm:SMP-Al}. Let the domain $\Omega$ satisfy the interior ball condition in a neighborhood $U$ of a point $x^0\in\partial\Omega$. Suppose that \begin{equation}
\label{eq:1.11}
u\big|_{\partial\Omega\cap U}\equiv \inf\limits_{\Omega}u; \quad Du\big|_{\partial\Omega\cap U}\equiv0.
\end{equation}
Then $u\equiv const$ in $\Omega$.
\end{cor}

\begin{proof}
Extending the function $u$ by a constant outside $\Omega$ in a neighborhood of the point $x^0$, we fall into the conditions of Theorem~\ref{thm:SMP-Al}. 
\end{proof}

It is easy to see that Corollary~\ref{cor:Al} is much weaker than the normal derivative lemma, since the condition (\ref{eq:1.11}) must be satisfied on $\partial\Omega\cap U$, i.e. on a piece of the boundary. However, in contrast to the case of bounded lower-order coefficients, where the proofs of the strong maximum principle and of the normal derivative lemma are almost the same, the normal derivative lemma fails  under the conditions of Theorem~\ref{thm:SMP-Al}! We provide the corresponding counterexample (see \cite{NU2009}, \cite{Sa2010}, \cite{Na2012}).
\medskip

Let $u(x)=x_n\cdot \ln^\alpha (|x|^{-1})$ in the half-ball $B_r^+=B_r\cap\{x_n>0\}$. Then it is easily seen that $u\in W^2_n(B_r^+)$ for $r\leq \frac 12$ and $\alpha<\frac {n-1}n$. Further, direct calculation shows that $u$ satisfies the equation
$$
-\Delta u+b^n(x)D_nu=0\qquad\mbox{with} \quad 
|b^n|\le \frac{C(\alpha)}{|x|\ln(|x|^{-1})}\in L_n(B_r^+).
$$
Finally,  $u>0$ in $B_r^+$, and $u$ attains its minimum at the boundary point $0$. However, for $\alpha<0$ we evidently have $D_nu(0)=0$.

\begin{rem}
Notice that the weakened form of the normal derivative lemma (see \cite{Nd1983E}) holds true in this example. We conjecture that such a statement is fulfilled for a general uniformly elliptic operator ${\cal L}$ with $b^i\in L_n(\Omega)$, but as far as we know, this question remains open. 
\end{rem}

\begin{rem}
The above example also shows that the condition $b^i\in L_n(\Omega)$ is insufficient for the gradient estimates on $\partial\Omega$ of the solution to the Dirichlet problem. Indeed, for $\alpha>0$ we have $D_nu(0)=+ \infty$. 
\end{rem}

The article by O.A.~Ladyzhenskaya and N.N.~Uraltseva \cite{LU1988E} is of primary importance (a brief report was published three years earlier in \cite{LU1985E}).
There, for the first time, an iterative method for estimating the solution in a neighborhood of the boundary was applied. In the simplest case, it is as follows. 

Assume that a function $u$ is defined in the cylinder $Q_{1,1}$, satisfies the equation ${\cal L}u=f$ and the boundary condition $u|_{x_n=0}=0$. Let us introduce a sequence of cylinders $Q_{r_k,h_k}$, where $r_k=2^{-k}$ and $h_k$ is a suitably chosen sequence such that $h_k=o(r_k)$ as $k\to\infty$. Denote 
$$
M_k=\sup\limits_{Q_{r_k,h_k}}\frac{u(x)}{h_k}
$$
and apply the Aleksandrov--Bakelman estimate to the difference
$$
u(x)-M_kh_k \cdot \mathfrak{v}\big(\tfrac {x'}{r_k}, \tfrac {x_n}{h_k}\big),
$$
where $\mathfrak{v}$ is a certain special barrier function. The resulting estimate, taken at the points $x\in Q_{r_{k+1},h_{k+1}}$, gives a recurrence relation between $M_{k+1}$ and $M_k$. Iterating it, we obtain $\limsup\limits_k M_k<\infty$, which gives an upper bound for $D_nu(0)$ in terms of $\sup\limits_{Q_{1,1}}u$ and some integral norm of the right-hand side. 

In \cite{LU1988E}, this scheme was applied to the equation ${\cal L}u=f$ with a uniformly elliptic operator ${\cal L}$ under the following assumptions:
\begin{equation}\label{eq:1.12}
u\in W^2_n(\Omega),\quad b^i\in L_q(\Omega),\quad  f_+\in L_q(\Omega),\qquad q>n,
\end{equation}
in a domain in one of the following two classes:

1) convex domains;

2) $W^2_q$-smooth domains\footnote{This means that any point $x^0\in\partial\Omega$ has a neighborhood $U$ such that there is a $W^2_q$-diffeomorphism mapping the set $U\cap\Omega$ onto $Q_{1,1}$, and the norms of direct and inverse diffeomorphisms are estimated uniformly with respect to $x^0$. This assumption ensures that the conditions (\ref{eq:1.12}) are invariant under local flattening of the boundary.}.
\medskip

As we already mentioned in \S\,\ref{ss:2.3}, the paper \cite{AN1995aE} established an Alek\-san\-drov--Bakelman type estimate in $\Omega\subset Q_{R,R}$ for operators of the form (\ref{eq:nondiv_operator}) with ``composite'' lower-order coefficients $b^i=b^i_{(1)}+b^i_{(2)}$ provided
\begin{equation}\label{eq:1.13}
b^i_{(1)}\in L_n(\Omega),\quad \big|b^i_{(2)}(x)\big|\leq Cx_n^{\gamma-1},\quad \gamma\in(0,1).
\end{equation}
Based on this result, a bound for $\mbox{ess}\sup\partial_{\bf n}u$ on $\partial\Omega$ in a $W^2_q$-smooth domain, $q>n$, was established in \cite{AN1995aE} subject to the conditions
$$
\aligned
&b^i=b^i_{(1)}+b^i_{(2)}; && b^i_{(1)}\in L_q(\Omega), & \big|b^i_{(2)}(x)\big|\leq Cx_n^{\gamma-1},\\
&{\cal L}u= f^{(1)}+f^{(2)}; && f^{(1)}_+\in L_q(\Omega), & f^{(2)}_+(x)\leq Cx_n^{\gamma-1},
\endaligned
\qquad\gamma\in(0,1).
$$

M.V.~Safonov \cite{Sa2008} (see also \cite{Sa2018}) developed a new approach based on the boundary Harnack inequality (see \S\,\ref{ss:4.3}). By this ap\-proach, he established in a unified way
\begin{enumerate}
\item the normal derivative lemma under the condition
${\cal L}_0u\geq0$, in a domain satisfying the interior ${\cal C}^{1,{\cal D}}$-paraboloid condition\footnotemark[34];
\item an upper bound for $\partial_{\bf n}u(0)$ under the conditions ${\cal L}_0u\leq0$, $u|_{\partial\Omega\cap B_r}=0$, in a domain  satisfying the exterior ${\cal C}^{1,{\cal D}}$-paraboloid condition\footnote{In \cite{Sa2008}, the function $\phi$ defining an interior or exterior paraboloid satisfies the assumption $\int_0^\varepsilon \tau^{-2}\phi(\tau)\,d\tau<\infty$. This assumption is formally more general than the standard ${\cal C}^{1,{\cal D}}$-condition, but Lemma 2.4 in \cite{Na2012} shows that the obtained requirement on the domain is in essence equivalent to the usual one.}.
\end{enumerate}

In \cite{Sa2010}, the (slightly improved) iterative method by Ladyzhenskaya--Uraltseva was applied\footnote{It was noted in \cite{Sa2010} that the normal derivative lemma under assumption $b^i\in L_q(\Omega)$, $q>n$, was in fact obtained already in \cite[Lemma 4.4]{LU1988E}. This fact remained unnoticed for more than 20 years!} to derive the normal derivative lemma in the domain $\Omega=Q_{R,R}$ under the conditions
$$
u\in W^2_{n,{\rm loc}}(\Omega)\cap {\cal C}(\overline{\Omega}),\quad \min\limits_{\overline\Omega} u=u(0);\qquad b^i\in L_n(\Omega),\quad b^n\in L_q(\Omega),\quad q>n.
$$
Thus, it turns out that, in comparison with $b^i\in L_n(\Omega)$, it suffices to strength\-en the assumption only on the normal component of the vector ${\bf b}$.
\medskip

In \cite{Na2012}, both the normal derivative lemma and the gradient estimate of the solution to the Dirichlet problem on the boundary of the domain are obtained under the currently sharpest conditions. Moreover, the duality of these statements is explicitly demonstrated. The result is achieved by a combination of the Ladyzhenskaya--Uraltseva--Safonov technique and the Aleksandrov--Bakelman type estimate \cite{Li2000}, where the assumption on $b^i_{(2)}$ from (\ref{eq:1.13}) is refined to $\big |b^i_{(2)}(x)\big|\leq\frac {\sigma(x_n)}{x_n}$, $\sigma\in{\cal D}$.
\smallskip

We give the formulation of this result.

\begin{thm}\label{thm:Na12} 
Let ${\cal L}$ be a uniformly elliptic operator of the form (\ref{eq:nondiv_operator}) in $\Omega=Q_{R,R}$. Let $b^i=b^i_{(1)}+b^i_{(2)}$, and let the following conditions be satisfied:
$$
b^i_{(1)}\in L_n(\Omega), \quad \|b^n_{(1)}\|_{n,Q_{r,r}}\leq \sigma(r)\ \ \mbox{for} \ \ r\leq R;\quad \big|b^i_{(2)}(x)\big|\leq\frac {\sigma(x_n)}{x_n};\quad \sigma\in{\cal D}.
$$
Suppose that $u\in W^2_{n,{\rm loc}}(\Omega)\cap {\cal C}(\overline{\Omega})$. 
\begin{enumerate}
\item If $u>0$ in $Q_{R,R}$, $u(0)=0$, and ${\cal L}u\geq0$, then
$$
\inf\limits_{0<x_n<R}\frac{u(0,x_n)}{x_n}>0.
$$
\item If $u|_{x_n=0}\le0$, $u(0)=0$, and ${\cal L}u=f^{(1)}+f^{(2)}$, where
$$
\|f^{(1)}_+\|_{n,Q_{r,r}}\leq \sigma(r)\quad  \mbox{for} \quad r\leq R;\qquad f^{(2)}_+(x)\leq\frac {\sigma(x_n)}{x_n},
$$
then
$$
\sup\limits_{0<x_n<R}\frac{u(0,x_n)}{x_n}\leq C,
$$
where $C<\infty$ is determined by known quantities.
\end{enumerate}
\end{thm}

It is important to note that the occurrence of term $b^i_{(2)}$ allows us to perform a coordinate transformation using the regularized distance in a neighborhood of an insufficiently smooth boundary. Thus, Theorem~\ref{thm:Na12} implies corresponding assertions in domains satisfying the interior/exterior ${\cal C}^{1,{\cal D}}$-paraboloid condition\footnote{Cf. \cite{HLW2020}, where the existence of $D_nu(0)$ is proved for viscosity solutions of the equation ${\cal L}_0u=f$.}.
\medskip

In \cite{AN2016}, a new counterexample was constructed. It shows the sharpness of the interior ${\cal C}^{1,{\cal D}}$-paraboloid condition for the normal derivative lemma. We present its formulation in the simplest case.

\begin{thm}\label{thm:Hopf-counter}
Let $\Omega$ be locally convex in a neighborhood of the origin, that is,
$$
\Omega\cap B_R=\big\{x\in\R^n\,\big|\, F(x')<x_n<\sqrt{R^2-|x'|^2}\big\}
$$
for some $R>0$. Here $F$ is a convex function, $F \geq 0$, and $F(0)=0$.

Further, suppose that $u \in W^2_{n,{\rm loc}}( \Omega) \cap {\cal C}( \overline{\Omega})$ is a solution of the equation ${\cal L}_0u=0$ with a uniformly elliptic operator ${\cal L}_0$, and $u|_{\partial\Omega\cap B_R}=0$. 

If the function
$$
\delta (r)=\sup\limits_{|x'|\leq r}\frac{F(x')}{|x'|}
$$ 
does not satisfy the Dini condition at zero then $\lim\limits_{\varepsilon\to+0}\,\frac  {u(\varepsilon x_n)}\varepsilon=0$.
\end{thm}

Note that if $\delta(r)$ satisfies the Dini condition at zero then $\Omega$ satisfies the interior ${\cal C}^{1,{\cal D}}$-paraboloid condition at the origin. Thus, for {\bf\textit{locally convex}} domains, the Dini condition at zero for the function $\delta(r)$ is necessary and sufficient for the validity of the normal derivative lemma.

We emphasize that all previous counterexamples of this type (\cite{W1967}, \cite{VM1967E}, \cite{Hi1970E} and \cite{Sa2008}) require the absence of the Dini condition for the function $\inf\limits_{|x'|\leq r}\frac{F(x')}{|x'|}$. Roughly speaking, in those counterexamples the Dini condition must fail in all directions, whereas in Theorem \ref{thm:Hopf-counter} it is enough to violate it in one direction.
\medskip

For general domains, a more subtle counterexample was constructed in \cite{Sa2018}. However, it is too complicated to describe it here.

\subsection{The Harnack inequality}

As already mentioned in the Introduction, the Harnack inequality, which can be considered as a quantitative version of the strong maximum principle, was first proved by C.G.A. Harnack \cite{H1887} for harmonic functions on the plane. Since Harnack's proof is based on the Poisson formula, it is obviously valid in any dimension. Harnack's formulation is included into most textbooks:
\medskip

\textit{If $u\geq0$ is a harmonic function in $B_R\subset\R^n$ then}
\begin{equation}\label{eq:Harnack1}
u(0)\,\frac {(R-|x|)R^{n-2}}{(R+|x|)^{n-1}}\leq u(x)\leq u(0)\,\frac {(R+|x|)R^{n-2}}{(R-|x|)^{n-1}}.
\end{equation}
For $\Omega=B_R$ and $\Omega'=B_{\theta R}$, $\theta<1$, this immediately implies (\ref{eq:Harnack}) with $C=\big(\frac {1+\theta}{1-\theta}\big)^n$.
\bigskip

In this Section we assume that the uniform ellipticity condition (\ref{eq:uniell}) is satisfied.
\medskip

L.~Lichtenstein in \cite{L1912} proved the inequality (\ref{eq:Harnack}) for operators of general form ${\cal L}+c(x)$, $c\geq0$, with ${\cal C}^2$-smooth coefficients (as in  \cite{H1887}, in the two-dimensional case).

J. Serrin \cite{S1956} established the Harnack inequality in the two-dimen\-sion\-al case for operators ${\cal L}+c(x)$, $c\geq0$, with {\bf\textit{bounded}} coefficients. This result was obtained simultaneously and independently in \cite{BNi1955}. For the case $n\geq3$, Serrin also proved (\ref{eq:Harnack}) under the condition\footnote{More precisely, the principal coefficients of the operator must satisfy the Dini condition in some neighborhood of $\partial\Omega$.}
$a^{ij}\in{\cal C}^{0,{\cal D}}(\Omega)$.
\medskip

An important improvement was made by E.M.~Landis \cite{La1968aE} (see also \cite[Ch.~1]{La1971E}). Using the {\bf growth lemma} proposed by himself, he proved the Harnack inequality in arbitrary dimension for the operator ${\cal L}_0$ with bounded co\-ef\-fi\-cients under the additional assumption that the eigenvalues of the matrix ${\cal A}$ have sufficiently small dispersion\footnote{Conditions of this form were first introduced in \cite{Cor1956E}, which is why Landis calls (\ref{eq:Cordes}) the Cordes type condition.}. Namely, the following relations are assumed to hold (after multiplying the matrix ${\cal A}$ by a suitable positive function):
\begin{equation}
\label{eq:Cordes}
{\bf Tr}({\cal A})\equiv1, \qquad \nu>\frac 1{n+2}
\end{equation}
(obviously, the inequality $\nu\leq \frac 1n$ always holds, and equality is possible only for the Laplace operator).

Notice that all the above results were obtained for classical solutions $u\in{\cal C}^2(\Omega)$.
\medskip

Finally, the key step belongs to N.V.~Krylov and M.V.~Safonov \cite{KrSa1980E}, \cite{Sa1980E} (see also \cite{KrSa1979E}). Combining the Landis method with Alek\-san\-drov--Bakelman estimates (in the elliptic case) and Krylov estimates \cite{Kr1974E}--\cite{Kr1976E} (in the parabolic case), they managed to obtain the inequality (\ref {eq:Harnack}) for {\bf\textit{strong}} solutions of elliptic \cite{Sa1980E} and parabolic \cite{KrSa1980E} equations with general operators ${\cal L}+c(x)$, $c \geq0$ (with bounded coefficients), without  assuming that the matrix ${\cal A}$ is continuous or that the dispersion of its eigenvalues is small\footnote{Note that if $c\equiv0$, then the Harnack inequality easily implies an a priori estimate for the H\"older norm of a solution. Extending this estimate (also obtained in \cite{KrSa1980E}, \cite{Sa1980E}) to quasilinear equations, O.A.~Ladyzhenskaya and N.N.~Uraltseva further established the solvability of the Dirichlet problem for non-divergence quasilinear equations under natural structure conditions only (see the survey \cite{LU1986E}). Subsequently, this result was extended to other boundary value problems for quasilinear and fully non-linear equations.}.\medskip

For operators ${\cal L}$ with $b^i\in L_n(\Omega)$, the Harnack inequality was proved in \cite{Sa2010} (see also \cite{Na1990E}). The papers \cite{FrSa2001} and \cite{Sa2015E} demonstrate a unified approach to proving the Harnack inequality for divergence and non-di\-ver\-gence type operators. At the same time, \cite{FrSa2001} showed\footnote{See also \cite{ChS2017} in this connection.} that for operators of mixed (divergence-non-divergence) form
$$
-D_i(a^{ij}(x)D_j)-\tilde a^{ij}(x)D_iD_j
$$
(matrices of the principal coefficients ${\cal A}$ and $\widetilde{\cal A}$ satisfy the uniform ellipticity condition) the Harnack inequality can fail even for $n=1$.

We also mention the papers \cite{AmFT2001} and \cite{Sa2020E}, where the Harnack inequality and the H\"older continuity of solutions were considered in the ``abstract'' context of metric and quasi-metric spaces.

\section{Divergence type operators}\label{sec:div}

In this Section, we consider operators with the following structure:
\begin{equation}
{\mathfrak L} \equiv -D_i(a^{ij}(x)D_j)+b^i(x)D_i
\label{eq:div_operator}
\end{equation}  
(in the case ${\bf b}\equiv 0$, we write ${\mathfrak L}_0$ instead of ${\mathfrak L}$) and operators of more general form
\begin{equation}\label{eq:hat-L}
\widehat{\mathfrak L}\equiv -D_i(a^{ij}(x)D_j+d^i)+b^i(x)D_i+c(x).
\end{equation}

The matrix of principal coefficients ${\cal A}$ is symmetric and satisfies the ellipticity condition 
\begin{equation}
\label{eq:ell-x}
\nu(x)|\xi|^2\leq a^{ij}(x)\xi_i\xi_j\leq{\cal V}(x)|\xi|^2 \quad\mbox{for all}\ \  \xi\in\R^n
\end{equation}
or the uniform ellipticity condition (\ref{eq:uniell}) for almost all $x\in\Omega$. In (\ref{eq:ell-x}), the functions $\nu(x)$ and ${\cal V}(x)$ are positive and finite\footnote{We emphasize that, in contrast to operators of non-divergence type, the properties of the operator ${\mathfrak L}$ are not preserved when multiplied by an arbitrary positive function. Therefore, the behavior of the functions $\nu(x)$ and ${\cal V}(x)$ should be considered separately.} almost everywhere in $\Omega$.

The solution of the equation $\widehat{\mathfrak L}u=0$ is understood here as a {\bf weak solution}, i.e. a function $u\in W^1_{2,{\rm loc}}(\Omega)$ such that the {\bf integral identity}
$$
\langle \widehat{\mathfrak L}u,\eta\rangle:=
\int\limits_{\Omega}(a^{ij}D_juD_i\eta+b^iD_iu\,\eta+d^iuD_i\eta+cu\eta)\,dx=0
$$
is satisfied for arbitrary test function $\eta\in {\cal C}^\infty_0(\Omega)$. Respectively, a {\bf weak supersolution} ($\widehat{\mathfrak L}u\geq0$)  is a function $u\in W^1_{2,{\rm loc}}(\Omega)$ such that
\begin{equation}
\label{eq:supersol}
\int\limits_{\Omega}(a^{ij}D_juD_i\eta+b^iD_iu\,\eta+d^iuD_i\eta+cu\eta)\,dx \geq 0
\end{equation}
for arbitrary {\bf\textit{non-negative}} test function $\eta\in {\cal C}^\infty_0(\Omega)$. A weak subsolution ($\widehat{\mathfrak L}u\leq0$) is defined in a similar way.
\medskip

Let us prove the weak maximum principle for the operator $\widehat{\mathfrak L}$ under the simplest restrictions on the coefficients.

\begin{thm}\label{thm:WMP-div}
Let $n\geq3$. Suppose that $\widehat{\mathfrak L}$ is an operator of the form (\ref{eq:hat-L}) in a domain $\Omega\subset\mathbb R^n$, the condition (\ref{eq:uniell}) is fulfilled, 
$$
b^i, d^i\in L_n(\Omega),\quad c\in L_{\frac n2}(\Omega),
$$
and the function ${\mathfrak u}\equiv1$ is a weak supersolution of the equation $\widehat{\mathfrak L}u=0$ in $\Omega$.

Let $u\in W^1_{2,{\rm loc}}(\Omega)$, $\widehat{\mathfrak L}u\geq0$ in $\Omega$, and $u\geq0$ on $\partial\Omega$.\footnote{Similarly to the footnote \ref{foot:28}, this means that for any $\varepsilon>0$ the inequality $u+\varepsilon>0$ holds in some neighborhood of $\partial\Omega$.} Then $u\geq0$ in $\Omega$.
\end{thm}

\begin{proof}
1. First of all, note that the bilinear form $\langle \widehat{\mathfrak L}u,\eta\rangle$ is continuous on $W^1_{2,{\rm loc}}(\Omega)\times{\stackrel{\circ}W\vphantom {W}\!\!^1_2}(\Omega')$ if $\overline{\Omega}{}'\subset\Omega$. Indeed, applying the H\"older and Sobolev inequalities, we obtain
$$
\aligned
|\langle \widehat{\mathfrak L}u,\eta\rangle|\leq &\  \nu^{-1}\|Du\|_{2,\Omega'}\|D\eta\|_{2,\Omega'}+\|{\bf b}\|_{n,\Omega}\|Du\|_{2,\Omega'}\|\eta\|_{2^*,\Omega'}\\
+ &\ \|{\bf d}\|_{n,\Omega}\|D\eta\|_{2,\Omega'}\|u\|_{2^*,\Omega'}+\|c\|_{\frac n2,\Omega}\|u\|_{2^*,\Omega'}\|\eta\|_{2^*,\Omega'}\\
\leq &\ C\big(\|Du\|_{2,\Omega'}+\|u\|_{2,\Omega'}\big)\cdot\|D\eta\|_{2,\Omega'}
\endaligned
$$
(here and below $2^*=\frac {2n}{n-2}$ is the critical Sobolev exponent). Therefore, in the definition of a weak solution (sub/supersolution), one can take any test functions $\eta\in{\stackrel{\circ}W\vphantom {W}\!\!^1_2}(\Omega)$ with compact support.
\medskip

2. Assume the converse. Let ${\rm ess}\inf\limits_{\Omega}u=-A<0$ (the case $A=\infty$ is not excluded). Then for any $0<k<A$ the function $\eta=(u+k)_-\in{\stackrel{\circ}W\vphantom {W}\!\!^1_2}(\Omega)$ is non-negative and compactly supported in $\Omega$, and therefore the inequality (\ref{eq:supersol}) holds true. Since $D(u+k)_-=-Du\cdot\chi_{\{u<-k\}}$, this gives
$$
\aligned
\int\limits_{\{u<-k\}}\!a^{ij}D_juD_iu\,dx\leq &  \int\limits_{\{u<-k\}}\!(b^iD_iu\,\eta+d^iuD_i\eta+cu\eta)\,dx \\
= & \int\limits_{\{u<-k\}}\!(b^i-d^i)D_iu\,\eta\,dx+\int\limits_{\{u<-k\}}\!(d^iD_i(u\eta)+c(u\eta))\,dx.
\endaligned
$$
The latter term here is non-positive, since ${\mathfrak u}\equiv1$ is a weak supersolution. Using (\ref{eq:uniell}) in the left-hand side and the H\"older and Sobolev inequalities in the right-hand side, we arrive at
\begin{equation}
\label{eq:u<-k}
\nu \|Du\|^2_{2,\{u<-k\}}\leq 
(\|{\bf b}\|_{n,\{u<-k\}}+\|{\bf d}\|_{n,\{u<-k\}}) \|Du\|^2_{2,\{u<-k\}}. 
\end{equation}
If $A=\infty$ then the first factor in the right-hand side tends to zero as $k\to\infty$, which gives a contradiction.

If $A<\infty$ then $Du=0$ almost everywhere on the set $\{u=-A\}$, and we can rewrite (\ref{eq:u<-k}) as follows: $\nu\leq \|{\bf b}\|_{n,\mathscr{A}_k}+\|{\bf d}\|_{n,\mathscr{A}_k}$,
where
$$
\mathscr{A}_k=\{x\in\Omega\,\big|\, -A<u(x)<-k,\ Du(x)\ne0\}.
$$
Evidently, $|\mathscr{A}_k|\to0$ as $k\to A$. Therefore, $\|{\bf b}\|_{n,\mathscr{A}_k}+\|{\bf d}\|_{n,\mathscr{A}_k}\to0$, and we again have a contradiction.
\end{proof}

\begin{rem}
In a recent paper \cite{KR2020}, the weak maximum principle is proved in the so-called John domain for functions $u\in W^1_2(\Omega)$ under the following assumptions:
\begin{itemize}
    \item $\widehat{\mathfrak L}u\geq0$ in $\Omega$;
    \item the conormal derivative condition  $(a^{ij}D_ju+d^iu){\bf n}_i\geq0$ is satisfied instead of $u\geq0$ on $\partial\Omega$ (this means that the inequality (\ref{eq:supersol}) holds for all non-negative functions $\eta\in W^1_2(\Omega)$).

\end{itemize}
\end{rem}

\subsection{The Harnack inequality and the strong maximum principle}

In contrast to non-divergence type operators\footnote{Compare the years of the first obtained results in the table:
\vspace{0.1cm}

\begin{tabular} {|l|c|c|}
\hline & Strong max. principle & The Harnack inequality \\
\hline The Laplace operator & 1839 \cite{G1839E}, \cite{E1839} & 1887 \cite{H1887}\hphantom{, [11]} \\
\hline Operators with smooth coeff. & 1892 \cite{P1892}\hphantom{, [111]} & 1912 \cite{L1912}\hphantom{, [11]} \\
\hline Operators with discont. coeff. & 1927 \cite{H1927}\hphantom{, [111]} & 1955 \cite{S1956}, \cite{BNi1955} \\
\hline\end{tabular}
\smallskip
}, almost all results on the strong maximum principle for divergence type operators were obtained as a consequence of the corresponding Harnack inequalities. In this regard, we present the history of these results in parallel.
\medskip

The works of W. Littman \cite{Lt1959}, \cite{Lt1963} stand somewhat apart. They deal with operators
\begin{equation}
\label{double-div}
{\cal L}^*\equiv -D_i D_ja^{ij}(x)-D_ib^i(x),
\end{equation}
formally adjoint to operators of the form (\ref{eq:nondiv_operator}). A weak supersolution to the equation ${\cal L}^*u+cu=0$ is a function $u\in L_{1,{\rm loc}}(\Omega)$ such that for any non-negative test function $\eta\in {\cal C}^\infty_0(\Omega)$ the inequality
$$
\langle {\cal L}^*u+cu,\eta\rangle:=\int\limits_{\Omega}u ({\cal L}\eta+c\eta)\,dx\geq0
$$
holds true. In \cite{Lt1959} the coefficients of the operator were assumed smooth, while in \cite{Lt1963} the conditions were substantially weakened. Let us formulate the latter result.
\medskip

\noindent\textit{Let ${\cal L}$ be an operator of the form (\ref{eq:nondiv_operator}). Suppose that $a^{ij}$, $b^i$, and $c$ belong to ${\cal C}^{0,\alpha}(\Omega)$, $\alpha\in(0,1)$, and the assumption (\ref{eq:uniell}) is satisfied. Let $u$ be a weak supersolution to the equation ${\cal L}^*u+cu=0$ in $\Omega$. Then}
\vspace{-0.1cm}

\begin{enumerate}
\item \textit{$u$ cannot attain zero minimum in $\Omega$ unless $u\equiv0$.}
\item
\textit{If ${\mathfrak u}\equiv1$ is a weak supersolution to the equation ${\cal L}^*u+cu=0$ in $\Omega$\footnote{In this case, it means that $-D_i D_j(a^{ij})-D_i(b^i)+c\geq0$ in the sense of distributions.} then $u$ cannot attain negative minimum in $\Omega$ unless $u\equiv const$ (in this case $u$ is a weak solution).}
\item
\textit{If $-{\mathfrak u}$ is a weak supersolution to the equation ${\cal L}^*u+cu=0$ in $\Omega$ then $u$ cannot attain positive minimum in $\Omega$ unless $u\equiv const$ (in this case $u$ is a weak solution).}
\end{enumerate}
Further developments of these results for operators of the form (\ref{double-div}) can be found in the papers \cite{Es1994}, \cite{Es2000}, \cite{MMcO2007},  \cite{DEK2018} (see also the references therein).\medskip

Let us return to divergence type equations. The Harnack inequality for a uniformly elliptic operator ${\mathfrak L}_0$ with measurable coefficients was first proved by J.~Moser \cite{Mo1961}.\footnote{ As shown in \cite{DiBTr1984} (see also \cite{DiB1989} and \cite{LZh2017}), the inequality (\ref{eq:Harnack}) can also be obtained from E.~De Giorgi's proof \cite{DeG1957} of the H\"older continuity of weak solutions to the equation ${\mathfrak L}_0u=0$.} In the paper by G.~Stampacchia \cite{St1965} this result was generalized to operators of the form (\ref{eq:hat-L}) under the conditions
\begin{equation}
\label{eq:Stamp}
b^i\in L_n(\Omega),\quad d^i\in L_q(\Omega),\quad c\in L_{\frac q2}(\Omega),\qquad q>n.
\end{equation}
A similar result can be extracted from the paper \cite{S1964} devoted to quasilinear equations.

As a corollary, the strong maximum principle is proved in \cite{St1965}\footnote{See also \cite{Ch1967} and \cite{HH1969}.} in two versions:
\begin{enumerate}
    \item for the operator $\widehat{\mathfrak{L}}$ provided ${\rm ess}\inf\limits_{\Omega}u=0$;
    \item for the operator ${\mathfrak{L}}$.
\end{enumerate}
We give a somewhat simplified proof of the second assertion. This proof is based on Moser's idea \cite{Mo1960}, but does not use the Harnack inequality.

\begin{thm}\label{thm:SMP-div}
Let $\mathfrak{L}$ be a uniformly elliptic operator of the form (\ref{eq:div_operator}) in the domain $\Omega\subset\mathbb R^n$, $n\ge3$, and let $b^i\in L_n(\Omega)$. Suppose that $u\in W^1_{2,{\rm loc}}(\Omega)$ and ${\mathfrak L}u\geq0$ in $\Omega$. If $u$ attains its minimum\footnote{This statement is understood as follows: ${\rm ess}\liminf\limits_{x\to x^0}u={\rm ess}\inf\limits_{\Omega}u$.} at a point $x^0\in\Omega$, then $u\equiv const$. 
\end{thm}

\begin{proof} 1. Similarly to Step 1 of the proof of Theorem~\ref{thm:WMP-div}, in the definition of a weak solution (sub/supersolution) one can take any compactly supported test function $\eta\in{\stackrel{\circ}W\vphantom {W}\!\!^1_2}(\Omega)$.

\medskip

2. Now let $v$ be a weak subsolution, i.e. ${\mathfrak L}v\leq0$ in $\Omega$. We substitute the test function $\eta=\varphi'(v)\cdot\varsigma$ into the inequality $\langle {\mathfrak L}v,\eta\rangle\leq0$. Here $\varsigma$ is a non-negative Lipschitz function supported in
$\overline{B}_{2R}\subset\Omega$, and $\varphi$ is a convex Lipschitz function on $\mathbb R$ that vanishes on the negative semiaxis. This gives
\begin{equation}\label{eq:Moser}
\int\limits_{B_{2R}\cap\{u>0\}}\Big(a^{ij}D_jVD_i\varsigma+\frac {\varphi''(v)}
{\varphi^{\prime2}(v)}\,a^{ij}D_jVD_iV\varsigma+b^iD_iV\varsigma\Big)\,dx\leq0,
\end{equation}
where $V=\varphi(v)\in W^1_{2,{\rm loc}}(\Omega)$. In particular, since the second term in (\ref{eq:Moser}) is non-negative, $V$ is also a weak subsolution.

We set\footnote{To be more formal, one should take $\varphi'(v)=p\,\min\{v_+,N\}^{p-1}$ for some $N>0$, and $\varsigma=\varphi(v)\zeta^2$, and then pass to the limit as $N\to\infty$ in (\ref{eq:ee}).} in (\ref{eq:Moser}) $\varphi(\tau)=\tau_+^p$, $p>1$, and $\varsigma=V\zeta^2$, where $\zeta$ is a smooth cut-off function in $B_{2R}$. We arrive at
\begin{equation}\label{eq:Moser_power}
\int\limits_{B_{2R}}\!\tfrac {2p-1}{p}\,a^{ij}D_jVD_iV\,\zeta^2\,dx\leq -\int\limits_{B_{2R}}\!\Big(2a^{ij}D_jV\,VD_i\zeta\,\zeta+b^iD_iV\,V\zeta^2\Big)\,dx.
\end{equation}
We estimate the left-hand side in (\ref{eq:Moser_power}) from below using (\ref{eq:uniell}),
and the right-hand side from above by the H\"older and Sobolev inequalities:
$$
\aligned
&\ \nu \|DV\,\zeta\|^2_{2,B_{2R}}\\
\leq &\  2\nu^{-1}\|DV\,\zeta\|_{2,B_{2R}}\|VD\zeta\|_{2,B_{2R}}+\|{\bf b}\|_{n,B_{2R}}\|DV\,\zeta\|_{2,B_{2R}}\|V\,\zeta\|_{2^*,B_{2R}}\\
\leq &\ N(n)\|{\bf b}\|_{n,B_{2R}} \|DV\,\zeta\|^2_{2,B_{2R}}+C\|DV\,\zeta\|^2_{2,B_{2R}}\|VD\zeta\|^2_{2,B_{2R}}.
\endaligned
$$
By the Lebesgue theorem, we have $N(n)\|{\bf b}\|_{n,B_{2R_*}}\leq\frac {\nu}2$ for sufficiently small $R_*$. For $R\leq R_*$ this implies
\begin{equation}\label{eq:ee}
\|DV\,\zeta\|_{2,B_{2R}}\le
C(n,\nu,\|{\bf b}\|_{n,\Omega})\cdot \|VD\zeta\|_{2,B_{2R}}.
\end{equation}

We put in (\ref{eq:ee}) $R_k=R(1+2^{-k})$, $k\in \mathbb N\cup\{0\}$, and take $\zeta=\zeta_k$ such that
$$ 
\zeta_k\equiv1\ \ \mbox{in}\ \ B_{R_{k+1}},\quad \zeta_k\equiv0\ \ \mbox{outside}\ \ B_{R_k},
\qquad |D\zeta_k|\leq \frac {2^{k+2}}{R}.
$$ 
We obtain
\begin{equation}\label{eq:eee}
\|DV\,\zeta_k\|_{2,B_{R_k}}\leq
\frac {C(n,\nu,\|{\bf b}\|_{n,\Omega})}{R}\cdot 2^k\|V\|_{2,B_{R_k}}.
\end{equation}
Now for $p=p_k\equiv (2^*/2)^k$ we deduce from the Sobolev inequality and (\ref{eq:eee}) that
\begin{multline}\label{eq:iteration}
\bigg(\!\Xint{\quad\ -}\limits_{B_{R_{k+1}}}\!\! v_+^{2p_{k+1}}dx\bigg)^{\frac {1}{2p_{k+1}}}\!\!\leq
\bigg(N(n)\Xint{\ \,  -}\limits_{B_{R_k}} (V\zeta_k)^{2^*} dx\bigg)^{\frac {1}{2^*p_k}}\\
\leq 
\bigg(4^kC\Xint{\ \,  -}\limits_{B_{R_k}} V^2\,dx\bigg)
^{\frac {1}{2p_k}}\!=
\bigg(4^kC\Xint{\ \,  -}\limits_{B_{R_k}} v_+^{2p_k}
dx\bigg)^{\frac 1{2p_k}}\!,
\end{multline}
where $C$ depends only on $n$, $\nu$, and $\|{\bf b}\|_{n,\Omega}$.

Iterating (\ref{eq:iteration}), we conclude that any (weak) subsolution $v$ admits the estimate
\begin{equation}\label{eq:estmax}
{\rm ess}\sup\limits_{B_{R}} v_+\leq C(n,\nu,\|{\bf b}\|_{n,\Omega}) \cdot\bigg(\Xint{\ \, -}\limits_{B_{2R}}\!v_+^2dx\bigg)^{\frac 12}, \qquad R\leq R_*.
\end{equation}

3. Let us now turn to the proof of the Theorem. Without loss of generality, we can suppose that ${\rm ess}\inf\limits_\Omega u=0$.

Assume that the statement is wrong. Then there is an interior point $x^0\in\Omega$ such that ${\rm ess}\liminf\limits_{x\to x^0}u=0$, but for some $k>0$, $\delta >0$ and $R\leq R_*$ the inequality
\begin{equation}\label{eq:tiny1}
\big|\{u\ge k\}\cap B_R(x^0)\big|\geq\delta\cdot 
|B_R|
\end{equation}
holds. Without loss of generality, we assume that
$\overline{B}_{2R}(x^0)\subset\Omega$.
We place the origin at the point $x^0$ and introduce the function $v_\varepsilon(x)=1-\frac uk -\varepsilon$, $\varepsilon>0$. Obviously $v_\varepsilon$ is a subsolution.
\smallskip

We apply the inequality (\ref{eq:Moser}) with $V=\varphi(v^\varepsilon)\equiv \big(\ln\frac 1{1-v^\varepsilon}\big)_+$ (this is allowed since $v^\varepsilon<1$) and $\varsigma=\zeta^2$, where $\zeta$ is a smooth cut-off function equal to $1$ in $B_R$. Taking into account that $\frac {\varphi''}{\varphi^{\prime2}}\equiv1$ and using (\ref{eq:uniell}) together with the H\"older and Sobolev inequalities, we obtain
$$
\aligned
\nu\|DV\,\zeta\|^2_{2,B_{2R}}\leq &\
\!\int\limits_{B_{2R}}a^{ij}D_jVD_iV\,\zeta^2\,dx\leq - \!\int\limits_{B_{2R}}\!\!\Big(2a^{ij}D_jV\,\zeta D_i\zeta+b^iD_iV\,\zeta^2\Big)\,dx\\
\leq &\ C(n,\nu,\|{\bf b}\|_{n,\Omega})\cdot \|DV\,\zeta\|_{2,B_{2R}}\|D\zeta\|_{2,B_{2R}},
\endaligned
$$
whence
\begin{equation}\label{eq:eeee}
\|DV\,\zeta\|_{2,B_{2R}}\leq C(n,\nu,\|{\bf b}\|_{n,\Omega})R^{\frac n2-1}.
\end{equation}

Now we observe that $V$ vanishes on the set $\{u\ge k\}\cap B_R$, and $\zeta\equiv 1$ on this set. It follows from the proof of Lemma 5.1 in \cite[Chapter II]{LSU1967E} that this implies
$$
\big|\{u\ge k\}\cap B_R\big|\cdot V(x)\,\zeta(x)\le
\frac {(4R)^n}{n}\int\limits_{B_{2R}}\frac {|DV(y)|\,\zeta(y)}{|y-x|^{n-1}}\,dy.
$$
We take in both parts the norm in $L_{2^*}$ and estimate the right-hand side by the Hardy--Littlewood--Sobolev inequality (see, e.g., \cite[Theorem 1.18.9/3]{Tb1978}). Taking into account (\ref{eq:tiny1}) and (\ref{eq:eeee}) we obtain
$$
\|V\zeta\|_{2^*, B_{2R}}\leq \frac {C(n)}{\delta}\, \|DV\,\zeta\|_{2,B_{2R}}\leq C(n,\nu,\delta,\|{\bf b}\|_{n,\Omega})R^{\frac n2-1},
$$
and therefore,
$$
\bigg(\Xint{\, -}\limits_{B_R}V^2dx\bigg)^{\frac 12}\leq C(n)R^{1-\frac n2} 
\|V\zeta\|_{2^*, B_{2R}}\leq
C(n,\nu,\delta,\|{\bf b}\|_{n,\Omega}).
$$
Finally, since $V$ is a subsolution, we can apply the estimate (\ref{eq:estmax}). This gives ${\rm ess}\sup\limits_{B_{R/2}} V_+\leq C$, which is equivalent to
$$
{\rm ess}\inf\limits_{B_{R/2}}u\geq k(\exp(-C)-\varepsilon).
$$
Since the constant $C$ does not depend on $\varepsilon$, we get a contradiction with the assumption ${\rm ess}\liminf\limits_{x\to 0}u=0$.
\end{proof}

\begin{rem}
The last term in (\ref{eq:Moser_power}) can be estimated as follows:
$$
\int\limits_{B_{2R}}|b^iD_iV\,V\zeta^2|\,dx
\leq \|{\bf b}\|_{L_{n,\infty}(B_{2R})}\|DV\,\zeta\|_{2,B_{2R}}\|V\,\zeta\|_{L_{2^*,2}(B_{2R})}
$$
(recall that $L_{p,q}$ is the Lorentz space). Then one can use the strengthened Sobolev embedding theorem ${\stackrel{\circ}W\vphantom {W}\!\!^1_2}(\Omega) 
\hookrightarrow L_{2^*,2}(\Omega)$ (see \cite{Pe1966}). This implies that the assumption $b^i\in L_n(\Omega)$ can be weakened to $b^i\in L_{n,q}(B_{2R})$ with any $q<\infty$. The counterexample in the beginning of \S\,\ref{ss:2.4} shows that one cannot put in general $q=\infty$. However, if the norm $\|{\bf b}\|_{L_{n,\infty}(\Omega)}$ is sufficiently small then the proof runs without changes. 

The Harnack inequality also holds true under the same assumptions (the proof of Theorem 2.5\,$^\prime$ in \cite{NU2011E} can be transferred completely to this case).
\end{rem}

\begin{rem}
In the two-dimensional case, the statement of Theorem~\ref{thm:SMP-div} (and even Theorem~\ref{thm:WMP-div}) is false\footnote{This fact is not noted in \cite{St1965} and \cite{Ch1967}.}; here is a corresponding counterexample from the paper \cite{Fi2013E}.
\medskip

For $n=2$ we set $u(x)=\ln^{-1}(|x|^{-1})$. Obviously, for $r\leq \frac 12$ the function $u\in W^1_2(B_r)$ is a weak solution of the equation
$$
-\Delta u+b^i(x)D_iu=0\qquad\mbox{with} \quad 
b^i(x)=\frac{2x_i}{|x|^2\ln(|x|^{-1})}\in L_2(B_r).
$$
However, $u$ attains its minimum at the origin. 
\medskip

Thus, for $n=2$ the condition on $b^i$ must be strengthened. For example, one can estimate the last term in (\ref{eq:Moser_power}) as follows (cf. \cite[Theorem 3.1]{ANPS2021}):
$$
\int\limits_{B_{2R}}|b^iD_iV\,V\zeta^2|\,dx
\leq \|{\bf b}\|_{L_{\Phi_1}(B_{2R})}\|DV\,\zeta\|_{2,B_{2R}}\|V\,\zeta\|_{L_{\Phi_2}(B_{2R})},
$$
where $L_\Phi$ stands for the Orlicz space with the $N$-function $\Phi$ (see, 
e.g., \cite[Section 10]{BIN1996E}), 
$$
\Phi_1(t)=t^2\ln(1+t), \qquad \Phi_2(t)=\exp(t^2)-1.
$$
Using the Yudovich--Pohozhaev embedding theorem ${\stackrel{\circ}W\vphantom 
{W}\!\!^1_2}(\Omega) 
\hookrightarrow L_{\Phi_2}(\Omega)$ (see, e.g., \cite[Subsection 10.6]{BIN1996E}), we obtain the strong maximum principle under the assumption ${\bf b}\ln^{\frac 12}\big(1+|{\bf b}|\big)\in L_2(\Omega)$ which was introduced in \cite{NU2011E}. Under the same assumption, the Harnack inequality also holds (see \cite[Theorem 2.5\,$^\prime$]{NU2011E}). The above example shows that the power $\frac 12$ of the logarithm cannot be reduced.
\end{rem}

Since the second half of the 1960s the number of papers on the Harnack inequality for divergence type equations (even linear ones) has grown rapidly. We will focus on three important directions in the development of this topic.

\paragraph{1. Non-uniformly elliptic operators.} In several papers, operators with the ellipticity condition (\ref{eq:ell-x}) were studied under various assumptions about the functions $\nu(x)$ and ${\cal V}(x)$.
\medskip

N.S.~Trudinger \cite{Tr1971} proved the Harnack inequality for operators ${\mathfrak L}_0$ under the assumption
$$
\nu^{-1}\in L_q(\Omega),\quad \nu^{-1}{\cal V}^2\in L_r(\Omega), \qquad \frac 1q+\frac 1r <\frac 2n.
$$
In \cite{Tr1973}, operators of more general form (\ref{eq:hat-L}) were considered under a weaker condition
\begin{equation}\label{eq:Tru}
\nu^{-1}\in L_q(\Omega),\quad {\cal V}\in L_r(\Omega), \qquad \frac 1q+\frac 1r <\frac 2n;
\end{equation}
the lower-order coefficients were assumed to satisfy some weighted summa\-bil\-i\-ty conditions determined by the matrix ${\cal A}$.\footnote{In the case of a uniformly elliptic operator, these conditions are close to the Stampacchia conditions (\ref{eq:Stamp}).}

Under these conditions, the Harnack inequality was established in \cite{Tr1973}, as well as the strong maximum principle in the following form:
\medskip

\noindent\textit{Let $u$ be a weak supersolution of the equation $\widehat {\mathfrak L}u=0$ in $\Omega$. If ${\mathfrak u}\equiv1$ is also a supersolution, then $u$ cannot attain its negative minimum in $\Omega$ unless $u\equiv const$ (in this case $u$ is a weak solution).}
\medskip

For operators of the simplest form ${\mathfrak L}_0$, the restriction on exponents in (\ref{eq:Tru}) was weakened to $\frac 1q+\frac 1r <\frac 2{n- 1}$ in the recent paper \cite{BSch2021}. On the other hand, an example in the paper \cite{FSSC1998} shows that for $n\ge4$ and $\frac 1q+\frac 1r >\frac 2{n-1}$ the equation ${\mathfrak L}_0u= 0$ in $B_R$ can have a weak solution unbounded in $B_{\frac{R}{2}}$. The question of the validity of the Harnack inequality in the borderline case $\frac 1q+\frac 1r =\frac 2{n-1}$ is still open.
\medskip

In \cite{FKS1982}, operators ${\mathfrak L}_0$ were considered under the following conditions\footnote{These conditions appeared earlier in the paper \cite{EP1972}, devoted to quasilinear equations, but there additional restrictions (\ref{eq:Tru}) were imposed on the functions $\nu(x)$ and ${\cal V}(x)$.}:
\begin{enumerate}
     \item there exists $N\geq1$ such that ${\cal V}(x)\leq N\cdot\nu(x)$ for almost all $x\in\Omega$;
     \item $\nu$ belongs to the Muckenhoupt class $A_2$, i.e.
\begin{equation}\label{eq:Mc2}
\sup\limits_{x\in\R^n,\,r>0}\Big(\!\Xint{\quad -}\limits_{B_r(x)}\nu(y)\,dy \cdot\!\!\Xint{\quad -}\limits_{B_r(x)}\nu^{-1}(y)\,dy\Big)<\infty.
\end{equation}
\end{enumerate}
Under these conditions, the Harnack inequality and the strong maximum principle are proved in \cite{FKS1982}. In addition, a counterexample showing that weakening the condition
$\nu\in A_2$ to $\nu\in \bigcup_{p>2}A_p$ does not ensure the fulfillment of the Harnack inequality\footnote{This counterexample does not violate the strong maximum principle.} is given there.

In \cite{DeCV1996}, the results of \cite{FKS1982} were generalized to operators of the form (\ref{eq:hat-L}). In this case, 
the lower-order coefficients satisfy the following conditions:
\begin{equation}\label{eq:DECV}
\frac {b^i}\nu\in L_m(\Omega),\quad \frac {d^i}\nu\in L_q,\quad \frac c\nu\in L_{\frac q2},\qquad q>m,
\end{equation}
(here $m$ is the exponent called in \cite{DeCV1996} ``intrinsic dimension'' generated by the behavior of the weight $\nu$; for uniformly elliptic operators we have $m=n$, and these conditions become (\ref{eq:Stamp})).

We also mention the papers \cite{CW1986} and \cite{ChRuSe1989}, where the Harnack inequality was established for the operator ${\mathfrak L}_0$ under ``abstract'' conditions on the func\-tions $\nu(x)$ and ${ \cal V}(x)$. Namely, certain weighted Sobolev and Poincaré in\-equal\-i\-ties should be satisfied.

\paragraph{2. Lower-order coefficients from the Kato classes.} Notice that the Lebesgue spaces (as well as the Lorentz and Orlicz spaces) are rearrangement invariant: the norm of a function $f$ in these spaces is determined only by the behavior of the measure of the set $\{x\in\Omega\,\big| \, |f(x)|>N\}$ as $N\to\infty$. A more subtle description of the coefficients singularities can be given in terms of the Kato classes.
\medskip

Recall that a function $f\in L_1(\Omega)$ belongs to the class ${\cal K}_{n,\beta}$, $\beta\in(0,n)$, if  
\begin{equation}\label{eq:omega-r}
\omega_\beta(r):=\sup_{x\in\Omega}\,\int\limits_{\Omega\cap B_r(x)}\frac{|f(y)|}{|x-y|^{n-\beta}}\,dy\to 0 \quad\mbox{as} \quad r\to 0.
\end{equation}
As usual, $f\in{\cal K}_{n,\beta,{\rm loc}}$
means that $f\chi_{\Omega'}\in{\cal K}_{n,\beta}$  for arbitrary subdomain $\Omega'$ such
that $\overline{\Omega}{}'\subset\Omega$.

The functionals $\omega_\beta(r)$ and the spaces defined by them were introduced by M. Schechter in \cite{Sch1968}, \cite[Ch. 5,\S\,1; Ch. 7, \S\,7]{Sch1971} and studied in detail in \cite{Sch1986}\footnote{For particular values of $\beta$, condition (\ref{eq:omega-r}) was used in \cite{Stu1956} and \cite{Kato1972}. In this regard, ${\cal K}_{n,\beta}$ are usually called the Kato or Kato--Stummel classes, which is a typical example of Arnold's principle \cite{Ar1998E}. Some generalizations of the ${\cal K}_{n,\beta}$ classes can be found, for instance, in \cite{ErG2005}.}. For further development of the subject and references see \cite{ZhY2009}.

All the results of this subsection refer to the case $n\geq3$.
\medskip

In the paper \cite{AiS1982}, the Harnack inequality was established for the operator $-\Delta+c(x)$ under the assumption $c\in {\cal K}_{n,2}$. In \cite{ChFG1986} this result was extended to uniformly elliptic operators of the form ${\mathfrak L}_0+c(x)$ under the same condition\footnote{See also \cite{CrFZh1988} and \cite{Si1990} in this connection. \label{foot:51}}.

In the paper \cite{Kur1994}, the Harnack inequality was proved for uniformly elliptic operators of a more general form ${\mathfrak L}+c(x)$ under the assumption\footnote{In an earlier paper \cite{CrZh1987} the operator $-\Delta +b^i(x)D_i$ was considered under stronger restrictions $(b^i)^2\in {\cal K}_{n,2,{\rm loc}}$ and $b^i\in {\cal K}_{n,1,{\rm loc}}$.
 \label{foot:52}} 
\begin{equation} \label{eq:Kato2}
(b^i)^2, c\in {\cal K}_{n,2,{\rm loc}}.
\end{equation}
Finally, the paper \cite{Za2002} combines the two directions described above. Namely, the Harnack inequality is proved for the non-uniformly elliptic operators of the form (\ref{eq:hat-L}). The functions $\nu(x)$ and ${\cal V}(x)$ in (\ref{eq:ell-x}) satisfy the assumptions ${\cal V}(x)\leq N\cdot\nu(x)$ and (\ref{eq:Mc2}), while the functions $(b^i)^2$, $(d^ i)^2$ and $c$ belong to the weighted analogue of the Kato class ${\cal K}_{n,2}$ with an additional constraint\footnote{We assume that this constraint is of a technical nature, but as far as we know, this question remains open.}: the corresponding analog of the function $\omega_2$ from (\ref{eq:omega-r}) admits the estimate $O(r^\gamma)$ for some $\gamma>0$ as $r\to0$.
\medskip

The assumption (\ref{eq:Kato2}) in the general case is very close to optimal. Variations of (\ref{eq:Kato2}) are possible if some additional conditions are imposed on the matrix ${\cal A}$.

The paper \cite{Zh1996} considers a uniformly elliptic operator of the form (\ref{eq:div_operator}) with $a^{ij}\in{\cal C}^{0,\alpha}(\Omega)$, $\alpha\in(0,1)$. This restriction allowed to prove the Harnack inequality under the assumption $b^i\in{\cal K}_{n,1}$.

Note that the H\"older condition on the principal coefficients in \cite{Zh1996} is redundant: using the estimates of the Green's function and its derivatives from \cite{GrW1982}, the same result can be obtained for $a^{ij}\in{\cal C}^ {0,{\cal D}}(\Omega)$.

In the recent paper \cite{KzNa2021}, a case in a sense intermediate has been considered. The principal coefficients of the uniformly elliptic operator ${\mathfrak L}$ in this paper belong to the Sarason space $VMO(\Omega)$. This means that $\omega^{ij}(\rho)\to 0$ as $\rho\to 0$, where
\begin{equation} \label{eq:BMO}
\omega^{ij}(\rho):=\sup_{x\in\Omega}\,\sup_{r\leq\rho}\!\!\! \Xint{\qquad -}\limits_{\Omega\cap B_r(x)}\Big|a^{ij}(y)-\!\!\!\!\!\!\Xint{\qquad -} \limits_{\Omega\cap B_r(x)}\!a^{ij}(z)\,dz\Big|\,dy .
\end{equation}
In this case, the condition $|b^i|^{\beta}\in {\cal K}_{n,\beta}$, $\beta>1$ is imposed on the lower-order coefficients with the additional restriction\footnote{In the case of $n=2$, also studied in \cite{KzNa2021}, the condition (\ref{eq:Kato-beta}) is somewhat modified.}
\begin{equation} \label{eq:Kato-beta}
\sup_{x\in\Omega}\int\limits_{\ \Omega\cap B_r(x)\setminus B_{\frac r2}(x)}\frac{|b^i(y)|^{\beta}}{|x-y|^{n-\beta}}\,dy\leq \sigma^\beta(r), \quad \sigma\in{\cal D}.
\end{equation}
For such operators the strong maximum principle is proved in  \cite{KzNa2021}. Note that the Harnack inequality can also be proved under these assumptions. Whether it is possible to remove or at least relax the restriction (\ref{eq:Kato-beta}) is still unclear.

\paragraph{3. Operators with ${\rm div}({\bf b})\leq0$.} When studying hydrodynamic problems, one often encounters (see, for example, \cite{Zh2004}, \cite{CSYT2008}, \cite{CSTY2009}, \cite{KNSS2009}) the operators $-\Delta+b^i(x)D_i$ (or, more generally, operators of the form (\ref{eq:div_operator})) with the additional structure condition $D_i(b^i)=0$ or $D_i(b^i)\le0$ understood in the sense of distributions. Recall that this means, respectively,
$$
\int\limits_{\Omega}b^iD_i\eta\,dx=0 \quad\mbox{for all}\quad \eta\in {\cal C}^\infty_0(\Omega)
$$
or
$$
\int\limits_{\Omega}b^iD_i\eta\,dx \geq 0 \quad\mbox{for all non-negative}\quad \eta\in {\cal C}^\infty_0(\Omega).
$$
This condition allows to significantly weaken the regularity assumptions for the coefficients $b^i$.
\medskip

In the paper \cite{SSSZ2012}, the Harnack inequality was established for the operator $-\Delta+b^i(x)D_i$ with $D_i(b^i)=0$ under the assumption $b^i\in BMO^{-1}(\Omega)$. It means that
$b^i=D_j(B^{ij})$ in the sense of distributions,
where $B^{ij}\in BMO(\Omega)$, i.e. functions $\omega^{ij}(\rho)$ defined in (\ref{eq:BMO}) (with $B^{ij}$ instead of $a^{ij}$) are bounded\footnote{Obviously, $L_n(\Omega)\subset BMO^{-1}(\Omega)$ due to the embedding $W^1_n(\Omega)\hookrightarrow BMO(\Omega)$. }. If this is true, the relation $D_i(b^i)=0$ is ensured by the additional condition $B^{ij}(x)=-B^{ji}(x)$ for almost all $x\in\Omega$.
\medskip

The paper \cite{NU2011E} studied uniformly elliptic operators of the form (\ref{eq:div_operator}) with $D_i(b^i)\leq0$. The requirements on lower-order coefficients were described in terms of the Morrey spaces.
\medskip

Recall that the space ${\mathbb M}_p^\alpha(\Omega)$, $1\leq p<\infty$, $\alpha\in(0,n)$, consists of functions $f\in L_p( \Omega)$ for which
$$
\|f\|_{\mathbb M^{\alpha}_p(\Omega)}:=
\sup\limits_{B_r(x)\subset\Omega}r^{-\alpha}\|f\|_{p,B_r(x)}<\infty.
$$

In particular, in \cite{NU2011E} the Harnack inequality was proved under the as\-sump\-tion\footnote{Obviously, $L_n(\Omega)\subset \mathbb M^{\frac nq-1}_q(\Omega)$ by the H\"older inequality.} $b^i\in \mathbb M^{\frac nq-1}_q(\Omega)$, $\frac n2<q<n$. N.D. Filonov constructed an extremely subtle counterexample (\cite[Theorem 1.6]{Fi2013E}) showing that even under the assumption $D_i(b^i)=0$ the exponent $\alpha=\frac nq-1$ cannot be reduced.\medskip

The strong maximum principle was established in \cite{NU2011E} for Lipschitz supersolutions\footnote{For weak supersolutions the requirements on $b^i$ in \cite{NU2011E} are somewhat stronger.} under the assumption $b^i\in L_q(\Omega)$, $q>\frac n2$. However, using the approximation (\cite[Theorem 3.1]{FiSh2018}) one can obtain the following partial gen\-er\-al\-iza\-tion of this result:
\medskip

\noindent\textit{Let $\Omega\subset\mathbb R^n$, $n\ge3$. Suppose that the function $u\in W^1_{2,{\rm loc}}(\Omega)$ is a weak solution of the equation $-\Delta u+b^i(x)D_iu=0$ in $\Omega$ , and
$$
D_i(b^i)=0;\quad b^i\in L_q(\Omega); \quad q>\frac n2\ \ \mbox{for}\ \ n\geq4;\quad q=2\ \ \mbox{for}\ \ n=3.
$$
If $u$ attains its minimum at a point $x^0\in\Omega$ then $u\equiv const$.}\medskip

On the other hand, the following counterexample was constructed in \cite{FiSh2018}.\medskip

Let $n\geq4$, and let $u(x)=\ln^{-1}(|x'|^{-1})$. Then $u\in W^1_2(B_r)$ for $r\leq \frac 12$. Further, direct calculation shows that $u$ is a weak solution of the equation $-\Delta u+b^i(x)D_iu=0$ with\footnote{There is a typo in \cite{FiSh2018} in the formula for $b^n$.}

$$
b^i(x)=\begin{cases}
\bigg(\dfrac{n-3}{|x'|}+\dfrac 2{|x'|\ln(|x'|^{-1})}\bigg)\, \dfrac {x_i}{|x'|},
& i<n;\\
\\
-\bigg(\dfrac{(n-3)^2}{|x'|}+\dfrac  {2(n-3)}{|x'|\ln(|x'|^{-1})}+\dfrac  2{|x'|\ln^2(|x'|^{-1})}\bigg)\,\dfrac {x_n}{|x'|},
& i=n.
\end{cases}
$$
It is easy to see that $D_i(b^i)=0$, and $b^i\in L_q(B_r)$ for all $q<\frac{n-1}2$. However, the strong maximum principle does not hold. 

In the recent paper \cite{FiHd2021E} (see also \cite{Hd2021E}), a vector field ${\bf b}\in L_{\frac{n-1}{2}}(B_r)$ with $ D_i(b^i)=0$ is constructed, for which the equation $-\Delta u+b^i(x)D_iu=0$ has a weak solution unbounded in $B_{\frac{r}{2}}$. This can also be considered as violation of the strong maximum principle. The question of the validity of the strong maximum principle for $\frac{n-1}2<q\leq\frac n2$ under the assumption $D_i(b^i)=0$ is open.

\subsection{The normal derivative lemma}

The history of the normal derivative lemma for weak (super)solutions of the equation ${\mathfrak L}u=0$ is rather short. The first result here was obtained by R. Finn and D. Gilbarg in 1957, see \cite{FGi1957}. They considered uniformly elliptic operators of the form (\ref{eq:div_operator}) with $a^{ij}\in{\cal C}^{0,\alpha}(\Omega)$ and $b^i\in{\cal C}(\overline\Omega)$ in a two-dimensional ${\cal C}^{1,\alpha}$-smooth domain, $\alpha\in(0,1)$.

Only in 2015 this result was generalized to the $n$-dimensional case \cite{SdL2015}; the boundary of the domain in this paper was assumed to be smooth\footnote{Also in \cite{SdL2015} some papers with incorrect use of the normal derivative lemma for weak solutions were listed.}. In \cite{KzKu2018E} the normal derivative lemma was proved for all $n \geq3$ under the same conditions on $a^{ij}$ and $\partial\Omega$ as in \cite{FGi1957}, and for $b^i\in L^q(\Omega)$, $q>n$.
\medskip

Back in 1959, a counterexample showing that the requirement for principal coefficients cannot be relaxed to $a^{ij}\in{\cal C}(\overline\Omega)$ was constructed in \cite{Gi1959}.\footnote{It is given in various forms in \cite[Ch. 3]{GTr1983}, \cite[Ch.2]{PS2007}.} Here we give a more general example (see \cite{Na2012}).
\medskip

Let $\Omega$ be a domain in $\R^n$ such that $\Omega\cap \{x_n<h\}={\mathfrak T}(\phi,h)$ and $\phi\in {\cal C}^1$, but $\phi'$ does not satisfy the Dini condition at zero. As mentioned in \S\ref{ss:2.2}, it is shown in \cite{VM1967E} that the normal derivative lemma for the Laplace operator does not hold in such a domain.

Now we flatten the boundary in a neighborhood of the origin. This gives us an operator ${\mathfrak L}_0$ with {\bf\textit{continuous}} principal coefficients for which the normal derivative lemma fails in a {\bf\textit{smooth}} domain.\medskip

This example shows that the natural condition on the principal 
co\-ef\-fi\-cients of the operator is the Dini condition. In this regard, we 
note the work of V.A. Kozlov and V.G. Maz'ya \cite{KzM2003}. In this paper, for 
the operator ${\mathfrak L}_0$ a more subtle condition on the coefficients 
$a^{ij}$ is obtained, which provides the gradient estimate for the solution at 
points of the (smooth) boundary $\partial\Omega$. Apparently, from the 
asymptotics of the solution obtained in \cite{KzM2003}, one can also deduce a 
condition for the fulfillment of the normal derivative lemma, which is more 
precise than the Dini condition\footnote{B. Sirakov informed us in private 
communication that he has proved the normal derivative lemma provided that 
$a^{ij}$ satisfy the {\bf mean-Dini condition}, that is, $\omega^{ij}\in{\cal 
D}$ in (\ref{eq:BMO}). This assumption is stronger than $a^{ij}\in {\cal 
C}(\overline\Omega)$ but weaker that $a^{ij}\in{\cal C}^{0,{\cal D}}(\Omega)$. 
See in this connection \cite{DEK2018}, where ${\cal C}^1$-estimate up to the 
boundary was proved for solutions to the equations under the same assumption.
}.\medskip

To demonstrate the main idea we prove the normal derivative lemma for the simplest operator ${\mathfrak L}_0$ with $a^{ij}\in{\cal C}^{0,{\cal D}}(\Omega)$ \footnote{Obviously, it suffices to fulfill this condition only in some neighborhood of $\partial\Omega$. Apparently, this condition can be kept only on $\partial\Omega$, see \cite{KzM2003}.} under minimal assumptions on the boundary of the domain.

\begin{thm} \label{thm:HO-div}
Let the domain $\Omega\subset\R^n$ satisfy the interior ${\cal C}^{1,{\cal D}}$-paraboloid condition. Suppose that the coefficients of the operator ${\mathfrak L}_0$ satisfy the conditions (\ref{eq:uniell}) and $a^{ij}\in{\cal C}^{0,{\cal D}}( \Omega)$. Let $u\not\equiv const$ be a weak supersolution
of the equation ${\mathfrak L}_0u=0$ in $\Omega$.

If $u$ is continuous in $\overline{\Omega}$ and attains its minimum at $x^0\in \partial\Omega$, then
for any strictly interior direction $\boldsymbol{\ell}$ the inequality
$$
\liminf\limits_{\varepsilon\to+0}\,\frac  {u(x^0+\varepsilon\boldsymbol{\ell})-u(x^0)}\varepsilon\,>0.
$$
holds true.
\end{thm}

\begin{proof}
Without loss of generality, we assume that $x^0=0$ and $\Omega={\mathfrak T}(\phi,h)$, with $\phi\in{\cal C}^{1,{\cal D}}$. Further, the restrictions on $a^{ij}$ are preserved under coordinate transformations of the class $\mathcal{C}^{1,\mathcal{D}}$. Therefore, we can flatten $\partial\Omega$ in a neighborhood of $x^0$ and assume that $B_R \cap \{x_n>0\} \subset \Omega$ for some $R>0$.

For $0<r<R/2$ consider the point $x^r=(0,\dots,0,r)$ and the annulus $\pi=B_r(x^r)\setminus \overline{B}_{\frac r2}(x^r)\subset \Omega$.

The assumption $a^{ij}\in{\cal C}^{0,{\cal D}}(\Omega)$ gives
\begin{equation}\label{eq:Dini} 
|a^{ij}(x)-a^{ij}(y)|\leq \sigma (|x-y|),\quad x,y \in \overline{\pi},\qquad \sigma\in{\cal D}.
\end{equation}

Let $x^*$ be an arbitrary point in $\overline{\pi}$. Following \cite{FGi1957}, we define the barrier function ${\mathfrak V}$ and the auxiliary function $\Psi_{x^*}$ as solutions to the following boundary value problems:
\begin{equation*} 
\left\{ 
\begin{aligned}
{\mathfrak L}_0{\mathfrak V}&=0 \quad \text{in}\ \pi,\\
{\mathfrak V}&=1 \quad \text{on}\ \partial B_{\frac r2}(x^r),\\
{\mathfrak V}&=0 \quad \text{on}\ \partial B_r(x^r),
\end{aligned}
\right.
\qquad \quad
\left\{ 
\begin{aligned}
{\mathfrak L}_0^{x^*}\Psi_{x^*}&=0 \quad \text{in}\ \pi,\\
\Psi_{x^*}&=1 \quad \text{on}\ \partial B_{\frac r2}(x^r),\\
\Psi_{x^*}&=0 \quad \text{on}\ \partial B_r(x^r),
\end{aligned}
\right.
\end{equation*}
where ${\mathfrak L}_0^{x^*}$ is the operator with constant coefficients
$$
{\mathfrak L}_0^{x^*}\Psi_{x^*}:=-D_i(a^{ij}(x^*)D_j\Psi_{x^*}).
$$  

It is well known that $\Psi_{x^*} \in \mathcal{C}^\infty(\overline{\pi})$. Further, the existence of a (unique) weak solution ${\mathfrak V}$ follows from the general linear theory. Moreover, Lemma~3.2 in \cite{GrW1982} shows that ${\mathfrak V}\in {\cal C}^1(\overline{\pi})$, and for $y\in \overline{\pi}$ the following estimate holds:
\begin{equation}
\label{eq:3.11-GW82}
|D{\mathfrak V}(y)|\leq \frac{N_1(n,\nu,\sigma)}r.
\end{equation}

We set ${\mathfrak w}={\mathfrak V}-\Psi_{x^*}$ and notice that ${\mathfrak w}=0$ on $\partial \pi$. Therefore, ${\mathfrak w}$ admits the representation via the Green's function ${\cal G}_{x^*}$ of the operator ${\mathfrak L}_0^{x^*}$ in $\pi$:
\begin{align*}
{\mathfrak w}(x)=\int\limits_{\pi} {\cal G}_{x^*}(x,y)\,{\mathfrak L}_0^{x^*}{\mathfrak w}(y)\,dy 
\stackrel{(\star)}=
\int\limits_{\pi} {\cal G}_{x^*}(x,y) \left({\mathfrak L}_0^{x^*}{\mathfrak V}(y)-{\mathfrak L}_0 {\mathfrak V}(y)\right)\,dy, 
\end{align*}
(the equality $(\star)$ follows from relation ${\mathfrak L}_0^{x^*}\Psi_{x^*}={\mathfrak L}_0 {\mathfrak V}=0$).

Integrating by parts we have
\begin{equation} \label{eq:w-Green}
{\mathfrak w}(x)=\int\limits_{\pi}D_{y_i}{\cal G}_{x^*}(x,y)\left(a^{ij}(x^*)-a^{ij}(y)\right)D_j{\mathfrak V}(y)\,dy.
\end{equation}
Differentiating both parts of the equality (\ref{eq:w-Green}) with respect to $x_k$, we obtain
\begin{equation} \label{eq:Dw-Green}
\begin{gathered}
D_k{\mathfrak w}(x^*)=\int\limits_{\pi}D_{x_k}D_{y_i}{\cal G}_{x^*}(x^*,y)\left(a^{ij}(x^*)-a^{ij}(y)\right)D_j{\mathfrak V}(y)\,dy,\\
k=1,\dots,n.
\end{gathered}
\end{equation}

The derivatives of the Green's function ${\cal G}_{x^*}(x,y)$ can be estimated as follows (see, e.g., \cite[Theorem 3.3]{GrW1982}):
\begin{equation} \label{eq:DGreen}
|D_xD_y{\cal G}_{x^*}(x,y)|\leq \frac{N_2(n,\nu)}{|x-y|^n} \,, \qquad x,y \in 
\overline{\pi}.
\end{equation}

The substitution of (\ref{eq:3.11-GW82}), (\ref{eq:DGreen}) and (\ref{eq:Dini}) into (\ref{eq:Dw-Green}) gives
$$
|D{\mathfrak w}(x^*)|\leq \frac{N_1N_2}r \int\limits_{B_{2r}(x^*)}\frac {\sigma (|x^*-y|)}{|x^*-y|^n}\,dy,
$$
and we arrive at
\begin{equation} \label{eq:grad-z-psi}
|D{\mathfrak V}(x^*)-D\Psi_{x^*} (x^*)| \leq \frac{N_3(n,\nu,\sigma)}r \int\limits_0^{2r} \frac{\sigma (\tau)}{\tau}\,d\tau,\qquad x^*\in\overline{\pi}.
\end{equation}

Since the normal derivative lemma holds for operators with constant coefficients, we obtain for any strictly interior direction $\boldsymbol{\ell}$
$$
\partial_{\boldsymbol{\ell}}\Psi_0(0)\geq \frac{N_4(n, \nu,\boldsymbol{\ell})}r>0.
$$
By (\ref{eq:grad-z-psi}), for sufficiently small $r>0$ we have
$$
\partial_{\boldsymbol{\ell}}{\mathfrak V}(0) \geq \partial_{\boldsymbol{\ell}}\Psi_0(0)-|D{\mathfrak V}(0)-D\Psi_0(0)|\geq
\frac{N_4}r-\frac{N_3}r\int\limits_0^{2r} \frac{\sigma (\tau)}{\tau}\,d\tau\geq \frac{N_4}{2r}. 
$$

We fix such $r$. Since $u\not\equiv const$, the strong maximum principle yields $u-u(0)>0$ on $\partial B_{\frac r2}(x^r)$. Therefore, for sufficiently small $\varkappa>0$
$$
{\mathfrak L}_0(u-u(0)-\varkappa {\mathfrak V})\geq 0\quad \mbox{in} \ \ \pi;\qquad u-u(0)-\varkappa {\mathfrak V}\geq 0 \quad \mbox{on} \ \ \partial \pi.
$$
Now the weak maximum principle gives $u-u(0) \geq \varkappa {\mathfrak V}$ in $\pi$. Since at the origin this inequality becomes equality, we have
$$
\liminf\limits_{\varepsilon\to+0}\,\frac  {u(\varepsilon\boldsymbol{\ell})-u(0)}\varepsilon\,\geq\varkappa \partial_{\boldsymbol{\ell}}{\mathfrak V}(0),
$$
and the statement follows.
\end{proof}

Now we formulate a more general result established in \cite{AN2019}. The 
conditions on the lower-order coefficients for the validity of the normal 
derivative lemma obtained in this paper are the most precise at the moment.

\begin{thm} \label{thm:AN19}
Let the domain $\Omega\subset\R^n$ and the principal coefficients of the operator ${\mathfrak L}$ satisfy the conditions of Theorem \ref{thm:HO-div}. Let us also assume that
\begin{equation}\label{eq:weak-Kato}
\sup_{x\in\Omega}\int\limits_{\ \Omega\cap B_r(x)}\frac{|{\bf b}(y)|}{|x-y|^{n-1}}\cdot\frac{{\rm d}(y)}{{\rm d}(y)+|x-y|}\,dy\to 0 \quad\mbox{as} \quad r\to 0. 
\end{equation}
Let $u\in W^1_2(\Omega)$ be a weak supersolution of the equation ${\mathfrak L}u=0$, let $u\not\equiv const$ in $\Omega$, and let $b^iD_iu\in L_1(\Omega)$. Then the conclusion of Theorem \ref{thm:HO-div} holds true.
\end{thm}

\begin{rem}
In any subdomain of $\Omega'$ such that $\overline{\Omega}{}'\subset\Omega$, the condition (\ref{eq:weak-Kato}) coincides with $b^i\in{\cal K}_{n,1}$, cf. (\ref{eq:Kato-beta}). Therefore, (\ref{eq:weak-Kato}) implies $b^i\in{\cal K}_{n,1,{\rm loc}}$. On the other hand, it is shown in \cite{AN2019} that the assumptions on $b^i$ imposed in Theorem \ref{thm:Na12} imply (\ref{eq:weak-Kato}).
\end{rem}

\begin{rem}
The normal derivative lemma for divergence type operators is directly related to the properties of the Green's functions for these operators.

The Green's function for a uniformly elliptic operator ${\mathfrak L}_0$ with measurable coefficients was first constructed in the seminal paper \cite{LtStW1963E}. Among other results of this work, we notice the estimate\footnote{Later this estimate was extended to more general operators of the form (\ref{eq:div_operator}). Recent results in this area, as well as a historical survey, can be found in \cite{AlkSur2022}.} 
$$
\frac {C^{-1}}{|x-y|^{n-2}}\leq{\cal G}(x,y)\leq \frac C{|x-y|^{n-2}},
$$
which holds for the Green's function in the whole $\R^n$, $n\geq3$ (here $C$ depends only on $n$ and $\nu$).

A very important role belongs also to the article \cite{GrW1982}, where, among other results, the following estimates were proved for the Green's function of a uniformly elliptic operator ${\mathfrak L}_0$ with coefficients satisfying the Dini condition, in the domain $\Omega\subset\R^n$, $n\geq3$, satisfying the exterior ball condition:
$$
\aligned
{\cal G}(x,y)\leq &\ \frac C{|x-y|^{n-2}}\cdot\frac{{\rm d}(x)}{{\rm d}(x)+|x-y|}\cdot\frac{{\rm d}(y)}{{\rm d}(y)+|x-y|};\\
|D_x{\cal G}(x,y)|\leq &\ \frac C{|x-y|^{n-1}}\cdot\frac{{\rm d}(y)}{{\rm d}(y)+|x-y|};\\
|D_xD_y{\cal G}(x,y)|\leq &\ \frac C{|x-y|^n}
\endaligned
$$
(the constant $C$ depends on $n$, $\nu$, on the function $\sigma$ in the Dini condition for the coefficients and on the domain $\Omega$).
\medskip

Thus, the assumption (\ref{eq:weak-Kato}), roughly speaking, means that the function $|{\bf b}(y)|\cdot |D_x{\cal G}(x,y)|$ is integrable uniformly with respect to $x$.
\end{rem}

\section{Some generalizations and applications}\label{sec:appl}

As already mentioned in the Introduction, in this Section we provide a brief presentation of some subjects that either generalize the main assertions of our survey or directly rely on them.

\subsection{Symmetry of solutions to nonlinear boundary value problems}

We start with the celebrated {\bf moving plane method}. It was first applied by A.D. Aleksandrov \cite{Al1958cE} to the problem of characterizing a sphere by the property of constancy of its mean curvature (or some other functions of the principal curvatures)\footnote{The problem statement and the history of the problem are given in \cite{Al1956E}; see also \cite{Al1958bE}. For generalizations of this result see, e.g., \cite{Lyy1997}--\nocite{LNi2006}\cite{LNi2006a}.}.
The method was later rediscovered by J. Serrin \cite{S1971}
when solving the following overdetermined problem in the unknown ${\cal C}^2$-smooth domain: 
$$
-\Delta u=1 \quad \mbox{in} \quad\Omega,\qquad u\big|_{\partial\Omega}=0, \qquad \partial_{\bf n}u\big|_{\partial\Omega}=const.
$$
It is shown in \cite{S1971} that such a problem is solvable only if $\Omega$ is a ball.

The method owes its popularity to the article \cite{GNN1979}, which considered the problem
\begin{equation}
\label{eq:GNN}
-\Delta u=f(u) \quad \mbox{in} \quad B_R,\qquad u\big|_{\partial B_R}=0
\end{equation}
and its generalizations. Let us formulate the basic result of this work.

\begin{thm}\label{thm:GNN}
Let $f\in{\cal C}^1_{\rm loc}(\R_+)$, and let $u\in{\cal C}^2(\overline{B}_R)$ be a positive in $B_R$ solution to the problem (\ref{eq:GNN}). Then  $u=u(r)$ (the function $u$ is radially symmetric) and $u'(r)<0$ for $0<r<R$.
\end{thm}

Let us sketch the proof of Theorem \ref{thm:GNN}.
Obviously, it suffices to show that $u$ is an even function of $x_n$ and $D_nu(x)<0$ for $x_n>0$.

For $0<\lambda<R$, denote by $\Sigma_\lambda$ the segment cut off from the ball by the plane $\Pi_\lambda=\{x\,\big|\,x_n=\lambda\}$. For $x\in\overline{\Sigma}_\lambda$ we denote by $\widehat x_\lambda=(x',2\lambda-x_n)$ the point symmetric to $x$ with respect to $\Pi_\lambda$.

Consider the function $v_\lambda(x)=u(\widehat x_\lambda)-u(x)$ in $\overline{\Sigma}_\lambda$. It satisfies the equation
$$
-\Delta v_\lambda + c(x)v_\lambda=0;\qquad c(x)=\frac {f(u(\widehat x_\lambda))-f(u(x))}{u(x)-u(\widehat x_\lambda)}\in L_\infty(\Sigma_\lambda).
$$
For $\lambda$ sufficiently close to 1, the function $v_\lambda$ is positive in $\Sigma_\lambda$ (the graph of the ``reflected'' function lies above the original one) and attains {\bf zero} minimum at $\Pi_\lambda$. By the normal derivative lemma (item {\it B1} of Theorem \ref{thm:HO}), we have $\partial_{\bf n} v_\lambda(x)=2D_nu(x)<0$ on $\Pi_\lambda$.\footnote{Recall that the sign of the coefficient $c(x)$ is not important here. Note also that if $f(0)<0$, then $D_nu$ can vanish at $x\in\Pi_\lambda\cap\partial B_R$, but it is shown in \cite{GNN1979} that in this case $D_nD_nu(x)>0$; this is sufficient for the subsequent argument.} Therefore, one can slightly reduce $\lambda$ (shift the plane $\Pi_\lambda$ to the center of the ball) such that the inequality $v_\lambda>0$ in $\Sigma_\lambda$ will still be satisfied.

Denote by $\lambda_0$ the greatest lower boundary of those $\lambda$ for which $v_\lambda>0$ in $\Sigma_\lambda$. If we assume that $\lambda_0>0$ then $v_{\lambda_0}>0$ on the ``circular'' part of $\partial\Sigma_{\lambda_0}$. By the strong maximum principle (item {\it A1} of Theorem \ref{thm:HO}), $v_{\lambda_0}>0$ in $\Sigma_{\lambda_0}$. But then we can repeat the previous argument and obtain that the plane $\Pi_{\lambda_0}$ can be shifted a little more towards the center, which is impossible. Thus, $\lambda_0=0$, and $v_0\equiv0$, i.e. $u(x',-x_n)\equiv u(x)$. The theorem is proved.
\medskip

As pointed out in \cite{GNN1979}, if $f(0)\geq0$ then the a priori positivity of $u$ can be replaced by the assumption $u\geq0$, $u\not\equiv0$. It is also obvious that the condition $f\in{\cal C}^1_{\rm loc}(\R_+)$ can be replaced by the local Lipschitz condition. The following example given in \cite{GNN1979} shows that the H\"older condition on $f$ is, in general, not sufficient.
\medskip

Let $p>2$, and let $u(x)=(1-|x-x^0|^2)_+^p$. Direct calculation shows that $u$ is a solution to the problem (\ref{eq:GNN}) for $R>|x^0|+1$ if we put\footnote{There is a typo in \cite{GNN1979} in this formula.}
$$
f(u)=2p(n-2+2p)u^{1-\frac 1p}-4p(p-1)u^{1-\frac 2p} \in{\cal C}^{0,1 -\frac 2p}_{\rm loc}(\R_+).
$$
The H\"older exponent can be made arbitrarily close to $1$ by choosing $p$, but the assertion of the theorem does not hold\footnote{Nevertheless, if $f>0$ then the Lipschitz condition on $f$ can be weakened, see, e.g., \cite{Ls1981}
and \cite{GKPR1998}.}.
\medskip

The article \cite{GNN1979} (as well as \cite{GNN1981}, where equations of the form (\ref{eq:GNN}) were considered in the whole space) gave rise to a huge number of improvements and generalizations. Among them, we lay emphasis on the paper by H.~Berestycki and L. Nirenberg \cite{BeNi1991}. There, by using the Aleksandrov--Bakelman max\-i\-mum principle, the results of \cite{GNN1979} are extended to strong solutions for a rather wide class of uniformly elliptic nonlinear equations. Applications of the moving plane method to the $p$-Laplacian type degenerate operators can be found in \cite{DPR2003}, \cite{DSc2004}, \cite{EspSc2020}, \cite{OSV2020}, see also references therein.\medskip

A number of papers use the {\bf moving sphere method}, that is a combination of the moving plane method with conformal transformations (see, e.g., \cite{LiZh1995} and \cite{Pd1997}). We also mention the paper \cite{MKR2007}, where discrete analogues of the results in \cite{GNN1979} were obtained.
\medskip

Other applications of the strong maximum principle and the normal derivative lemma to the proof of symmetry properties in geometric problems can be found, for example, in \cite{Ae1960}, \cite{S1969}, \cite{PaSt1973}, \cite{TW1983}, \cite{ Ol1984}, \cite{Mu2011} (see also \cite{Sp1981}). The applications of the Aleksandrov--Bakelman maximum principle and its variants to the study of symmetry properties for solutions to nonlinear boundary value problems and to the proof of isoperimetric inequalities are discussed in \cite{Cb2000}, \cite{Cb2008}, \cite{CbROS2016}, see also the survey \cite{Cb2002}.

\subsection{Phragm\'en--Lindel\"of type theorems}\label{ss:4.2}

The Phragm\'en--Lindel\"of principle in its original formulation \cite{PhL1908} describes the behavior at infinity of a function analytic in an unbounded domain.

For solutions of uniformly elliptic (non-divergence type) equations of general form, such theorems were first proved by E.M. Landis \cite{La1959E}--\cite{La1963E} (brief reports were previously published in \cite{La1956E}--\nocite{La1956bE}\cite{La1956aE}). The principal co\-ef\-fi\-cients of the operator in \cite{La1963E} satisfy the Dini condition, and the behavior of the domain at infinity is described in terms of measure.

More exact Phragm\'en--Lindel\"of type theorems can be obtained if the domains are described in terms of capacity. The first results of this kind were established in \cite{Mz1967E}, \cite{Bl1970E} for divergence type equations with measurable principal co\-ef\-fi\-cients and in \cite{Bl1965E}, \cite{Bl1970E} for non-divergence type equations under the H\"older condition on the principal coefficients.

Finally, the crucial step was taken by E.M. Landis \cite{La1968E} (see also \cite[Ch.~1]{La1971E}), who, using the concept of $s$-capacity introduced by himself, proved the Phragm\'en--Lindel\"of type theorems for non-divergence type equations with measurable principal coefficients.

We present, for instance, one of the results from \cite[Ch. 1, \S\,6]{La1971E}.

\begin{thm}
\label{thm:Fr-Lin}
Let $\Omega$ be an unbounded domain lying inside the infinite layer
$$
\Omega\subset \{x\in\R^n \,\big|\, |x_n|<h\}.
$$
Let an operator ${\cal L}_0$ satisfy the condition (\ref{eq:uniell}), and let $u\in {\cal C}^2(\Omega)$ be a classical subsolution\footnote{Using the Aleksandrov--Bakelman maximum principle, this result can be adapted for strong subsolutions $u\in W^2_{n,{\rm loc}}(\Omega)$.} of the equation ${\cal L}_0u=0$ satisfying the condition $u |_{\partial\Omega}\leq0$.

If $u(x)>0$ at some point $x\in\Omega$, then
$$
\liminf\limits_{R\to\infty}\frac {\max\limits_{|x|=R} u(x)}{\exp\big(\frac Ch R\big)}>0,
$$
where the constant $C>0$ depends only on $n$ and $\nu$.
\end{thm}

We also note the work of V.G. Maz'ya \cite{Mz1970E}, where related questions were studied for quasilinear $p$-Laplacian type operators.

In the case where the non-tangential derivative is given on a part of $\partial\Omega$, the Phragm\'en--Lindel\"of type theorems were proved in \cite{LaIb1995E}, \cite{IbLa1996E}, \cite{IbLa1997} for di\-ver\-gence type equations and in \cite{CIN2018} (see also \cite{IbNa2017E}) for non-divergence type equations. Note that in the last two papers, a weakened form of the normal derivative lemma \cite{Nd1983E} was used.\medskip

The above results are linked to the {\bf Landis conjecture} that is the problem of the fastest possible rate of convergence to zero for a nontrivial solution to a uniformly elliptic equation
in the domain $\Omega=\R^n\setminus B_R$. It was first formulated in the survey \cite{KoLa1988E} for the equation
\begin{equation}\label{eq:Landconj}
    -\Delta u +c(x)u=0
\end{equation} 
with $c\in L_\infty(\Omega)$ (in this case, the expected answer is exponential 
decay: if $|u(x)|=O(\exp(-N|x|))$ as $|x|\to\infty$ for any $N>0$, then 
$u\equiv0$). This problem has not been completely solved even for the simplest 
equation (\ref{eq:Landconj}). Recent results in this area, as well as a 
historical survey, can be found in \cite{SSu2021} (see also \cite{LMNN2020}).

\subsection{Boundary Harnack inequality}\label{ss:4.3}

If the normal derivative lemma does not hold, the following statement can be considered as its weaker version:
\medskip

\noindent {\bf Boundary Harnack inequality.}
\textit{Let $0\in\Omega$, and let ${\mathbb L}$ be an elliptic operator in $\Omega$. If $u_1$ and $u_2$ are positive solutions of the equation ${\mathbb L}u=0$ in $\Omega$ satisfying the condition $u_1|_{\partial\Omega\cap B_R}=u_2|_{\partial\Omega\cap B_R}=0$, then the inequality
\begin{equation}
C^{-1}\,\frac {u_1(0)}{u_2(0)}\leq\frac {u_1(x)}{u_2(x)} \leq C\,\frac {u_1(0)}{u_2(0)}
\label{eq:BoundHarnack}
\end{equation}
holds true in the subdomain $\Omega\cap B_{R/2}$, where $C$ is a constant independent on $u_1$ and $u_2$.
}

\begin{rem}
\label{rem:BoundHarnack}
If, for example, $\Omega$ is a ${\cal C}^{1, {\cal D}}$-smooth domain, and ${\cal L}$ is a uni\-form\-ly elliptic operator of the form (\ref{eq:nondiv_operator}) with bounded coefficients, then (\ref{eq:BoundHarnack}) easily follows from the normal derivative lemma, the gradient estimate for solutions on $\partial\Omega$, and the usual Harnack inequality.
\end{rem}

\begin{rem}
\label{rem:Krylov}
In the important particular case of flat boundary $x_n=0$ and an operator ${\cal L}_0$, where $u_2(x)=x_n$ can be taken, the boundary Harnack inequality was first obtained by N.V. Krylov \cite{Kr1983E} in order to obtain boundary estimates in ${\cal C}^{2, \alpha}$ for solutions of {\bf\textit{nonlinear}} equations.
\end{rem}

To describe the results of this subsection, we need new classes of domains:
\begin{itemize}
\item Nontangentially accessible domains (NTA domains);
\item Uniform domains;
\item Domains satisfying the $\lambda$-John condition, $\lambda\geq1$; when $\lambda=1$ just say ``John domains'';
\item Twisted H\"older domains (THD domains), with clarification ``of order $\alpha\in(0,1]$'' (THD-$\alpha$) if necessary.
\end{itemize}
The exact definitions of these classes can be found in the corresponding works listed in Table~\ref{tab1}. For the reader's convenience, we present only the relations between them (see, for example, \cite{KiSa2011aE})\footnote{The relation ($\bigtriangleup$) is not stated explicitly in \cite{KiSa2011aE} but follows from Remark 2.5 in this paper.}:
$$
\aligned
{\cal C}^{0,1}\subset\mbox{NTA}\subset\mbox{Uniform}\subset\mbox{John}= &\ \mbox{THD-$1$};\\
{\cal C}^{0,\alpha}\subset \tfrac 1\alpha\mbox{-John} \stackrel{(\bigtriangleup)}{=} &\ \mbox{THD-$\alpha$}.
\endaligned
$$

In Table~\ref{tab1} it is assumed by default that the principal coefficients of the operators are measurable and satisfy the condition (\ref{eq:uniell}).

\begin{table}[ht]    
\centering
\begin{tabular} {|l|c|c|c|c|c|c|}
\hline \hfil Operator & ${\cal C}^{0,1}\rule[-2mm]{0mm}{6mm}$ & NTA & Unif. & John & ${\cal C}^{0,\alpha}$ & THD\footnotemark  \\
\hline $-\Delta\rule[-2mm]{0mm}{6mm}$ & \cite{Dh1977}\footnotemark[74]  & \cite{JK1982}  & \cite{Aw2001}  &   &   & \\
\hline $\hphantom{-}{\cal L}+c(x)\rule[-2mm]{0mm}{6mm}$ \footnotemark[75] & \cite{An1978}\hphantom{\footnotemark[75a]}  &   &   &   &   & \\
\hline $\hphantom{-}{\mathfrak L}_0\rule[-2mm]{0mm}{6mm}$ & \cite{CFMS1981}\hphantom{\footnotemark[75a]}  &   &   &   & \cite{BaBaBu1991}\footnotemark[76] & \cite{BaBu1991} \\
\hline $-\Delta+b^i(x)D_i\rule[-2mm]{0mm}{6mm}$ \footnotemark[77] & \cite{CrZh1987}\hphantom{\footnotemark[77]}  &   &   &   &   & \\
\hline $\hphantom{-}{\cal L}_0\rule[-2mm]{0mm}{6mm}$ & \cite{FGMMS1988}\footnotemark[78]  &   &   &   &   & \\
\hline $\hphantom{-}{\mathfrak L}_0+c(x)$, \ $c\in {\cal K}_{n,2}\rule[-2mm]{0mm}{6mm}$ & \cite{CrFZh1988}\footnotemark[79]  &   &   &   &   & \\
\hline $\hphantom{-}\widehat{\mathfrak L}\rule[-1.5mm]{0mm}{6mm}$ & \cite{DeCV1996}\footnotemark[80]  &   &   &   &   & \\
\hline $\hphantom{-}{\cal L}$,\qquad $b^i\in L_\infty(\Omega)\rule[-2mm]{0mm}{6mm}$ &   &   &   &   & \cite{BaBu1994}\footnotemark[81]  & \\
\hline $\hphantom{-}{\cal L}$,\qquad $b^i\in L_n(\Omega)\rule[-2mm]{0mm}{6mm}$ & \cite{Sa2010}\hphantom{\footnotemark[81]}  &   &   & \cite{KiSa2014} &   & \cite{KiSa2011aE} \\
\hline\end{tabular}
    \caption{Boundary Harnack inequality in various classes of domains}
    \label{tab1}
\end{table}
\footnotetext{Results are obtained for $\alpha>\frac 12$; counterexamples are constructed in \cite{BaBu1991} for $\alpha<\frac 12$ and in \cite{KiSa2011E} for $\alpha=\frac 12$.}
\addtocounter{footnote}{1}
\footnotetext{See also \cite{Wu1978}.}
\addtocounter{footnote}{1}
\footnotetext{The coefficients satisfy the H\"older condition.}
\addtocounter{footnote}{1}
\footnotetext{See also \cite{Fe1998}.}
\addtocounter{footnote}{1}
\footnotetext{For the assumptions on the coefficients $b^i$, see Footnote \ref{foot:52}.}
\addtocounter{footnote}{1}
\footnotetext{See also \cite{Bau1984}; a somewhat more general condition on the domain is considered in \cite{Aw2008}.}
\addtocounter{footnote}{1}
\footnotetext{See Footnote \ref{foot:51}.}
\addtocounter{footnote}{1}
\footnotetext{The principal coefficients satisfy the assumptions (\ref{eq:ell-x}), ${\cal V}(x)\leq N\cdot\nu(x)$ and (\ref{eq:Mc2}), and the lower-order ones are subject to condition (\ref{eq:DECV}). }
\addtocounter{footnote}{1}
\footnotetext{The result is obtained for $\alpha>\frac 12$, while for $\alpha<\frac 12$ a counterexample is constructed. If $\partial\Omega$ additionally satisfies the condition $(A)$ introduced by O.A. Ladyzhenskaya and N.N. Uraltseva (see, for instance, \cite{LU1986E}), then the result is obtained for all $\alpha>0$.}

In the recent papers \cite{DeSS2020}, \cite{DeSS2022},
a unified approach to the proof of the boundary Harnack inequality for divergence and non-divergence types op\-er\-a\-tors is demonstrated\footnote{Earlier, similar ideas appeared in the works of M.V. Safonov, see \cite{FrSa2001}, \cite{Sa2010}, \cite{KiSa2011aE}.}.

A variation of the boundary Harnack inequality for supersolutions and ``almost supersolutions'' of the equation ${\mathfrak L}u+cu=0$ with bounded coefficients was obtained in \cite{AllSh2019}\footnote{See in this connection \cite{Sw1997}, where an estimate is established for a superharmonic function satisfying the zero Dirichlet condition in a two-dimensional domain with corners in terms of the first eigenfunction of the Dirichlet Laplacian.}.
\medskip

The boundary Harnack inequality is linked to results similar to the weak Harnack inequality for the quotient $\dfrac {u(x)}{{\rm d}(x)}$ (see \cite{S2022} and references therein). Let us give, for example, one of the results of \cite{S2022}.

\begin{thm} \label{thm:sirakov}
Let $u$ be a non-negative weak supersolution of the equation\footnote{It is important that no condition is imposed on the behavior of $u$ near $\partial\Omega$.}
${\mathfrak L}u=f$ in a ${\cal C}^{1,1}$-smooth domain. Assume that the condition (\ref{eq:uniell}) is satisfied as well as the following assumptions:
$$
a^{ij}\in W^1_q(\Omega), \quad b^i\in L_q(\Omega), \quad f_-\in L_q(\Omega);\qquad q>n.
$$
Then 
$$
\bigg(\int\limits_{\Omega}\Big(\frac{u(x)}{{\rm d}(x)}\Big)^s dx\bigg)^{\frac 1s}\leq C \bigg(\inf\limits_{x\in\Omega}\,\frac{u(x)}{{\rm d}(x)}+\|f_-\|_{q,\Omega}\bigg)
$$
for any $s<1$. The constant $C$ depends on $n$, $\nu$, $s$, $q$, on the norms of coefficients $a^{ij}$ and $b^i$ in corresponding spaces, on $\mbox {\rm diam}(\Omega)$, and on the prop\-er\-ties of $\partial\Omega$.
\end{thm}

\noindent
The harmonic function $x_n\cdot|x|^{-n}$ in the half-ball $B_r^+=B_r\cap\{x_n>0\}$ shows that the constraint $s<1$ is sharp.
\medskip

We also mention some papers (see, e.g., \cite{LlSC2014}, \cite{Ll2015}, \cite{BlM2019}) where the boundary Harnack inequality was obtained in the ``abstract'' context of metric spaces.

\subsection{Other results for linear operators}

In the papers \cite{An1979}, \cite{BrP2003}, a generalized strong maximum principle is established for the operators $-\Delta+c(x)$ with $c\in L_1(\Omega)$; solutions are understood in the sense of measures. For further results in this direction see \cite{BST2015}, \cite{OPo2016}, \cite{PoW2020}.
\medskip

It is well known that the validity of the weak maximum principle for a second-order elliptic operator is equivalent to the positivity of the first eigenvalue for the corresponding Dirichlet problem. In the paper \cite{BeNiV1994}, a generalized first eigenvalue is defined for uniformly elliptic operators ${\cal L}+c(x)$ with bounded coefficients in an arbitrary bounded domain\footnote{For operators with smooth coefficients in smooth domains, this formula actually gives the first eigenvalue; in this case, the supremum can be taken over smooth functions $\phi$ positive in $\Omega$. For the Laplacian, the formula (\ref{eq:eigen}) was apparently first highlighted in \cite{Ba1937}. Then it was generalized to various classes of operators (see \cite{BeNiV1994} and references therein).}
(the supremum is taken over $\phi\in W^2_{n,{\rm loc}}(\Omega)$, $\phi>0$ in $\Omega$)

\begin{equation} \label{eq:eigen}
\lambda_1=\sup\limits_{\phi}\inf\limits_{x\in\Omega}\,\frac {{\cal L}\phi(x)+c(x)\phi(x)}{\phi(x)}.
\end{equation}
It is shown in \cite{BeNiV1994} that the weak maximum principle (as well as the ``improved'' weak maximum principle introduced in this article) for the operator ${\cal L}+c(x)$ is equivalent to the inequality $\lambda_1>0$.
\medskip

In recent decades the study of partial differential equations on complicated structures has become very popular. In a number of papers (see, e.g., \cite{GoN2013}, \cite{DFFZ2015}, \cite{BBB2016}, \cite{MT2016} and references therein), conditions for the validity of the strong maximum principle, the Harnack inequality, the normal derivative lemma, and the boundary Harnack inequality were studied for {\bf subelliptic} operators, including sub-Laplacians on homogeneous Carnot groups.
\medskip

In the papers \cite{GaP2000E}, \cite{OPS2012E}, \cite{OPS2014E} the strong maximum principle and the normal derivative lemma were considered for the simplest elliptic operators on {\bf stratified sets}, which are cell complexes with some special properties\footnote{The simplest examples of such operators are the operators of the Venttsel problem and the two-phase Venttsel problem.}.

\subsection{Nonlinear operators}

Even the simplest keyword search shows that in recent years the number of articles on the topic of the survey, concerning nonlinear operators, can been estimated at dozens per year. Therefore, this subsection has an obviously dotted character without even a minimum completeness.\medskip

The Harnack inequality for divergence form quasilinear operators was first proved in \cite{S1964} and then for wider classes of operators in \cite{Tr1967} and \cite{Tr1981}. These works are now classics. We also note the paper \cite{DiBTr1984}, where the Harnack inequality was established for {\bf\textit{quasi-minimizers}} of variational problems.
\medskip

In the paper \cite{DV1972}, the normal derivative lemma from \cite{V1963}, \cite{V1964} was generalized to the quasilinear case.

For the operators similar to $p$-Laplacian 
\begin{equation} \label{eq:p-lap}
\Delta_pu\equiv D_i(|Du|^{p-2}D_iu), \qquad p>1 
\end{equation}
the normal derivative lemma was first proved in \cite{To1983}. Among the recent gen\-er\-al\-iza\-tions of this result, we mention the papers \cite{MS2015} and \cite{CaRSc2019}.

In the papers \cite{Cel2002} and \cite{BCM2006}, sharp conditions for the strong maximum principle and the normal derivative lemma were obtained for the minimizers of the functional
$$
J[u]=\int\limits_\Omega f(Du)\,dx.
$$

J.L.~V\'azquez \cite{Vz1984} proved the strong maximum principle for the equation
\begin{equation} \label{eq:Vaz}
-\Delta_pu+f(u)=0  \quad \mbox{in} \quad \Omega \subset \R^n, \qquad n \geq 2.
\end{equation}

\begin{thm}\label{thm:Vaz}
Let $f\in{\cal C}(\R_+)$ be a nondecreasing function, and let $f(0)=0$. Then a necessary and sufficient condition providing that arbitrary (nonzero) nonnegative supersolution of the equation (\ref{eq:Vaz}) does not vanish in $\Omega$ is the relation
\begin{equation} \label{eq:Vaz1}
\int\limits_0^\delta \frac {dt}{\big(F(t)\big)^{\frac 1p}}=\infty, \qquad\mbox{where}\quad F(t)=\int\limits_0^t f(s)\,ds.
\end{equation}
\end{thm}
\noindent For generalizations of this result to wider classes of quasilinear operators see \cite{PS2000}, \cite{FQ2002}, \cite{SSu2021}. In the paper \cite{Ju2016}, the corresponding Harnack inequality is established (for $p=2$):

\begin{thm}\label{thm:Vaz1}
Suppose that $f: \R_+\to\R_+$ is a non-decreasing function.
Let $u\in W^1_2(B_{2R})$ be a solution to the equation ${\mathfrak 
L}_0u+f(u)=0$, where ${\mathfrak L}_0$ is a (divergence type) uniformly elliptic 
operator with measurable coefficients. Denote $M=\sup\limits_{B_R}u$ 
and $m=\inf\limits_{B_R}u$. Then\footnote{Notice that for $f\equiv0$ 
this relation becomes the classical Harnack inequality.}
$$
\int\limits_m^M \frac {dt}{\big(F(t)\big)^{\frac 12} +t}\leq C,
$$
where $F$ is defined in (\ref{eq:Vaz1}), and the constant $C$ depends only on $n$ and $\nu$ (in particular, it does not depend on $f$!).
\end{thm}

The boundary Harnack inequality for operators of $p$-Laplacian type in a ${\cal C}^2$-smooth domain was established in \cite{BVBV2006}. Subsequently, it was proved for the wider classes of domains discussed in \S\, \ref{ss:4.3} (see \cite{Ny2013} and references therein). The boundary Harnack inequality for the maximal and minimal Pucci operators was proved in \cite{S2018} (see also \cite{BgM2018}).
\medskip

Nowadays, popular objects of research are also $p(x)$-Laplacians, i.e. operators of the form (\ref{eq:p-lap}), where the exponent $p$ is a function of the $x$ variables. The Harnack inequality for such operators was first proved in \cite{Alk1997E} (for recent generalizations see, e.g., \cite{AlkSur2019E} and \cite{SkVo2021}). In \cite{AdLu2016}, the boundary Harnack in\-equal\-i\-ty was established in a ${\cal C}^{1,1}$-smooth domain.

\subsection{Nonlocal operators}

In recent decades, interest in the study of nonlocal (integro-differential) operators has increased significantly. Among them, {\bf fractional Laplacians} show up. The simplest of these (and historically the first), the fractional Laplacian in $\R^n$ of order $s$, is defined using the Fourier transform\footnote{To define accurately this and similar operators, as well as the notion of a weak (sub/super)solution to the corresponding equations, it would be necessary to introduce the Sobolev--Slobodetskii spaces (\cite[Ch. 2--4]{Tb1978}; see also \cite{DNPV2012}). We will not do this, pitying the reader.}:
$$
(-\Delta)^s u = {\cal F}^{-1} \big(|\xi|^{2s}({\cal F}u)(\xi)\big), \qquad s>0;
$$
for $s\in(0,1)$ this operator can be defined via a hypersingular integral:
$$
\big((-\Delta)^s u\big)(x)=C_{n,s}\cdot{\rm P.\!V.} \int\limits_{\R^n} \frac{u(x)-u(y)}{|x-y|^{n+2s}}~dy, \qquad C_{n,s}=\frac{s2^{2s}\Gamma(\frac{n}{2}+s)}{\pi^{\frac{n}{2}}\Gamma(1-s)}.
$$

M. Riesz \cite{Ri1938} (see also \cite[Ch. IV, \S\,5]{Lf1966E}) proved a direct analog of the Harnack inequality (\ref{eq:Harnack1}) for $(-\Delta )^s$ with $s\in(0,1)$:
\medskip

\textit{Let $u\geq0$ in $\R^n$, and let $(-\Delta)^s u=0$ in $B_R$. Then, for $x\in B_R$ we have}
$$
u(0)\,\frac {(R-|x|)^s R^{n-2s}}{(R+|x|)^{n-s}}\leq u(x)\leq u(0)\,\frac {(R+|x|)^s R^{n-2s}}{(R-|x|)^{n-s}}.
$$

In contrast to the case of the whole space, the fractional Laplacians in the domain $\Omega\subset\R^n$ certainly depend on the boundary conditions (there are fractional  Dirichlet Laplacians, Neumann Laplacians, etc.). Moreover, even for a fixed type of boundary conditions, there are several essentially different definitions of fractional Laplacians: {\bf restricted, spectral}, and some others. Note that to compare restricted and spectral Dirichlet Laplacian, in \cite{MN2014}, \cite{MN2016} the classical normal derivative lemma for weakly degenerate operators (see \cite{KH1977E}, \cite{ABMMZR2011E}) was used.

Proofs of the strong maximum principle for various fractional Laplacians of order $s\in(0,1)$ in $\Omega$ can be found in \cite{Sl2007}, \cite{CDDS2011}, \cite{IMS2015}; in the paper \cite{MN2019}, a unified approach was proposed for a large family of fractional Laplacians and more general nonlocal operators. On the other hand, it is shown in \cite{AJS2018}, \cite{AJS2021} that even the weak maximum principle does not hold for the restricted fractional Dirichlet Laplacian with $s>1$ in a domain of general form\footnote{Note that in $ \R^n$ as well as in the ball $\Omega=B_R$ the strong maximum principle holds for any $s>0$, which is also shown in \cite{AJS2018}.}.
\medskip 

The boundary Harnack inequality for the operator $(-\Delta)^s$, $s\in(0,1)$, in a Lipschitz domain was proved in \cite{Bg1997}. Due to the non-locality of the operator, its formulation differs from the standard one (see \S\,\ref{ss:4.3}):
\medskip

\noindent\textit{Let $0\in\Omega$. If $u_1$ and $u_2$ are non-negative functions in $\R^n$, continuous in the ball $B_R$, satisfying the equation $(-\Delta)^s u=0$ in $\Omega\cap B_R$ and the condition $u_1|_{B_R\setminus\Omega}=u_2|_{B_R\setminus\Omega}=0$, then the inequality (\ref{eq:BoundHarnack}) holds true\footnote{If $u_2(0)=0$ then $u_2\equiv0$. Therefore, we can assume that $u_2(0)>0$.} in the subdomain $\Omega\cap B_{R/2}$ with constant $C$ depending only on $n$, $s$, $\Omega$ and $R$.}
\medskip

\noindent Later this result was extended to {\bf\textit{arbitrary}} domains $\Omega$ and to a wide class of integro-differential operators (see \cite{ROS2019} and references therein).

In \cite{ROS2014} was constructed a barrier which is sufficient to prove an analogue of the normal derivative lemma in the following form:
\medskip

\noindent\textit{Let $\Omega$ be a ${\cal C}^{1,1}$-smooth domain, and let $s\in(0,1)$. Suppose that $u$ is a weak su\-per\-so\-lu\-tion to the equation $(-\Delta)^s u=0$ in $\Omega$, and $u=0$ in $\R^n\setminus\Omega$. If $u\not\equiv0$ then}
\begin{equation}
\inf\limits_{x\in\Omega}\, \frac  {u(x)}{{\rm d}^s(x)}\,>0.
\label{eq:bpp1}
\end{equation}

\noindent
Further generalizations of this result can be found, for instance, in \cite{RO2016}, \cite{RO2018}. For operators of fractional $p$-Laplacian type, a similar assertion was proved in \cite{DPQ2017}.
\medskip

For the {\bf\textit{spectral}} fractional Dirichlet Laplacian, instead of (\ref{eq:bpp1}), the in\-equal\-i\-ty $\inf\limits_{x\in\Omega}\, \dfrac {u(x) }{{\rm d}(x)}\,>0$ holds under the same assumptions (see Theorem~1.2 in \cite{KiSoV2019}, where more general functions of the Laplace operator with Dirichlet conditions are also considered). An analog of the normal derivative lemma for the {\bf regional} fractional Laplacian with $s \in (\frac{1}{2}, 1)$ was obtained in a recent preprint \cite{AFT2021}.
\medskip

In \cite{GuS2012}, a generalization of the Aleksandrov--Bakelman maximum prin\-ci\-ple for non-local analogs of the maximal and minimal Pucci operators is obtained.
\medskip

An application of the moving plane method to problems with fractional Laplacians can be found in \cite{FJ2015} and in the papers cited there.

\small
\bibliography{Bibliography_Hopf_surveyR}

\begin{thebibliography}{100}

\bibitem{AFT2021}
N.~Abatangelo, M.~Fall, and R.~Temgoua.
\newblock A {H}opf lemma for the regional fractional {L}aplacian.
\newblock Preprint, arxiv.org/abs/2112.09522, 2021.

\bibitem{AJS2018}
N.~Abatangelo, S.~Jarohs, and A.~Salda\~{n}a.
\newblock On the loss of maximum principles for higher-order fractional
  {L}aplacians.
\newblock {\em Proc. Amer. Math. Soc.}, 146(11):4823--4835, 2018.

\bibitem{AJS2021}
N.~Abatangelo, S.~Jarohs, and A.~Salda\~{n}a.
\newblock Fractional {L}aplacians on ellipsoids.
\newblock {\em Math. Eng.}, 3(5):Paper No. 038, 34, 2021.

\bibitem{AdLu2016}
T.~Adamowicz and N.~Lundstr\"{o}m.
\newblock The boundary {H}arnack inequality for variable exponent
  {$p$}-{L}aplacian, {C}arleson estimates, barrier functions and
  {$p(\cdot)$}-harmonic measures.
\newblock {\em Ann. Mat. Pura Appl. (4)}, 195(2):623--658, 2016.

\bibitem{Ae1960}
A.~Aeppli.
\newblock On the uniqueness of compact solutions for certain elliptic
  differential equations.
\newblock {\em Proc. Amer. Math. Soc.}, 11:826--832, 1960.

\bibitem{Aw2001}
H.~Aikawa.
\newblock Boundary {H}arnack principle and {M}artin boundary for a uniform
  domain.
\newblock {\em J. Math. Soc. Japan}, 53(1):119--145, 2001.

\bibitem{Aw2008}
H.~Aikawa.
\newblock Equivalence between the boundary {H}arnack principle and the
  {C}arleson estimate.
\newblock {\em Math. Scand.}, 103(1):61--76, 2008.

\bibitem{AmFT2001}
H.~Aimar, L.~Forzani, and R.~Toledano.
\newblock H\"{o}lder regularity of solutions of {PDE}'s: a geometrical view.
\newblock {\em Commun. Partial Differ. Equ.}, 26(7-8):1145--1173, 2001.

\bibitem{AiS1982}
M.~Aizenman and B.~Simon.
\newblock Brownian motion and {H}arnack inequality for {S}chr{\"o}dinger
  operators.
\newblock {\em {Commun. Pure Appl. Math.}}, 35:209--273, 1982.

\bibitem{Al1954E}
A.~Aleksandrov.
\newblock Some theorems on partial differential equations of second order.
\newblock {\em Vestnik Leningrad. Univ. Ser. Mat. Fiz. Him.}, 9(8):3--17, 1954.
\newblock [Russian].

\bibitem{Al1956E}
A.~Aleksandrov.
\newblock Uniqueness theorems for surfaces in the large. {I}.
\newblock {\em Vestnik Leningrad. Univ. Ser. Mat. Fiz. Him.}, 11(19):5--17,
  1956.
\newblock [Russian].

\bibitem{Al1958aE}
A.~Aleksandrov.
\newblock Dirichlet's problem for the equation $\text{Det}\| z_{ij}\| =\varphi
  (z_ 1,\ldots ,z_ n,z,x_ 1,\ldots ,x_ n)$.
\newblock {\em Vestnik Leningrad. Univ. Ser. Mat. Meh. Astr.}, 13(1):5--24,
  1958.
\newblock [Russian].

\bibitem{Al1958E}
A.~Aleksandrov.
\newblock Investigations on the maximum principle. {I}.
\newblock {\em Izv. Vys\v{s}. U\v{c}ebn. Zaved. Matematika}, 1958(5
  (6)):126--157, 1958.
\newblock [Russian].

\bibitem{Al1958bE}
A.~Aleksandrov.
\newblock Uniqueness theorems for surfaces in the large. {III}.
\newblock {\em Vestnik Leningrad. Univ. Ser. Mat. Meh. Astr.}, 13(7(2)):14--26,
  1958.
\newblock [Russian].

\bibitem{Al1958cE}
A.~Aleksandrov.
\newblock Uniqueness theorems for surfaces in the large. {V}.
\newblock {\em Vestnik Leningrad. Univ. Ser. Mat. Meh. Astr.}, 13(19(4)):5--8,
  1958.
\newblock [Russian].

\bibitem{Al1959E}
A.~Aleksandrov.
\newblock Investigation on the maximum principle. {II}.
\newblock {\em Izv. Vys\v{s}. U\v{c}ebn. Zaved. Matematika}, 1959(3
  (10)):3--12, 1959.
\newblock [Russian].

\bibitem{Al1959aE}
A.~Aleksandrov.
\newblock Investigation on the maximum principle. {III}.
\newblock {\em Izv. Vys\v{s}. U\v{c}ebn. Zaved. Matematika}, 1959(5
  (12)):16--32, 1959.
\newblock [Russian].

\bibitem{Al1960bE}
A.~Aleksandrov.
\newblock Certain estimates for the {D}irichlet problem.
\newblock {\em Dokl. Akad. Nauk SSSR}, 134:1001--1004, 1960.
\newblock [Russian]; English transl. in: \textit{Soviet Math. Dokl.},
  1:1151--1154, 1961.

\bibitem{Al1960E}
A.~Aleksandrov.
\newblock Investigation on the maximum principle. {IV}.
\newblock {\em Izv. Vys\v{s}. U\v{c}ebn. Zaved. Matematika}, 1960(3
  (16)):3--15, 1960.
\newblock [Russian].

\bibitem{Al1960aE}
A.~Aleksandrov.
\newblock Investigation on the maximum principle. {V}.
\newblock {\em Izv. Vys\v{s}. U\v{c}ebn. Zaved. Matematika}, 1960(5
  (18)):16--26, 1960.
\newblock [Russian].

\bibitem{Al1961E}
A.~Aleksandrov.
\newblock Investigation on the maximum principle. {VI}.
\newblock {\em Izv. Vys\v{s}. U\v{c}ebn. Zaved. Matematika}, 1961(1
  (20)):3--20, 1961.
\newblock [Russian].

\bibitem{Al1963E}
A.~Aleksandrov.
\newblock Uniqueness conditions and bounds for the solution of the {D}irichlet
  problem.
\newblock {\em Vestnik Leningrad. Univ. Ser. Mat. Meh. Astronom.}, 18(3):5--29,
  1963.
\newblock [Russian].

\bibitem{Al1966aE}
A.~Aleksandrov.
\newblock A general method of majorizing solutions of the {D}irichlet problem.
\newblock {\em Sibirsk. Mat. \v{Z}.}, 7:486--498, 1966.
\newblock [Russian].

\bibitem{Al1966dE}
A.~Aleksandrov.
\newblock Impossibility of general estimates of solutions and of uniqueness
  conditions for linear equations with norms weaker than those of {$L_n$}.
\newblock {\em Vestnik Leningrad. Univ. Ser. Mat. Meh. Astronom.},
  21(13(3)):5--10, 1966.
\newblock [Russian].

\bibitem{Al1966E}
A.~Aleksandrov.
\newblock Majorants of solutions of linear equations of order two.
\newblock {\em Vestnik Leningrad. Univ. Ser. Mat. Meh. Astronom.}, 21(1):5--25,
  1966.
\newblock [Russian].

\bibitem{Al1966bE}
A.~Aleksandrov.
\newblock On majorants of solutions and uniqueness conditions for elliptic
  equations.
\newblock {\em Vestnik Leningrad. Univ. Ser. Mat. Meh. Astronom.},
  21(7(2)):5--20, 1966.
\newblock [Russian].

\bibitem{Al1967bE}
A.~Aleksandrov.
\newblock Certain estimates for the derivative of the solution of the
  {D}irichlet problem on the boundary.
\newblock {\em Dokl. Akad. Nauk SSSR}, 173:487--490, 1967.
\newblock [Russian]; English transl. in: \textit{Soviet Math. Dokl.},
  8:396--400, 1967.

\bibitem{Al1967E}
A.~Aleksandrov.
\newblock Certain estimates of solutions of the {D}irichlet problem.
\newblock {\em Vestnik Leningrad. Univ. Ser. Mat. Meh. Astronom.},
  22(7(2)):19--29, 1967.
\newblock [Russian].

\bibitem{Al1967aE}
A.~Aleksandrov.
\newblock The maximum principle.
\newblock {\em Dokl. Akad. Nauk SSSR}, 173:247--250, 1967.
\newblock [Russian]; English transl. in: \textit{Soviet Math. Dokl.},
  8:352--355, 1967.

\bibitem{Al1965}
A.~Alexandrov.
\newblock The method of normal map in uniqueness problems and estimations for
  elliptic equations.
\newblock In {\em Seminari 1962/63 {A}nal. {A}lg. {G}eom. e {T}opol., {V}ol. 2,
  {I}st. {N}az. {A}lta {M}at.}, pages 744--786. Edizioni Cremonese, Rome, 1965.

\bibitem{Alk1997E}
Y.~Alkhutov.
\newblock The {H}arnack inequality and the {H}\"{o}lder property of solutions
  of nonlinear elliptic equations with a nonstandard growth condition.
\newblock {\em Differ. Uravn.}, 33(12):1651--1660, 1997.
\newblock [Russian]; English transl. in: \textit{Differ. Equ.},
  33(12):1653--1663, 1997.

\bibitem{AlkSur2019E}
Y.~Alkhutov and M.~Surnachev.
\newblock Harnack's inequality for the {$p(x)$}-{L}aplacian with a two-phase
  exponent {$p(x)$}.
\newblock {\em Tr. Semin. im. I. G. Petrovskogo}, 32:8--56, 2019.
\newblock [Russian]; English transl. in: \textit{J. Math. Sci. (N.Y.)},
  244(2):116--147, 2020.

\bibitem{AlkSur2022}
Y.~Alkhutov and M.~Surnachev.
\newblock Global {G}reen's function estimates for the convection-diffusion
  equation.
\newblock {\em Complex Var. Elliptic Equ.}, 67(5):1046--1075, 2022.

\bibitem{AllSh2019}
M.~Allen and H.~Shahgholian.
\newblock A new boundary {H}arnack principle (equations with right hand side).
\newblock {\em Arch. Rat. Mech. Anal.}, 234(3):1413--1444, 2019.

\bibitem{ABMMZR2011E}
R.~Alvarado, D.~Brigham, V.~Maz'ya, M.~Mitrea, and E.~Ziad\'{e}.
\newblock On the regularity of domains satisfying a uniform hour-glass
  condition and a sharp version of the {H}opf-{O}leinik boundary point
  principle.
\newblock {\em Problems in math. analysis}, 57:3--68, 2011.
\newblock [Russian]; English transl. in: \textit{J. Math. Sci. (N.Y.)},
  176(3):281--360, 2011.

\bibitem{An1978}
A.~Ancona.
\newblock Principe de {H}arnack \`a la fronti\`ere et th\'{e}or\`eme de {F}atou
  pour un op\'{e}rateur elliptique dans un domaine lipschitzien.
\newblock {\em Ann. Inst. Fourier (Grenoble)}, 28(4):169--213, 1978.
\newblock [French].

\bibitem{An1979}
A.~Ancona.
\newblock Une propri\'{e}t\'{e} d'invariance des ensembles absorbants par
  perturbation d'un op\'{e}rateur elliptique.
\newblock {\em Commun. Partial Differ. Equ.}, 4(4):321--337, 1979.
\newblock [French].

\bibitem{AN1995aE}
D.~Apushkinskaya and A.~Nazarov.
\newblock Boundary estimates for the first-order derivatives of a solution to a
  nondivergent parabolic equation with composite right-hand side and
  coefficients of lower-order derivatives.
\newblock {\em Problems in math. analysis}, 14:3--27, 1995.
\newblock [Russian]; English transl. in: \textit{J. Math. Sci.},
  77(4):3257--3276, 1995.

\bibitem{AN1995}
D.~Apushkinskaya and A.~Nazarov.
\newblock H\"{o}lder estimates of solutions to initial-boundary value problems
  for parabolic equations of nondivergent form with {W}entzel boundary
  condition.
\newblock In {\em Nonlinear evolution equations}, volume 164 of {\em Amer.
  Math. Soc. Transl. Ser. 2}, pages 1--13. AMS, Providence, RI, 1995.

\bibitem{AN1997E}
D.~Apushkinskaya and A.~Nazarov.
\newblock H\"{o}lder estimates for solutions to the degenerate boundary-value
  {V}enttsel problem for parabolic and elliptic equations of nondivergence
  type.
\newblock {\em Problems in math. analysis}, 17:3--19, 1997.
\newblock [Russian]; English transl. in: \textit{J. Math. Sci. (N.Y.)},
  97(4):4177--4188, 1999.

\bibitem{AN2001}
D.~Apushkinskaya and A.~Nazarov.
\newblock Linear two-phase {V}enttsel problems.
\newblock {\em Arkiv Mat.}, 39(2):201--222, 2001.

\bibitem{AN2016}
D.~Apushkinskaya and A.~Nazarov.
\newblock A counterexample to the {H}opf-{O}leinik lemma (elliptic case).
\newblock {\em Anal. PDE}, 9(2):439--458, 2016.

\bibitem{AN2019}
D.~Apushkinskaya and A.~Nazarov.
\newblock On the boundary point principle for divergence-type equations.
\newblock {\em Atti Accad. Naz. Lincei Rend. Lincei Mat. Appl.},
  30(4):677--699, 2019.

\bibitem{ANPS2021}
D.~Apushkinskaya, A.~Nazarov, D.~Palagachev, and L.~Softova.
\newblock Venttsel boundary value problems with discontinuous data.
\newblock {\em SIAM J. Math. Anal.}, 53(1):221--252, 2021.

\bibitem{ACP2011}
R.~Argiolas, F.~Charro, and I.~Peral.
\newblock On the {A}leksandrov-{B}akel'man-{P}ucci estimate for some elliptic
  and parabolic nonlinear operators.
\newblock {\em Arch. Rat. Mech. Anal.}, 202(3):875--917, 2011.

\bibitem{Ar1998E}
V.~Arnold.
\newblock On the teaching of mathematics.
\newblock {\em Uspekhi Mat. Nauk}, 53(1(319)):229--234, 1998.
\newblock [Russian].

\bibitem{AIM2006}
K.~Astala, T.~Iwaniec, and G.~Martin.
\newblock Pucci's conjecture and the {A}lexandrov inequality for elliptic
  {PDE}s in the plane.
\newblock {\em J. Reine Angew. Math.}, 591:49--74, 2006.

\bibitem{BaBaBu1991}
R.~Ba\~{n}uelos, R.~Bass, and K.~Burdzy.
\newblock H\"{o}lder domains and the boundary {H}arnack principle.
\newblock {\em Duke Math. J.}, 64(1):195--200, 1991.

\bibitem{B1958E}
I.~Bakelman.
\newblock On the theory of {M}onge-{A}mp\`ere's equations.
\newblock {\em Vestnik Leningrad. Univ. Ser. Mat. Meh. Astr.}, 13(1(1)):25--38,
  1958.
\newblock [Russian].

\bibitem{B1959E}
I.~Bakelman.
\newblock The {D}irichlet problem for equations of {M}onge-{A}mp\`ere type and
  their {$n$}-dimensional analogues.
\newblock {\em Dokl. Akad. Nauk SSSR}, 126:923--926, 1959.
\newblock [Russian].

\bibitem{B-disserE}
I.~Bakelman.
\newblock The first boundary value problem for non-linear elliptic equations.
  habilitation thesis in physics and mathematics, 1959.
\newblock A.I. Herzen Leningrad State Pedagogical Institute, Leningrad,
  [Russian].

\bibitem{B1961E}
I.~Bakelman.
\newblock On the theory of quasilinear elliptic equations.
\newblock {\em Sibirsk. Mat. \v{Z}.}, 2:179--186, 1961.
\newblock [Russian].

\bibitem{B1994}
I.~Bakelman.
\newblock {\em Convex analysis and nonlinear geometric elliptic equations}.
\newblock Berlin: Springer-Verlag, 1994.

\bibitem{BlM2019}
M.~Barlow and M.~Murugan.
\newblock Boundary {H}arnack principle and elliptic {H}arnack inequality.
\newblock {\em J. Math. Soc. Japan}, 71(2):383--412, 2019.

\bibitem{Ba1937}
J.~Barta.
\newblock {Sur la vibration fondamentale d'une membrane}.
\newblock {\em {C. R. Acad. Sci., Paris}}, 204:472--473, 1937.
\newblock [French].

\bibitem{BaBu1991}
R.~Bass and K.~Burdzy.
\newblock A boundary {H}arnack principle in twisted {H}\"{o}lder domains.
\newblock {\em Ann. of Math. (2)}, 134(2):253--276, 1991.

\bibitem{BaBu1994}
R.~Bass and K.~Burdzy.
\newblock The boundary {H}arnack principle for nondivergence form elliptic
  operators.
\newblock {\em J. London Math. Soc. (2)}, 50(1):157--169, 1994.

\bibitem{BBB2016}
E.~Battaglia, S.~Biagi, and A.~Bonfiglioli.
\newblock The strong maximum principle and the {H}arnack inequality for a class
  of hypoelliptic non-{H}\"{o}rmander operators.
\newblock {\em Ann. Inst. Fourier (Grenoble)}, 66(2):589--631, 2016.

\bibitem{Bau1984}
P.~Bauman.
\newblock Positive solutions of elliptic equations in nondivergence form and
  their adjoints.
\newblock {\em Ark. Mat.}, 22(2):153--173, 1984.

\bibitem{BSch2021}
P.~Bella and M.~Sch\"{a}ffner.
\newblock Local boundedness and {H}arnack inequality for solutions of linear
  nonuniformly elliptic equations.
\newblock {\em Commun. Pure Appl. Math.}, 74(3):453--477, 2021.

\bibitem{BeNi1991}
H.~Berestycki and L.~Nirenberg.
\newblock On the method of moving planes and the sliding method.
\newblock {\em Bol. Soc. Brasil. Mat. (N.S.)}, 22(1):1--37, 1991.

\bibitem{BeNiV1994}
H.~Berestycki, L.~Nirenberg, and S.~Varadhan.
\newblock The principal eigenvalue and maximum principle for second-order
  elliptic operators in general domains.
\newblock {\em Commun. Pure Appl. Math.}, 47(1):47--92, 1994.

\bibitem{BNi1955}
L.~Bers and L.~Nirenberg.
\newblock On linear and non-linear elliptic boundary value problems in the
  plane.
\newblock In {\em Convegno {I}nternazionale sulle {E}quazioni {L}ineari alle
  {D}erivate {P}arziali, {T}rieste, 1954}, pages 141--167. Edizioni Cremonese,
  Roma, 1955.

\bibitem{BCM2006}
S.~Bertone, A.~Cellina, and E.~Marchini.
\newblock On {H}opf's lemma and the strong maximum principle.
\newblock {\em Commun. Partial Differ. Equ.}, 31(4-6):701--733, 2006.

\bibitem{BST2015}
M.~Bertsch, F.~Smarrazzo, and A.~Tesei.
\newblock A note on the strong maximum principle.
\newblock {\em J. Differ. Equ.}, 259(8):4356--4375, 2015.

\bibitem{BIN1996E}
O.~V. Besov, V.~P. Il'in, and S.~M. Nikolski\u{\i}.
\newblock {\em Integral representations of functions and imbedding theorems}.
\newblock Fizmatlit, Nauka, Moscow, 2nd edition, 1996.
\newblock [Russian]; English transl. in: \textit{Scripta Series in
  Mathematics}. V.H. Winston \& Sons, Washington, D.C.; Halsted Press [John
  Wiley \& Sons], New York-Toronto, Ont.-London. Vol. I, 1978. Vol. II, 1979.

\bibitem{BVBV2006}
M.-F. Bidaut-V\'{e}ron, R.~Borghol, and L.~V\'{e}ron.
\newblock Boundary {H}arnack inequality and a priori estimates of singular
  solutions of quasilinear elliptic equations.
\newblock {\em Calc. Var. Partial Differ. Equ.}, 27(2):159--177, 2006.

\bibitem{Bl1965E}
G.~Blohina.
\newblock Theorems of {P}hragm\'{e}n-{L}indel\"{o}f type for a second-order
  linear elliptic equation.
\newblock {\em Dokl. Akad. Nauk SSSR}, 162(4):727--730, 1965.
\newblock [Russian]; English transl. in: \textit{Soviet Math. Dokl.},
  6:720--723, 1965.

\bibitem{Bl1970E}
G.~Blohina.
\newblock Theorems of {P}hragm\'{e}n-{L}indel\"{o}f type for a second order
  linear elliptic equation.
\newblock {\em Mat. Sb. (N.S.)}, 82(124)(4(8)):507--531, 1970.
\newblock [Russian].

\bibitem{Bg1997}
K.~Bogdan.
\newblock The boundary {H}arnack principle for the fractional {L}aplacian.
\newblock {\em Studia Math.}, 123(1):43--80, 1997.

\bibitem{Bo1967}
J.-M. Bony.
\newblock Principe du maximum dans les espaces de {S}obolev.
\newblock {\em C. R. Acad. Sci. Paris S\'{e}r. A-B}, 265:333--336, 1967.
\newblock [French].

\bibitem{BgM2018}
J.~Braga and D.~Moreira.
\newblock Inhomogeneous {H}opf-{O}le\u{\i}nik lemma and regularity of
  semiconvex supersolutions via new barriers for the {P}ucci extremal
  operators.
\newblock {\em Adv. Math.}, 334:184--242, 2018.

\bibitem{Br1992}
M.~Bramanti.
\newblock Potential theory for stationary {S}chr{\"o}dinger operators: a survey
  of results obtained with non-probabilistic methods.
\newblock {\em {Matematiche}}, 47(1):25--61, 1992.

\bibitem{B1931}
M.~Brelot.
\newblock {\'E}tude de l'{\'e}quation de la chaleur ${\Delta} u = c (m) u (m)$,
  $c(m)\geq 0$, au voisinage d'un point singulier du coefficient.
\newblock {\em Ann. Sci. \'Ec. Norm. Sup\'er. (3)}, 48:153--246, 1931.
\newblock [French].

\bibitem{BrP2003}
H.~Brezis and A.~Ponce.
\newblock Remarks on the strong maximum principle.
\newblock {\em Differential Integral Equations}, 16(1):1--12, 2003.

\bibitem{Cb1995}
X.~Cabr\'{e}.
\newblock On the {A}lexandroff--{B}akel'man--{P}ucci estimate and the reversed
  {H}\"{o}lder inequality for solutions of elliptic and parabolic equations.
\newblock {\em Commun. Pure Appl. Math.}, 48(5):539--570, 1995.

\bibitem{Cb2000}
X.~Cabr\'{e}.
\newblock Equacions en derivades parcials, geometria i control estoc{\`a}stic.
\newblock {\em Butll. Soc. Catalana. Mat.}, 15(1):7--27, 2000.
\newblock [Catalana].

\bibitem{Cb2002}
X.~Cabr\'{e}.
\newblock Topics in regularity and qualitative properties of solutions of
  nonlinear elliptic equations.
\newblock {\em Discrete Contin. Dyn. Syst.}, 8(2):331--359, 2002.

\bibitem{Cb2008}
X.~Cabr\'{e}.
\newblock Elliptic {PDE}'s in probability and geometry: symmetry and regularity
  of solutions.
\newblock {\em Discrete Contin. Dyn. Syst.}, 20(3):425--457, 2008.

\bibitem{CbROS2016}
X.~Cabr\'{e}, X.~Ros-Oton, and J.~Serra.
\newblock Sharp isoperimetric inequalities via the {ABP} method.
\newblock {\em J. Eur. Math. Soc. (JEMS)}, 18(12):2971--2998, 2016.

\bibitem{Cf1988}
L.~Caffarelli.
\newblock Elliptic second order equations.
\newblock {\em Rend. Sem. Mat. Fis. Milano}, 58:253--284, 1988.

\bibitem{Cf1989}
L.~Caffarelli.
\newblock Interior a priori estimates for solutions of fully nonlinear
  equations.
\newblock {\em Ann. of Math. (2)}, 130(1):189--213, 1989.

\bibitem{CfCb1995}
L.~Caffarelli and X.~Cabr\'{e}.
\newblock {\em Fully nonlinear elliptic equations}, volume~43 of {\em American
  Mathematical Society Colloquium Publications}.
\newblock AMS, Providence, RI, 1995.

\bibitem{CFMS1981}
L.~Caffarelli, E.~Fabes, S.~Mortola, and S.~Salsa.
\newblock Boundary behavior of nonnegative solutions of elliptic operators in
  divergence form.
\newblock {\em Indiana Univ. Math. J.}, 30(4):621--640, 1981.

\bibitem{CIN2018}
D.~Cao, A.~Ibraguimov, and A.~Nazarov.
\newblock Mixed boundary value problems for non-divergence type elliptic
  equations in unbounded domains.
\newblock {\em Asymptot. Anal.}, 109(1-2):75--90, 2018.

\bibitem{CDDS2011}
A.~Capella, J.~D\'{a}vila, L.~Dupaigne, and Y.~Sire.
\newblock Regularity of radial extremal solutions for some non-local semilinear
  equations.
\newblock {\em Commun. Partial Differ. Equ.}, 36(8):1353--1384, 2011.

\bibitem{CaRSc2019}
D.~Castorina, G.~Riey, and B.~Sciunzi.
\newblock Hopf {L}emma and regularity results for quasilinear anisotropic
  elliptic equations.
\newblock {\em Calc. Var. Partial Differ. Equ.}, 58(3):Paper No. 95, 18, 2019.

\bibitem{Cel2002}
A.~Cellina.
\newblock On the strong maximum principle.
\newblock {\em Proc. Amer. Math. Soc.}, 130(2):413--418, 2002.

\bibitem{CW1986}
S.~Chanillo and R.~Wheeden.
\newblock {Harnack's inequality and mean-value inequalities for solutions of
  degenerate elliptic equations}.
\newblock {\em {Commun. Partial Differ. Equ.}}, 11:1111--1134, 1986.

\bibitem{CPCM2013}
F.~Charro, G.~De~Philippis, A.~Di~Castro, and D.~M\'{a}ximo.
\newblock On the {A}leksandrov-{B}akelman-{P}ucci estimate for the infinity
  {L}aplacian.
\newblock {\em Calc. Var. Partial Differ. Equ.}, 48(3-4):667--693, 2013.

\bibitem{CSYT2008}
C.-C. Chen, R.~Strain, T.-P. Tsai, and H.-T. Yau.
\newblock Lower bounds on the blow-up rate of the axisymmetric
  {N}avier-{S}tokes equations.
\newblock {\em Int. Math. Res. Not. IMRN}, 2008(9):Art. ID rnn016, 1--31, 2008.

\bibitem{CSTY2009}
C.-C. Chen, R.~Strain, T.-P. Tsai, and H.-T. Yau.
\newblock Lower bounds on the blow-up rate of the axisymmetric
  {N}avier-{S}tokes equations. {II}.
\newblock {\em Commun. Partial Differ. Equ.}, 34(1-3):203--232, 2009.

\bibitem{ChS2017}
G.~Chen and M.~Safonov.
\newblock On second order elliptic and parabolic equations of mixed type.
\newblock {\em J. Funct. Anal.}, 272(8):3216--3237, 2017.

\bibitem{ChFG1986}
F.~Chiarenza, E.~Fabes, and N.~Garofalo.
\newblock Harnack's inequality for {S}chr{\"o}dinger operators and the
  continuity of solutions.
\newblock {\em {Proc. Amer. Math. Soc.}}, 98:415--425, 1986.

\bibitem{ChRuSe1989}
F.~Chiarenza, A.~Rustichini, and R.~Serapioni.
\newblock De {G}iorgi-{M}oser theorem for a class of degenerate non uniformly
  elliptic equations.
\newblock {\em {Commun. Partial Differ. Equ.}}, 14(5):635--662, 1989.

\bibitem{Ch1967}
M.~Chicco.
\newblock {Principio di massimo forte per sottosoluzioni di equazioni
  ellittiche di tipo variazionale}.
\newblock {\em {Boll. Unione Mat. Ital., III. Ser.}}, 22:368--372, 1967.
\newblock [Italian].

\bibitem{Ch1997}
M.~Chicco.
\newblock A maximum principle for mixed boundary value problems for elliptic
  equations in non-divergence form.
\newblock {\em Boll. Unione Mat. Ital. B (7)}, 11(3):531--538, 1997.

\bibitem{CCK2013}
S.~Cho, B.~Choe, and H.~Koo.
\newblock Weak {H}opf lemma for the invariant {L}aplacian and related elliptic
  operators.
\newblock {\em J. Math. Anal. Appl.}, 408(2):576--588, 2013.

\bibitem{Cor1956E}
H.~Cordes.
\newblock {\"U}ber die erste {R}andwertaufgabe bei quasilinearen
  {D}ifferential- gleichungen zweiter {O}rdnung in mehr als zwei {V}ariablen.
\newblock {\em {Math. Ann.}}, 131:278--312, 1956.
\newblock [German].

\bibitem{CrFZh1988}
M.~Cranston, E.~Fabes, and Z.~Zhao.
\newblock Conditional gauge and potential theory for the {S}chr\"{o}dinger
  operator.
\newblock {\em Trans. Amer. Math. Soc.}, 307(1):171--194, 1988.

\bibitem{CrZh1987}
M.~Cranston and Z.~Zhao.
\newblock Conditional transformation of drift formula and potential theory for
  $\frac 12{\Delta} +b(\cdot)\cdot \nabla$.
\newblock {\em {Commun. Math. Phys.}}, 112:613--625, 1987.

\bibitem{Dh1977}
B.~Dahlberg.
\newblock Estimates of harmonic measure.
\newblock {\em Arch. Rat. Mech. Anal.}, 65(3):275--288, 1977.

\bibitem{DPR2003}
L.~Damascelli, F.~Pacella, and M.~Ramaswamy.
\newblock A strong maximum principle for a class of non-positone singular
  elliptic problems.
\newblock {\em NoDEA Nonlin. Differ. Equ. Appl.}, 10(2):187--196, 2003.

\bibitem{DSc2004}
L.~Damascelli and B.~Sciunzi.
\newblock Regularity, monotonicity and symmetry of positive solutions of
  {$m$}-{L}aplace equations.
\newblock {\em J. Differ. Equ.}, 206(2):483--515, 2004.

\bibitem{DFQ2009}
G.~D\'{a}vila, P.~Felmer, and A.~Quaas.
\newblock Alexandroff-{B}akelman-{P}ucci estimate for singular or degenerate
  fully nonlinear elliptic equations.
\newblock {\em C. R. Math. Acad. Sci. Paris}, 347(19-20):1165--1168, 2009.

\bibitem{DeCV1996}
V.~De~Cicco and M.~Vivaldi.
\newblock Harnack inequalities for {F}uchsian type weighted elliptic equations.
\newblock {\em {Commun. Partial Differ. Equ.}}, 21(9-10):1321--1347, 1996.

\bibitem{DeG1957}
E.~De~Giorgi.
\newblock Sulla differenziabilit\`a e l'analiticit\`a delle estremali degli
  integrali multipli regolari.
\newblock {\em Mem. Accad. Sci. Torino. Cl. Sci. Fis. Mat. Nat. (3)}, 3:25--43,
  1957.
\newblock [Italian].

\bibitem{DeSS2020}
D.~De~Silva and O.~Savin.
\newblock A short proof of boundary {H}arnack principle.
\newblock {\em J. Differ. Equ.}, 269(3):2419--2429, 2020.

\bibitem{DeSS2022}
D.~De~Silva and O.~Savin.
\newblock On the boundary {H}arnack principle in {H}\"{o}lder domains.
\newblock {\em Math. Eng.}, 4(1):Paper No. 004, 12, 2022.

\bibitem{DPQ2017}
L.~Del~Pezzo and A.~Quaas.
\newblock A {H}opf's lemma and a strong minimum principle for the fractional
  {$p$}-{L}aplacian.
\newblock {\em J. Differ. Equ.}, 263(1):765--778, 2017.

\bibitem{DFFZ2015}
G.~Di~Fazio, M.~Fanciullo, and P.~Zamboni.
\newblock Harnack inequality for degenerate elliptic equations and sum
  operators.
\newblock {\em Commun. Pure Appl. Anal.}, 14(6):2363--2376, 2015.

\bibitem{DNPV2012}
E.~Di~Nezza, G.~Palatucci, and E.~Valdinoci.
\newblock Hitchhiker's guide to the fractional {S}obolev spaces.
\newblock {\em Bull. Sci. Math.}, 136(5):521--573, 2012.

\bibitem{DiB1989}
E.~DiBenedetto.
\newblock Harnack estimates in certain function classes.
\newblock {\em Atti Sem. Mat. Fis. Univ. Modena}, 37(1):173--182, 1989.

\bibitem{DiBTr1984}
E.~DiBenedetto and N.~Trudinger.
\newblock Harnack inequalities for quasiminima of variational integrals.
\newblock {\em Ann. Inst. H. Poincar\'{e} Anal. Non Lin\'{e}aire},
  1(4):295--308, 1984.

\bibitem{DEK2018}
H.~Dong, L.~Escauriaza, and S.~Kim.
\newblock On {${\mathcal C}^1$}, {${\mathcal C}^2$}, and weak type-$(1, 1)$
  estimates for linear elliptic operators: part {II}.
\newblock {\em Math. Ann.}, 370:447--489, 2018.

\bibitem{DV1972}
M.~A. Dow and R.~V\'{y}born\'{y}.
\newblock Maximum principles for some quasilinear second order partial
  differential equations.
\newblock {\em Rend. Sem. Mat. Univ. Padova}, 47:331--351, 1972.

\bibitem{E1839}
S.~Earnshaw.
\newblock On the nature of the molecular forces which regulate the constitution
  of the luminiferous ether.
\newblock {\em Cambridge Philos. Soc. Trans.}, 7:97--112, 1839.

\bibitem{EP1972}
D.~Edmunds and L.~Peletier.
\newblock A {H}arnack inequality for weak solutions of degenerate quasilinear
  elliptic equation.
\newblock {\em {J. London Math. Soc., II. Ser.}}, 5:21--31, 1972.

\bibitem{ErG2005}
Eridani and H.~Gunawan.
\newblock Stummel class of {M}orrey spaces.
\newblock {\em {Southeast Asian Bull. Math.}}, 29(6):1051--1056, 2005.

\bibitem{Es1994}
L.~Escauriaza.
\newblock Weak type-{$(1,1)$} inequalities and regularity properties of adjoint
  and normalized adjoint solutions to linear nondivergence form operators with
  {VMO} coefficients.
\newblock {\em Duke Math. J.}, 74(1):177--201, 1994.

\bibitem{Es2000}
L.~Escauriaza.
\newblock Bounds for the fundamental solution of elliptic and parabolic
  equations in nondivergence form.
\newblock {\em Commun. Partial Differ. Equ.}, 25(5-6):821--845, 2000.

\bibitem{EspSc2020}
F.~Esposito and B.~Sciunzi.
\newblock On the {H}opf boundary lemma for quasilinear problems involving
  singular nonlinearities and applications.
\newblock {\em J. Funct. Anal.}, 278(4):108346, 25, 2020.

\bibitem{E1956E}
D.~M. \'{E}\u{\i}dus.
\newblock Estimates on the derivatives of {G}reen's function.
\newblock {\em Dokl. Akad. Nauk SSSR}, 106:207--209, 1956.
\newblock [Russian].

\bibitem{FGMMS1988}
E.~Fabes, N.~Garofalo, S.~Mar\'{\i}n-Malave, and S.~Salsa.
\newblock Fatou theorems for some nonlinear elliptic equations.
\newblock {\em Rev. Mat. Iberoamer.}, 4(2):227--251, 1988.

\bibitem{FKS1982}
E.~Fabes, C.~Kenig, and R.~Serapioni.
\newblock {The local regularity of solutions of degenerate elliptic equations}.
\newblock {\em {Commun. Partial Differ. Equ.}}, 7:77--116, 1982.

\bibitem{FS1984}
E.~Fabes and D.~Stroock.
\newblock The {$L\sp p$}-integrability of {G}reen's functions and fundamental
  solutions for elliptic and parabolic equations.
\newblock {\em Duke Math. J.}, 51(4):997--1016, 1984.

\bibitem{FJ2015}
M.~Fall and S.~Jarohs.
\newblock Overdetermined problems with fractional {L}aplacian.
\newblock {\em ESAIM Control Optim. Calc. Var.}, 21(4):924--938, 2015.

\bibitem{Fh2020}
P.~Feehan.
\newblock Perturbations of local maxima and comparison principles for
  boundary-degenerate linear differential equations.
\newblock {\em Trans. Amer. Math. Soc.}, 373(8):5275--5332, 2020.

\bibitem{FQ2002}
P.~Felmer and A.~Quaas.
\newblock On the strong maximum principle for quasilinear elliptic equations
  and systems.
\newblock {\em Adv. Differential Equations}, 7(1):25--46, 2002.

\bibitem{Fe1998}
F.~Ferrari.
\newblock On boundary behavior of harmonic functions in {H}\"{o}lder domains.
\newblock {\em J. Fourier Anal. Appl.}, 4(4-5):447--461, 1998.

\bibitem{FrSa2001}
E.~Ferretti and M.~Safonov.
\newblock Growth theorems and {H}arnack inequality for second order parabolic
  equations.
\newblock In {\em Harmonic analysis and boundary value problems
  ({F}ayetteville, {AR}, 2000)}, volume 277 of {\em Contemp. Math.}, pages
  87--112. AMS, Providence, RI, 2001.

\bibitem{Fi2013E}
N.~Filonov.
\newblock On the regularity of solutions to the equation \mbox{$-\Delta
  u+b\cdot\nabla u=0$}.
\newblock {\em Zap. Nauchn. Sem. S.-Peterburg. Otdel. Mat. Inst. Steklov.
  (POMI)}, 410:168--186, 2013.

\bibitem{FiHd2021E}
N.~Filonov and P.~Khodunov.
\newblock {On the local boundedness of solutions to the equation $-\Delta u + a
  \partial_z u = 0$}.
\newblock {\em Zap. Nauchn. Sem. S.-Peterburg. Otdel. Mat. Inst. Steklov.
  (POMI)}, 508:173--184, 2021.
\newblock [Russian].

\bibitem{FiSh2018}
N.~Filonov and T.~Shilkin.
\newblock On some properties of weak solutions to elliptic equations with
  divergence-free drifts.
\newblock In {\em Mathematical analysis in fluid mechanics -- selected recent
  results}, volume 710 of {\em Contemp. Math.}, pages 105--120. AMS,
  Providence, RI, 2018.

\bibitem{FGi1957}
R.~Finn and D.~Gilbarg.
\newblock Asymptotic behavior and uniquenes of plane subsonic flows.
\newblock {\em Commun. Pure Appl. Math.}, 10:23--63, 1957.

\bibitem{Frn2000}
L.~Fraenkel.
\newblock {\em An introduction to maximum principles and symmetry in elliptic
  problems}, volume 128 of {\em Cambridge Tracts in Mathematics}.
\newblock Cambridge University Press, Cambridge, 2000.

\bibitem{FSSC1998}
B.~Franchi, R.~Serapioni, and F.~Serra~Cassano.
\newblock Irregular solutions of linear degenerate elliptic equations.
\newblock {\em Potential Anal.}, 9(3):201--216, 1998.

\bibitem{Fr1989}
M.~Franciosi.
\newblock Maximum principle for second order elliptic equations and
  applications.
\newblock {\em J. Math. Anal. Appl.}, 138(2):343--348, 1989.

\bibitem{G1839E}
C.~Gauss.
\newblock {A}llgemeine {T}heorie des {E}rdmagnetismus.
\newblock In {\em {B}eobachtungen des magnetischen {V}ereins im Jahre 1838}.
  Leipzig, 1839.
\newblock [German].

\bibitem{G1840E}
C.~Gauss.
\newblock {\em {A}llgemeine Lehrs{\"a}tze in {B}eziehung auf die im verkehrten
  {V}erh{\"a}ltnisse des {Q}uadrats der {E}ntfernung wirkenden {A}nziehungs-
  und {A}bsto{\ss}ungskr{\"a}fte.}
\newblock Wiedmannschen Buchhandlung, Leipzig, 1840.
\newblock [German].

\bibitem{GaP2000E}
A.~Gavrilov and O.~Penkin.
\newblock An analogue of the lemma on the normal derivative for an elliptic
  equation on a stratified set.
\newblock {\em Differ. Uravn.}, 36(2):226--232, 2000.
\newblock [Russian]; English transl. in \textit{Differ. Equ.}, 36(2):255--261,
  2000.

\bibitem{GNN1979}
B.~Gidas, W.~Ni, and L.~Nirenberg.
\newblock Symmetry and related properties via the maximum principle.
\newblock {\em Commun. Math. Phys.}, 68(3):209--243, 1979.

\bibitem{GNN1981}
B.~Gidas, W.~Ni, and L.~Nirenberg.
\newblock Symmetry of positive solutions of nonlinear elliptic equations in
  {${\R}\sp{n}$}.
\newblock In {\em Mathematical analysis and applications, {P}art {A}}, volume~7
  of {\em Adv. in Math. Suppl. Stud.}, pages 369--402. Academic Press, New
  York-London, 1981.

\bibitem{Gi1959}
D.~Gilbarg.
\newblock Some hydrodynamic applications of function theoretic properties of
  elliptic equations.
\newblock {\em Math. Zeitschr.}, 72:165--174, 1959.

\bibitem{GTr1983}
D.~Gilbarg and N.~Trudinger.
\newblock {\em Elliptic partial differential equations of second order}, volume
  224 of {\em Grundlehren der mathematischen Wissenschaften [Fundamental
  Principles of Mathematical Sciences]}.
\newblock Springer-Verlag, Berlin, 2nd edition, 1983.

\bibitem{G1932}
G.~Giraud.
\newblock Generalisation des probl{\`e}mes sur les operations du type
  elliptique.
\newblock {\em Bull. des Sciences Math.}, 56:316--352, 1932.
\newblock [French].

\bibitem{G1933}
G.~Giraud.
\newblock Probl{\`e}mes de valeurs {\`a} la fronti{\`e}re relatifs {\`a}
  certaines donn{\'e}es discontinues.
\newblock {\em Bull. Soc. Math. France}, 61:1--54, 1933.
\newblock [French].

\bibitem{GoN2013}
E.~G\"{o}tmark and K.~Nystr\"{o}m.
\newblock Boundary behavior of non-negative solutions to degenerate
  sub-elliptic equations.
\newblock {\em J. Differ. Equ.}, 254(8):3431--3460, 2013.

\bibitem{GKPR1998}
M.~Grossi, S.~Kesavan, F.~Pacella, and M.~Ramaswamy.
\newblock Symmetry of positive solutions of some nonlinear equations.
\newblock {\em Topol. Methods Nonlinear Anal.}, 12(1):47--59, 1998.

\bibitem{GrW1982}
M.~Gr\"{u}ter and K.-O. Widman.
\newblock The {G}reen function for uniformly elliptic equations.
\newblock {\em {Manuscripta Math.}}, 37:303--342, 1982.

\bibitem{GuS2012}
N.~Guillen and R.~Schwab.
\newblock Aleksandrov-{B}akelman-{P}ucci type estimates for
  integro-differential equations.
\newblock {\em Arch. Rat. Mech. Anal.}, 206(1):111--157, 2012.

\bibitem{H1887}
A.~Harnack.
\newblock {Die {G}rundlagen der {T}heorie des logarithmischen {P}otentiales und
  der eindeutigen {P}otentialfunction in der {E}bene}.
\newblock {Leipzig. Teubner}, 1887.
\newblock [German].

\bibitem{HH1969}
R.-M. Herv{\'e} and M.~Herv{\'e}.
\newblock {Les fonctions surharmoniques associees {\`a} un op{\'e}rateur
  elliptique du second ordre {\`a} coefficients discontinus}.
\newblock {\em {Ann. Inst. Fourier}}, 19(1):305--359, 1969.
\newblock [French].

\bibitem{Hi1969E}
B.~Him\v{c}enko.
\newblock The behavior of a superharmonic function near the boundary of a
  region of type {$A^{(1)}$}.
\newblock {\em Differ. Uravn.}, 5(10):1845--1853, 1969.
\newblock [Russian].

\bibitem{Hi1970bE}
B.~Him\v{c}enko.
\newblock The boundedness in a closed region of the gradient of a harmonic
  funcion.
\newblock {\em Uspehi Mat. Nauk}, 25(2(152)):279--280, 1970.
\newblock [Russian].

\bibitem{Hi1970aE}
B.~Him\v{c}enko.
\newblock A certain theorem of {M}.{V}. {K}eldy\v{s} and {M}.{A}. {L}avrent`ev.
\newblock {\em Dokl. Akad. Nauk SSSR}, 192(1):46--47, 1970.
\newblock [Russian]; English transl. in \textit{Soviet Math. Dokl.},
  11:595--596, 1970.

\bibitem{Hi1970E}
B.~Him\v{c}enko.
\newblock On the behavior of solutions of elliptic equations near the boundary
  of a domain of type {$A^{(1)}$}.
\newblock {\em Dokl. Akad. Nauk SSSR}, 193(2):304--305, 1970.
\newblock [Russian]; English transl. in \textit{Soviet Math. Dokl.},
  11:943--944, 1970.

\bibitem{H1927}
E.~Hopf.
\newblock Elementare {B}emerkungen {\"u}ber die {L}{\"o}sungen partieller
  {D}ifferentialgleichungen zweiter {O}rdnung vom elliptischen {T}ypus.
\newblock {\em {Sitzungsber. Preu{\ss}. Akad. Wiss., Phys.-Math. Kl.}},
  19:147--152, 1927.
\newblock [German].

\bibitem{H1952}
E.~Hopf.
\newblock A remark on linear elliptic differential equations of second order.
\newblock {\em Proc. Amer. Math. Soc.}, 3:791--793, 1952.

\bibitem{H2002}
E.~Hopf.
\newblock {\em Selected works of {E}berhard {H}opf with commentaries}.
\newblock American Mathematical Society, Providence, RI, 2002.
\newblock Edited by C.S. Morawetz, J.B. Serrin and Ya.G. Sinai.

\bibitem{HLW2020}
Y.~Huang, D.~Li, and L.~Wang.
\newblock A note on boundary differentiability of solutions of elliptic
  equations in nondivergence form.
\newblock {\em Manuscripta Math.}, 162(3-4):305--313, 2020.

\bibitem{IMS2015}
A.~Iannizzotto, S.~Mosconi, and M.~Squassina.
\newblock {$H^s$} versus {$C^0$}-weighted minimizers.
\newblock {\em NoDEA Nonlin. Differ. Equ. Appl.}, 22(3):477--497, 2015.

\bibitem{IbLa1996E}
A.~Ibragimov and E.~Landis.
\newblock On the behavior of solutions of the {N}eumann problem in unbounded
  domains.
\newblock {\em Tr. Semin. im. I.G. Petrovskogo}, 19:218--234, 1996.
\newblock [Russian]; English transl. in \textit{J. Math. Sci. (New York)},
  85(6):2373--2384, 1997.

\bibitem{IbLa1997}
A.~Ibragimov and E.~Landis.
\newblock Zaremba's problem for elliptic equations in the neighbourhood of a
  singular point or at infinity.
\newblock {\em Appl. Anal.}, 67(3-4):269--282, 1997.

\bibitem{IbNa2017E}
A.~Ibraguimov and A.~Nazarov.
\newblock On {P}hragm\'{e}n-{L}indel\"{o}f principle for non-divergence type
  elliptic equations and mixed boundary conditions.
\newblock {\em Mat. Fiz. Komp`yut. Model.}, 20(3(40)):65--76, 2017.

\bibitem{I2011}
C.~Imbert.
\newblock Alexandroff-{B}akelman-{P}ucci estimate and {H}arnack inequality for
  degenerate/singular fully non-linear elliptic equations.
\newblock {\em J. Differ. Equ.}, 250(3):1553--1574, 2011.

\bibitem{JK1982}
D.~Jerison and C.~Kenig.
\newblock Boundary behavior of harmonic functions in nontangentially accessible
  domains.
\newblock {\em Adv. in Math.}, 46(1):80--147, 1982.

\bibitem{Ju2016}
V.~Julin.
\newblock Generalized {H}arnack inequality for semilinear elliptic equations.
\newblock {\em J. Math. Pures Appl. (9)}, 106(5):877--904, 2016.

\bibitem{Ka1985E}
L.~Kamynin.
\newblock A theorem on the interior derivative for a weakly degenerate elliptic
  equation of second order.
\newblock {\em Mat. Sb. (N.S.)}, 126(168)(3):307--326, 1985.
\newblock [Russian]; English transl. in \textit{Math. USSR, Sb.}, 54:297--316,
  1986.

\bibitem{KH1971E}
L.~Kamynin and B.~Him\v{c}enko.
\newblock The maximum principle for second order parabolic equations.
\newblock {\em Dokl. Akad. Nauk SSSR}, 200:282--285, 1971.
\newblock [Russian]; English transl. in \textit{Soviet Math. Dokl.},
  12:1383--1387, 1971.

\bibitem{KH1972E}
L.~Kamynin and B.~Him\v{c}enko.
\newblock The maximum principle for a second order elliptic-parabolic equation.
\newblock {\em Sibirsk. Mat. \v{Z}.}, 13(4):773--789, 1972.
\newblock [Russian].

\bibitem{KH1975E}
L.~Kamynin and B.~Him\v{c}enko.
\newblock Theorems of {G}iraud type for a second order elliptic operator that
  is weakly degenerate near the boundary.
\newblock {\em Dokl. Akad. Nauk SSSR}, 224(4):752--755, 1975.
\newblock [Russian]; English transl. in \textit{Soviet Math. Dokl.},
  16:1287--1291, 1975.

\bibitem{KH1977E}
L.~Kamynin and B.~Him\v{c}enko.
\newblock Theorems of {G}iraud type for second order equations with a weakly
  degenerate non-negative characteristic part.
\newblock {\em Sibirsk. Mat. \v{Z}.}, 18(1):103--121, 1977.
\newblock [Russian]; English transl. in \textit{Sib. Math. J.}, 18:76--91,
  1977.

\bibitem{KH1978aE}
L.~Kamynin and B.~Him\v{c}enko.
\newblock Local estimates, near the boundary, of the solutions of second- order
  equations with a nonnegative characteristic form.
\newblock {\em Mat. Sb. (N.S.)}, 106(148)(8):162--182, 1978.
\newblock [Russian]; English transl. in \textit{Math. USSR, Sb.}, 34:715--735,
  1978.

\bibitem{KH1978E}
L.~Kamynin and B.~Him\v{c}enko.
\newblock On investigations of the maximum principle.
\newblock {\em Dokl. Akad. Nauk SSSR}, 240(4):774--777, 1978.
\newblock [Russian]; English transl. in \textit{Soviet Math. Dokl.},
  19:677--681, 1978.

\bibitem{KH1979cE}
L.~Kamynin and B.~Him\v{c}enko.
\newblock On investigations of the isotropic strict extremum principle.
\newblock {\em Dokl. Akad. Nauk SSSR}, 244(6):1312--1316, 1979.
\newblock [Russian]; English transl. in \textit{Soviet Math. Dokl.},
  20:224--228, 1979.

\bibitem{KH1979bE}
L.~Kamynin and B.~Him\v{c}enko.
\newblock The strict extremum principle for a {$D-(\Phi ,\,\Omega
  )$}-elliptically connected second-order operator.
\newblock {\em Differ. Uravn.}, 15(7):1307--1317, 1979.
\newblock [Russian].

\bibitem{KH1979E}
L.~Kamynin and B.~Him\v{c}enko.
\newblock A strict extremum principle for a weakly elliptically connected
  second-order operator.
\newblock {\em Zh. Vychisl. Mat. i Mat. Fiz.}, 19(1):129--142, 1979.
\newblock [Russian].

\bibitem{KH1979aE}
L.~Kamynin and B.~Him\v{c}enko.
\newblock A strict extremum principle that is isotropic in a plane domain.
\newblock {\em Sibirsk. Mat. Zh.}, 20(2):278--292, 1979.
\newblock [Russian]; English transl. in \textit{Sib. Math. J.}, 20:197--208,
  1979.

\bibitem{KH1980E}
L.~Kamynin and B.~Him\v{c}enko.
\newblock An aspect of the development of the theory of the isotropic strict
  extremum principle of {A}.{D}. {A}leksandrov.
\newblock {\em Differ. Uravn.}, 16(2):280--292, 1980.
\newblock [Russian]; English transl. in \textit{Differ. Equ.}, 16:181--189,
  1980.

\bibitem{KH1980aE}
L.~Kamynin and B.~Him\v{c}enko.
\newblock Strict extremum principle for a second-order {$D-(\Pi ,\,\Omega
  )$}-elliptically connected operator.
\newblock {\em Mat. Sb. (N.S.)}, 112(154)(1(5)):24--55, 1980.
\newblock [Russian].

\bibitem{Kas2007}
M.~Kassmann.
\newblock Harnack inequalities: an introduction.
\newblock {\em Bound. Value Probl.}, 2007:081415, 1--21, 2007.

\bibitem{Kato1972}
T.~Kato.
\newblock Schr{\"{o}}dinger operators with singular potentials.
\newblock {\em Israel J. Math.}, 13:135--148, 1972.

\bibitem{KL1937E}
M.~Keldysh and M.~Lavrentiev.
\newblock On the uniqueness of {N}eumann's problem.
\newblock {\em Dokl. Akad. Nauk SSSR}, XVI(3):151--152, 1937.
\newblock [Russian].

\bibitem{K1912}
O.~Kellogg.
\newblock Harmonic functions and {G}reen's integral.
\newblock {\em Trans. Amer. Math. Soc.}, 13:109--132, 1912.

\bibitem{KNd2001}
C.~Kenig and N.~Nadirashvili.
\newblock On optimal estimates for some oblique derivative problems.
\newblock {\em J. Funct. Anal.}, 187(1):70--93, 2001.

\bibitem{Hd2021E}
P.~Khodunov.
\newblock {\em On the local properties of solutions to hydrodynamics problems}.
\newblock Master's thesis, St. Petersburg State University, 2021.

\bibitem{KR2020}
D.~Kim and S.~Ryu.
\newblock The weak maximum principle for second-order elliptic and parabolic
  conormal derivative problems.
\newblock {\em Commun. Pure Appl. Anal.}, 19(1):493--510, 2020.

\bibitem{KiSa2011aE}
H.~Kim and M.~Safonov.
\newblock Boundary {H}arnack principle for second order elliptic equations with
  unbounded drift.
\newblock {\em Problems in math. analysis}, 61:109--122, 2011.
\newblock [Russian]; English transl. in \textit{J. Math. Sci. (N.Y.)},
  179(1):127--143, 2011.

\bibitem{KiSa2011E}
H.~Kim and M.~Safonov.
\newblock Carleson type estimates for second order elliptic equations with
  unbounded drift.
\newblock {\em Problems in math. analysis}, 58:195--207, 2011.
\newblock [Russian]; English transl. in \textit{J. Math. Sci. (N.Y.)},
  176(6):928--944, 2011.

\bibitem{KiSa2014}
H.~Kim and M.~Safonov.
\newblock {The boundary {H}arnack principle for second order elliptic equations
  in {J}ohn and uniform domains}.
\newblock In {\em {Proceedings of the St. Petersburg Mathematical Society. Vol.
  XV. Advances in mathematical analysis of partial differential equations.
  Workshop dedicated to the 90th anniversary of the O. A. Ladyzhenskaya
  birthday, Stockholm, Sweden, July 9--13, 2012}}, pages 153--176. Providence,
  RI: AMS, 2014.

\bibitem{KiSoV2019}
P.~Kim, R.~Song, and Z.~Vondra\v{c}ek.
\newblock Potential theory of subordinate killed {B}rownian motion.
\newblock {\em Trans. Amer. Math. Soc.}, 371(6):3917--3969, 2019.

\bibitem{KNSS2009}
G.~Koch, N.~Nadirashvili, G.~Seregin, and V.~\v{S}ver\'{a}k.
\newblock Liouville theorems for the {N}avier-{S}tokes equations and
  applications.
\newblock {\em Acta Math.}, 203(1):83--105, 2009.

\bibitem{Ko2013}
S.~Koike.
\newblock On the {ABP} maximum principle and applications.
\newblock {\em Suurikaisekikenkyusho Kokyuroku}, 1845:107--120, 2013.

\bibitem{KoLa1988E}
V.~Kondrat`ev and E.~Landis.
\newblock Qualitative theory of second-order linear partial differential
  equations.
\newblock {\em Itogi Nauki Tekh., Ser. Sovrem. Probl. Mat., Fundam.
  Napravleniya}, 32:99--215, 1988.
\newblock [Russian]; English transl. in \textit{Partial differential equations
  III. Encycl. Math. Sci.}, 32:87--192, 1991.

\bibitem{K1901}
A.~Korn.
\newblock {\em Lehrbuch der Potentialtheorie. {II}. Allgemeine Theorie des
  logarithmischen Potentials und der Potentialfunctionen in der Ebene}.
\newblock Berlin: F. D{\"u}mmler, 1901.
\newblock [German].

\bibitem{KzKu2018E}
V.~Kozlov and N.~Kuznetsov.
\newblock A comparison theorem for super- and subsolutions of $\nabla^2
  u+f(u)=0$ and its application to water waves with vorticity.
\newblock {\em Algebra Anal.}, 30(3):112--128, 2018.
\newblock \textit{St. Petersburg Math. J.}, 27(3):471--483, 2019.

\bibitem{KzM2003}
V.~Kozlov and V.~Maz'ya.
\newblock Asymptotic formula for solutions to the {D}irichlet problem for
  elliptic equations with discontinuous coefficients near the boundary.
\newblock {\em Ann. Sc. Norm. Super. Pisa Cl. Sci. (5)}, 2(3):551--600, 2003.

\bibitem{KzNa2021}
V.~Kozlov and A.~Nazarov.
\newblock {A comparison theorem for nonsmooth nonlinear operators}.
\newblock {\em {Potential Anal.}}, 54(3):471--481, 2021.

\bibitem{KM2020}
G.~Kresin and V.~Maz'ya.
\newblock On sharp {A}gmon-{M}iranda maximum principles.
\newblock Preprint, arxiv.org/abs/2009.01805, 2020.

\bibitem{Kr1974E}
N.~Krylov.
\newblock Some estimates for the density of the distribution of a stochastic
  integral.
\newblock {\em Izv. Akad. Nauk SSSR Ser. Mat.}, 38(1):228--248, 1974.
\newblock [Russian].

\bibitem{Kr1976aE}
N.~Krylov.
\newblock The maximum principle for parabolic equations.
\newblock {\em Uspehi Mat. Nauk}, 31(4(190)):267--268, 1976.
\newblock [Russian].

\bibitem{Kr1976E}
N.~Krylov.
\newblock Sequences of convex functions, and estimates of the maximum of the
  solution of a parabolic equation.
\newblock {\em Sibirsk. Mat. \v{Z}.}, 17(2):290--303, 1976.
\newblock [Russian]; English transl. in \textit{Sib. Math. J.}, 17:226--236,
  1976.

\bibitem{Kr1983E}
N.~Krylov.
\newblock Boundedly inhomogeneous elliptic and parabolic equations in a domain.
\newblock {\em Izv. Akad. Nauk SSSR Ser. Mat.}, 47(1):75--108, 1983.
\newblock [Russian]; English transl. in \textit{Math. USSR, Izv.}, 20:459--492,
  1983.

\bibitem{KrSa1979E}
N.~Krylov and M.~Safonov.
\newblock An estimate for the probability of a diffusion process hitting a set
  of positive measure.
\newblock {\em Dokl. Akad. Nauk SSSR}, 245(1):18--20, 1979.
\newblock [Russian]; English transl. in \textit{Soviet Math. Dokl.},
  20:253--256, 1979.

\bibitem{KrSa1980E}
N.~Krylov and M.~Safonov.
\newblock A property of the solutions of parabolic equations with measurable
  coefficients.
\newblock {\em Izv. Akad. Nauk SSSR Ser. Mat.}, 44(1):161--175, 1980.
\newblock [Russian].

\bibitem{KuTr2000}
H.-J. Kuo and N.~Trudinger.
\newblock A note on the discrete {A}leksandrov-{B}akelman maximum principle.
\newblock In {\em Proceedings of 1999 {I}nternational {C}onference on
  {N}onlinear {A}nalysis ({T}aipei)}, volume~4, pages 55--64, 2000.

\bibitem{KuTr2007}
H.-J. Kuo and N.~Trudinger.
\newblock New maximum principles for linear elliptic equations.
\newblock {\em Indiana Univ. Math. J.}, 56(5):2439--2452, 2007.

\bibitem{Kur1994}
K.~Kurata.
\newblock Continuity and {H}arnack's inequality for solutions of elliptic
  partial differential equations of second order.
\newblock {\em {Indiana Univ. Math. J.}}, 43(2):411--440, 1994.

\bibitem{Kz2019E}
N.~Kuznetsov.
\newblock Mean value properties of harmonic functions and related topics (a
  survey).
\newblock {\em Problems in math. analysis}, 99:3--21, 2019.
\newblock [Russian]; English transl. in: \textit{J. Math. Sci. (N.Y.)},
  242(2):177--199, 2019.

\bibitem{LSU1967E}
O.~Lady\v{z}enskaja, V.~Solonnikov, and N.~Ural'ceva.
\newblock {\em Linear and quasilinear equations off parabolic type}.
\newblock Nauka, Moscow, 1967.
\newblock [Russian]; English transl. in: \textit{Transl. Math. Monogr.},
  \textbf{23}, AMS, Providence, RI, 1968.

\bibitem{LU1985E}
O.~Ladyzhenskaya and N.~Ural'tseva.
\newblock Solvability of the first boundary value problem for quasilinear
  elliptic and parabolic equations in the presence of singularities.
\newblock {\em Dokl. Akad. Nauk SSSR}, 281(2):275--279, 1985.
\newblock [Russian]; English transl. in: \textit{Soviet Math. Dokl.},
  31:296--300, 1985.

\bibitem{LU1986E}
O.~Ladyzhenskaya and N.~Ural'tseva.
\newblock A survey of results on the solvability of boundary value problems for
  uniformly elliptic and parabolic second-order quasilinear equations having
  unbounded singularities.
\newblock {\em Uspekhi Mat. Nauk}, 41(5(251)):59--83, 1986.
\newblock [Russian]; English transl. in: \textit{Russian Math. Surv.},
  41(5):1--31, 1986.

\bibitem{LU1988E}
O.~Ladyzhenskaya and N.~Ural'tseva.
\newblock Estimates on the boundary of the domain of first derivatives of
  functions satisfying an elliptic or a parabolic inequality.
\newblock In {\em Boundary value problems of mathematical physics.~13}, volume
  179 of {\em Trudy Mat. Inst. Steklov.}, pages 102--125. Nauka, Moscow, 1988.
\newblock [Russian]; English transl. in: \textit{Proc. Steklov Inst. Math.},
  179:109--135, 1989.

\bibitem{La1956E}
E.~Landis.
\newblock On some properties of solutions of elliptic equations.
\newblock In {\em Proc. of the 3rd USSR Mathem. Congress, v.1}, pages 57--58.
  Moscow, 1956.
\newblock [Russian].

\bibitem{La1956bE}
E.~Landis.
\newblock On some properties of solutions of elliptic equations.
\newblock {\em Uspekhi Mat. Nauk}, 11(2(68)):235--237, 1956.
\newblock [Russian].

\bibitem{La1956aE}
E.~Landis.
\newblock On the {P}hragm\'{e}n-{L}indel\"{o}f principle for elliptic
  equations.
\newblock {\em Dokl. Akad. Nauk SSSR (N.S.)}, 107(4):508--511, 1956.
\newblock [Russian].

\bibitem{La1959E}
E.~Landis.
\newblock Some questions in the qualitative theory of elliptic and parabolic
  equations.
\newblock {\em Uspehi Mat. Nauk}, 14(1(85)):21--85, 1959.
\newblock [Russian].

\bibitem{La1963E}
E.~Landis.
\newblock Some questions in the qualitative theory of second-order elliptic
  equations (case of several independent variables).
\newblock {\em Uspehi Mat. Nauk}, 18(1(109)):3--62, 1963.
\newblock [Russian]; English transl. in: \textit{Russian Math. Survey},
  18(1):1--62, 1963.

\bibitem{La1968aE}
E.~Landis.
\newblock Harnack's inequality for second order elliptic equations of {C}ordes
  type.
\newblock {\em Dokl. Akad. Nauk SSSR}, 179(6):1272--1275, 1968.
\newblock [Russian]; English transl. in: \textit{Soviet Math. Dokl.},
  9:540--543, 1968.

\bibitem{La1968E}
E.~Landis.
\newblock {$s$}-capacity and its application to the investigation of solutions
  of a second order elliptic equation with discontinuous coefficients.
\newblock {\em Mat. Sb. (N.S.)}, 76(118)(2):186--213, 1968.
\newblock [Russian].

\bibitem{La1968bE}
E.~Landis.
\newblock {$s$}-capacity and the behavior of the solution of a second order
  elliptic equation with discontinuous coefficients in the neighborhood of a
  boundary point.
\newblock {\em Dokl. Akad. Nauk SSSR}, 180(1):25--28, 1968.
\newblock [Russian]; English transl. in: \textit{Soviet Math. Dokl.},
  9:582--586, 1968.

\bibitem{La1971E}
E.~Landis.
\newblock {\em Second order equations of elliptic and parabolic type}.
\newblock Nauka, Moscow, 1971.
\newblock [Russian]; English transl. in: \textit{Transl. Math. Monogr.},
  \textbf{171}, AMS, Providence, RI, 1998.

\bibitem{LaIb1995E}
E.~Landis and A.~Ibragimov.
\newblock Neumann problems in unbounded domains.
\newblock {\em Dokl. Akad. Nauk}, 343(1):17--18, 1995.
\newblock [Russian]; English transl. in: \textit{Dokl. Math.}, 52(1):11--12,
  1995.

\bibitem{Lf1966E}
N.~Landkof.
\newblock {\em Foundations of modern potential theory}.
\newblock Fizmatlit, Moscow, 1966.
\newblock [Russian]; English transl. in: \textit{Die Grundlehren der
  mathematischen Wissenschaften}, \textbf{180}, Springer-Verlag, New
  York-Heidelberg, 1972.

\bibitem{LZh2017}
D.~Li and K.~Zhang.
\newblock A note on the {H}arnack inequality for elliptic equations in
  divergence form.
\newblock {\em Proc. Amer. Math. Soc.}, 145(1):135--137, 2017.

\bibitem{Lyy1997}
Y.~Li.
\newblock Group invariant convex hypersurfaces with prescribed
  {G}auss-{K}ronecker curvature.
\newblock In {\em Multidimensional complex analysis and partial differential
  equations ({S}\~{a}o {C}arlos, 1995)}, volume 205 of {\em Contemp. Math.},
  pages 203--218. AMS, Providence, RI, 1997.

\bibitem{LNi2006}
Y.~Li and L.~Nirenberg.
\newblock A geometric problem and the {H}opf lemma. {I}.
\newblock {\em J. Eur. Math. Soc. (JEMS)}, 8(2):317--339, 2006.

\bibitem{LNi2006a}
Y.~Li and L.~Nirenberg.
\newblock A geometric problem and the {H}opf lemma. {II}.
\newblock {\em Chinese Ann. Math. Ser. B}, 27(2):193--218, 2006.

\bibitem{LiZh1995}
Y.~Li and M.~Zhu.
\newblock Uniqueness theorems through the method of moving spheres.
\newblock {\em Duke Math. J.}, 80(2):383--417, 1995.

\bibitem{L1909}
L.~Lichtenstein.
\newblock {\"U}ber eine {E}igenschaft der klassischen {G}reenschen {F}unktion.
\newblock {\em Math. Ann.}, 67:559--575, 1909.
\newblock [German].

\bibitem{L1912}
L.~Lichtenstein.
\newblock {B}eitr{\"a}ge zur {T}heorie der linearen partiellen
  {D}ifferentialgleichungen zweiter {O}rdnung vom elliptischen {T}ypus.
  {U}nendliche {F}olgen positiver {L}{\"o}sungen.
\newblock {\em Rend. Circ. Mat. Palermo}, 33:201--211, 1912.
\newblock [German].

\bibitem{L1918}
L.~Lichtenstein.
\newblock {\"U}ber einige {E}igenschaften der {G}leichgewichtsfiguren
  rotierender homogener fl{\"u}ssigkeiten, deren {T}eilchen einander nach dem
  {N}ewtonschen {G}esetz anziehen.
\newblock {\em Berl. Ber.}, 1918:1120--1135, 1918.
\newblock [German].

\bibitem{L1921}
L.~Lichtenstein.
\newblock {\"U}ber eine {E}igenschaft der klassischen {G}reenschen {F}unktion.
\newblock {\em Math. Zeitschr.}, 11:319--320, 1921.
\newblock [German].

\bibitem{L1924}
L.~Lichtenstein.
\newblock Neue {B}eitr{\"a}ge zur {T}heorie der linearen partiellen
  {D}ifferentialgleichungen zweiter {O}rdnung vom elliptischen {T}ypus.
\newblock {\em Math. Zeitschr.}, 20:194--212, 1924.
\newblock [German].

\bibitem{Li1985}
G.~Lieberman.
\newblock Regularized distance and its applications.
\newblock {\em Pacific J. Math.}, 117(2):329--352, 1985.

\bibitem{Li1986}
G.~Lieberman.
\newblock The {D}irichlet problem for quasilinear elliptic equations with
  continuously differentiable boundary data.
\newblock {\em {Commun. Partial Differ. Equ.}}, 11:167--229, 1986.

\bibitem{Li2000}
G.~Lieberman.
\newblock The maximum principle for equations with composite coefficients.
\newblock {\em Electron. J. Differ. Equations}, pages 1--17, 2000.

\bibitem{Li2003}
G.~Lieberman.
\newblock Maximum estimates for oblique derivative problems with right hand
  side in {$L^p,\ p<n$}.
\newblock {\em Manuscripta Math.}, 112(4):459--472, 2003.

\bibitem{Ll2015}
J.~Lierl.
\newblock Scale-invariant boundary {H}arnack principle on inner uniform domains
  in fractal-type spaces.
\newblock {\em Potential Anal.}, 43(4):717--747, 2015.

\bibitem{LlSC2014}
J.~Lierl and L.~Saloff-Coste.
\newblock Scale-invariant boundary {H}arnack principle in inner uniform
  domains.
\newblock {\em Osaka J. Math.}, 51(3):619--656, 2014.

\bibitem{Ls1981}
P.-L. Lions.
\newblock Two geometrical properties of solutions of semilinear problems.
\newblock {\em Applicable Anal.}, 12(4):267--272, 1981.

\bibitem{Ls1983}
P.-L. Lions.
\newblock A remark on {B}ony maximum principle.
\newblock {\em Proc. Amer. Math. Soc.}, 88(3):503--508, 1983.

\bibitem{Lt1959}
W.~Littman.
\newblock {A strong maximum principle for weakly {$L$}-subharmonic functions}.
\newblock {\em {J. Math. Mech.}}, 8:761--770, 1959.

\bibitem{Lt1963}
W.~Littman.
\newblock {Generalized subharmonic functions: monotonic approximations and an
  improved maximum principle}.
\newblock {\em {Ann. Sc. Norm. Super. Pisa, Sci. Fis. Mat., Ser. 3}},
  17:207--222, 1963.

\bibitem{LtStW1963E}
W.~Littman, G.~Stampacchia, and H.~Weinberger.
\newblock Regular points for elliptic equations with discontinuous
  coefficients.
\newblock {\em {Ann. Sc. Norm. Super. Pisa, Sci. Fis. Mat., Ser. 3}},
  17:43--77, 1963.

\bibitem{LMNN2020}
A.~Logunov, E.~Malinnikova, N.~Nadirashvili, and F.~Nazarov.
\newblock The {L}andis conjecture on exponential decay.
\newblock Preprint, arxiv.org/abs/2007.07034, 2020.

\bibitem{Lu1993}
Y.~Luo.
\newblock An {A}leksandrov-{B}akelman type maximum principle and applications.
\newblock {\em {J. Differ. Equ.}}, 101(2):213--231, 1993.

\bibitem{LuTr1991}
Y.~Luo and N.~Trudinger.
\newblock Linear second order elliptic equations with {V}enttsel' boundary
  conditions.
\newblock {\em Proc. Roy. Soc. Edinburgh Sect. A}, 118(3-4):193--207, 1991.

\bibitem{MT2016}
V.~Martino and G.~Tralli.
\newblock On the {H}opf-{O}leinik lemma for degenerate-elliptic equations at
  characteristic points.
\newblock {\em Calc. Var. Partial Differ. Equ.}, 55(5):Art. 115, 20, 2016.

\bibitem{Mz1967E}
V.~Maz'ya.
\newblock The behavior near the boundary of the solution of the dirichlet
  problem for an elliptic equation of the second order in divergence form.
\newblock {\em {Mat. Zametki}}, 2(2):209--220, 1967.
\newblock [Russian]; English transl. in: \textit{Math. Notes}, 2(2):610--617,
  1967.

\bibitem{Mz1970E}
V.~Maz'ya.
\newblock The continuity at a boundary point of the solutions of quasilinear
  elliptic equations.
\newblock {\em Vestnik Leningrad. Univ. Ser. Mat. Meh. Astr.},
  25(13(3)):42--55, 1970.
\newblock [Russian].

\bibitem{MMcO2007}
V.~Maz'ya and R.~McOwen.
\newblock Asymptotics for solutions of elliptic equations in double divergence
  form.
\newblock {\em Commun. Partial Differ. Equ.}, 32(1-3):191--207, 2007.

\bibitem{MKR2007}
P.~McKenna and W.~Reichel.
\newblock Gidas-{N}i-{N}irenberg results for finite difference equations:
  estimates of approximate symmetry.
\newblock {\em J. Math. Anal. Appl.}, 334(1):206--222, 2007.

\bibitem{MS2015}
H.~Mikayelyan and H.~Shahgholian.
\newblock Hopf's lemma for a class of singular/degenerate {PDE}'s.
\newblock {\em Ann. Acad. Sci. Fenn.}, 40:475--484, 2015.

\bibitem{Mi1967}
K.~Miller.
\newblock {Barriers on cones for uniformly elliptic operators}.
\newblock {\em {Ann. Mat. Pura Appl. (4)}}, 76:93--105, 1967.

\bibitem{Mn2003}
P.~Monk.
\newblock {\em Finite element methods for {M}axwell's equations}.
\newblock Numerical Mathematics and Scientific Computation. Oxford University
  Press, New York, 2003.

\bibitem{Mo1960}
J.~Moser.
\newblock A new proof of de {G}iorgi's theorem concerning the regularity
  problem for elliptic differential equations.
\newblock {\em {Commun. Pure Appl. Math.}}, 13:457--468, 1960.

\bibitem{Mo1961}
J.~Moser.
\newblock On {H}arnack's theorem for elliptic differential equations.
\newblock {\em {Commun. Pure Appl. Math.}}, 14:577--591, 1961.

\bibitem{M1894}
T.~Moutard.
\newblock Notes sur les {\'e}quations aux d{\'e}riv{\'e}es partielles.
\newblock {\em J. de l'{\'E}cole Polytechnique}, 64:55--69, 1894.
\newblock [French].

\bibitem{Mu2011}
R.~Musina.
\newblock Planar loops with prescribed curvature: existence, multiplicity and
  uniqueness results.
\newblock {\em Proc. Amer. Math. Soc.}, 139(12):4445--4459, 2011.

\bibitem{MN2014}
R.~Musina and A.~Nazarov.
\newblock On fractional {L}aplacians.
\newblock {\em Commun. Partial Differ. Equ.}, 39(9):1780--1790, 2014.

\bibitem{MN2016}
R.~Musina and A.~Nazarov.
\newblock On fractional {L}aplacians---2.
\newblock {\em Ann. Inst. H. Poincar\'{e} Anal. Non Lin\'{e}aire},
  33(6):1667--1673, 2016.

\bibitem{MN2019}
R.~Musina and A.~Nazarov.
\newblock Strong maximum principles for fractional {L}aplacians.
\newblock {\em Proc. Roy. Soc. Edinburgh Sect. A}, 149(5):1223--1240, 2019.

\bibitem{Nd1981E}
N.~Nadirashvili.
\newblock Lemma on the interior derivative and uniqueness of the solution of
  the second boundary value problem for second-order elliptic equations.
\newblock {\em Dokl. Akad. Nauk SSSR}, 261(4):804--808, 1981.
\newblock [Russian]; English transl. in: \textit{Soviet Math. Dokl.},
  24:598--601, 1981.

\bibitem{Nd1983E}
N.~Nadirashvili.
\newblock On the question of the uniqueness of the solution of the second
  boundary value problem for second-order elliptic equations.
\newblock {\em Mat. Sb. (N.S.)}, 122(164))(3(11)):341--359, 1983.
\newblock [Russian]; English transl. in: \textit{Math. USSR, Sb.}, 50:325--341,
  1985.

\bibitem{Nd1988E}
N.~Nadirashvili.
\newblock Some estimates in a problem with an oblique derivative.
\newblock {\em Izv. Akad. Nauk SSSR Ser. Mat.}, 52(5):1982--1090, 1988.
\newblock [Russian]; English transl. in: \textit{Math. USSR, Izv.},
  33(2):403--411, 1989.

\bibitem{Na1990E}
A.~Nazarov.
\newblock H\"{o}lder estimates for bounded solutions of problems with an
  oblique derivative for parabolic equations of nondivergence structure.
\newblock {\em Problems in math. analysis}, 11:37--46, 1990.
\newblock [Russian]; English transl. in: \textit{J. Soviet Math.},
  64(6):1247--1252, 1993.

\bibitem{Na2002E}
A.~Nazarov.
\newblock Boundary estimates for solutions to {V}enttsel's problem for
  parabolic and elliptic equations in a domain with boundary of class
  {$W_{q-1}^2$}.
\newblock {\em Problems in math. analysis}, 24:181--204, 2002.
\newblock [Russian]; English transl. in: \textit{J. Math. Sci. (N. Y.)},
  123(6):4527--4538, 2004.

\bibitem{Na2001E}
A.~Nazarov.
\newblock Estimates of the maximum for solutions of elliptic and parabolic
  equations in terms of weighted norms of the right-hand side.
\newblock {\em Algebra Anal.}, 13(2):151--164, 2002.
\newblock [Russian]; English transl. in: \textit{St. Petersbg. Math. J.},
  13(2):269--279, 2002.

\bibitem{Na2005E}
A.~Nazarov.
\newblock The maximum principle of {A}.{D}. {A}leksandrov.
\newblock {\em Sovrem. Mat. Prilozh.}, 29:127--143, 2005.
\newblock [Russian]; English transl. in: \textit{J. Math. Sci. (N. Y.)},
  142(3):2154--2171, 2007.

\bibitem{Na2012}
A.~Nazarov.
\newblock A centennial of the {Z}aremba-{H}opf-{O}leinik lemma.
\newblock {\em SIAM J. Math. Anal.}, 44(1):437--453, 2012.

\bibitem{NU1985E}
A.~Nazarov and N.~Ural'tseva.
\newblock Convex-monotone hulls and an estimate of the maximum of the solution
  of a parabolic equation.
\newblock {\em Zap. Nauchn. Sem. Leningrad. Otdel. Mat. Inst. Steklov. (LOMI)},
  147:95--109, 1985.
\newblock [Russian]; English transl. in: \textit{J. Soviet Math.}, 37:851--859,
  1987.

\bibitem{NU2009}
A.~Nazarov and N.~Uraltseva.
\newblock Qualitative properties of solutions to elliptic and parabolic
  equations with unbounded lower-order coefficients.
\newblock Preprint 2009-05, St. Petersburg Math. Soc. El. Prepr. Archive, 2009.

\bibitem{NU2011E}
A.~Nazarov and N.~Uraltseva.
\newblock The {H}arnack inequality and related properties of solutions of
  elliptic and parabolic equations with divergence-free lower-order
  coefficients.
\newblock {\em Algebra Anal.}, 23(1):136--168, 2011.
\newblock [Russian]; English transl. in: \textit{St. Petersburg Math. J.},
  23(1):93--115, 2012.

\bibitem{NV1994}
I.~Netuka and J.~Vesel\'{y}.
\newblock Mean value property and harmonic functions.
\newblock In {\em Classical and modern potential theory and applications
  ({C}hateau de {B}onas, 1993)}, volume 430 of {\em NATO Adv. Sci. Inst. Ser. C
  Math. Phys. Sci.}, pages 359--398. Kluwer Acad. Publ., Dordrecht, 1994.

\bibitem{N1888}
C.~Neumann.
\newblock Ueber die {M}ethode des arithmetischen {M}ittels. {Z}weite
  {A}bhandlung.
\newblock Leipz. Abh. XIV. 565--726, 1888.
\newblock [German].

\bibitem{Ne1962}
J.~{Ne\v{c}as}.
\newblock {Sur une m\'ethode pour r\'esoudre les \'equations aux d\'eriv\'ees
  partielles du type elliptique, voisine de la variationnelle}.
\newblock {\em {Ann. Sc. Norm. Super. Pisa, Sci. Fis. Mat., Ser. 3}},
  16:305--326, 1962.
\newblock [French].

\bibitem{Ny2013}
K.~Nystr\"{o}m.
\newblock {$p$}-harmonic functions in the {H}eisenberg group: boundary
  behaviour in domains well-approximated by non-characteristic hyperplanes.
\newblock {\em Math. Ann.}, 357(1):307--353, 2013.

\bibitem{Od1967a}
J.~Oddson.
\newblock {On the boundary point principle for elliptic equations in the
  plane}.
\newblock {\em {Bull. Amer. Math. Soc.}}, 74:666--670, 1968.

\bibitem{Od1967}
J.~Oddson.
\newblock {Some solutions of elliptic extremal equations in the plane}.
\newblock {\em {Matematiche}}, 23:273--289, 1968.

\bibitem{O1952E}
O.~Ole\u{\i}nik.
\newblock On properties of solutions of certain boundary problems for equations
  of elliptic type.
\newblock {\em Mat. Sbornik N.S.}, 30 (72):695--702, 1952.
\newblock [Russian].

\bibitem{ORa1971E}
O.~Ole\u{\i}nik and E.~Radkevi\v{c}.
\newblock Second order equations with nonnegative characteristic form.
\newblock In {\em Itogi {Nauki}}, {Ser}. {Mat}., {Mat}. {Anal}. 1969, pages
  7--252. VINITI, 1971.
\newblock [Russian].

\bibitem{Ol1984}
V.~Oliker.
\newblock Hypersurfaces in {${\R}\sp{n+1}$} with prescribed {G}aussian
  curvature and related equations of {M}onge-{A}mp\`ere type.
\newblock {\em Commun. Partial Differ. Equ.}, 9(8):807--838, 1984.

\bibitem{OSV2020}
F.~Oliva, B.~Sciunzi, and G.~Vaira.
\newblock Radial symmetry for a quasilinear elliptic equation with a critical
  {S}obolev growth and {H}ardy potential.
\newblock {\em J. Math. Pures Appl. (9)}, 140:89--109, 2020.

\bibitem{OPo2016}
L.~Orsina and A.~Ponce.
\newblock Strong maximum principle for {S}chr\"{o}dinger operators with
  singular potential.
\newblock {\em Ann. Inst. H. Poincar\'{e} Anal. Non Lin\'{e}aire},
  33(2):477--493, 2016.

\bibitem{OPS2012E}
S.~Oshchepkova, O.~Penkin, and D.~Savasteev.
\newblock Strong maximum principle for an elliptic operator on a stratified
  set.
\newblock {\em Mat. Zametki}, 92(2):276--290, 2012.
\newblock [Russian]; English transl. in: \textit{Math. Notes}, 92(2):249--259,
  2012.

\bibitem{OPS2014E}
S.~Oshchepkova, O.~Penkin, and D.~Savasteev.
\newblock The normal derivative lemma for the {L}aplacian on a polyhedral set.
\newblock {\em Mat. Zametki}, 96(1):116--125, 2014.
\newblock [Russian]; English transl. in: \textit{Math. Notes}, 96(1):122--129,
  2014.

\bibitem{Pd1997}
P.~Padilla.
\newblock Symmetry properties of positive solutions of elliptic equations on
  symmetric domains.
\newblock {\em Appl. Anal.}, 64(1-2):153--169, 1997.

\bibitem{P1892}
A.~Paraf.
\newblock {\em Sur le probl{\`e}me de {D}irichlet et son extension au cas de
  {l'\'e}quation lin{\'e}aire g{\'e}n{\'e}rale du second ordre}.
\newblock PhD thesis, Acad{\'e}mie de Paris, 1892.
\newblock [French].

\bibitem{PaSt1973}
L.~Payne and I.~Stakgold.
\newblock On the mean value of the fundamental mode in the fixed membrane
  problem.
\newblock {\em Applicable Anal.}, 3:295--306, 1973.

\bibitem{Pe1966}
J.~Peetre.
\newblock Espaces d'interpolation et th{\'e}or{\`e}me de {S}oboleff.
\newblock {\em {Ann. Inst. Fourier}}, 16(1):279--317, 1966.
\newblock [French].

\bibitem{PhL1908}
E.~{Phragm{\'e}n} and E.~{Lindel{\"o}f}.
\newblock {Sur une extension d'un principe classique de l'analyse et sur
  quelques propri{\'e}t{\'e}s de fonctions monog{\`e}nes dans le voisinage d'un
  point singulier}.
\newblock {\em {Acta Math.}}, 31:381--406, 1908.
\newblock [French].

\bibitem{P1905}
E.~Picard.
\newblock {\em {T}rait{\'e} {d'a}nalyse}, volume~II.
\newblock Gauthier-Villars, Paris, 1905.
\newblock [French].

\bibitem{Pi1927}
M.~Picone.
\newblock Maggiorazione degli integrali delle equazioni lineari
  ellittico-paraboliche alle derivate parziali del second'ordine.
\newblock {\em {Atti Accad. Naz. Lincei, Rend., VI. Ser.}}, 5:138--143, 1927.
\newblock [Italian].

\bibitem{PoW2020}
A.~Ponce and N.~Wilmet.
\newblock The {H}opf lemma for the {S}chr\"{o}dinger operator.
\newblock {\em Adv. Nonlinear Stud.}, 20(2):459--475, 2020.

\bibitem{PW1984}
M.~Protter and H.~Weinberger.
\newblock {\em Maximum principles in differential equations}.
\newblock Springer-Verlag, New York, 1984.
\newblock Corrected reprint of the 1967 original.

\bibitem{Pu1957}
C.~Pucci.
\newblock Propriet\`a di massimo e minimo delle soluzioni di equazioni a
  derivate parziali del secondo ordine di tipo ellittico e parabolico. {I}.
\newblock {\em Atti Accad. Naz. Lincei. Rend. Cl. Sci. Fis. Mat. Nat. (8)},
  23(6):370--375, 1957.
\newblock [Italian].

\bibitem{Pu1958}
C.~Pucci.
\newblock Propriet\`a di massimo e minimo delle soluzioni di equazioni a
  derivate parziali del secondo ordine di tipo ellittico e parabolico. {II}.
\newblock {\em Atti Accad. Naz. Lincei. Rend. Cl. Sci. Fis. Mat. Nat. (8)},
  24(1):3--6, 1958.
\newblock [Italian].

\bibitem{Pu1966a}
C.~Pucci.
\newblock Limitazioni per soluzioni di equazioni ellittiche.
\newblock {\em Ann. Mat. Pura Appl. (4)}, 74:15--30, 1966.
\newblock [Italian].

\bibitem{Pu1966b}
C.~Pucci.
\newblock Operatori ellittici estremanti.
\newblock {\em Ann. Mat. Pura Appl. (4)}, 72:141--170, 1966.
\newblock [Italian].

\bibitem{Pu1966}
C.~Pucci.
\newblock Su una limitazione per soluzioni di equazioni ellittiche.
\newblock {\em Boll. Unione Mat. Ital., III. Ser.}, 21:228--233, 1966.
\newblock [Italian].

\bibitem{Pu1986E}
C.~Pucci.
\newblock Maximizing elliptic operators, applications and conjectures.
\newblock In {\em Partial differential equations, {Proc}. {Int}. {Conf}.,
  {Novosibirsk} 1983}, pages 167--172. Izdat. ``Nauka'', Novosibirsk, 1986.
\newblock [Russian].

\bibitem{PS2000}
P.~Pucci and J.~Serrin.
\newblock A note on the strong maximum principle for elliptic differential
  inequalities.
\newblock {\em J. Math. Pures Appl. (9)}, 79(1):57--71, 2000.

\bibitem{PS2004}
P.~Pucci and J.~Serrin.
\newblock The strong maximum principle revisited.
\newblock {\em J. Differ. Equ.}, 196(1):1--66, 2004.

\bibitem{PS2007}
P.~Pucci and J.~Serrin.
\newblock {\em The maximum principle}, volume~73 of {\em Progress in Nonlinear
  Differential Equations and their Applications}.
\newblock Birkh\"auser Verlag, Basel, 2007.

\bibitem{PT2009}
F.~Punzo and A.~Tesei.
\newblock {Uniqueness of solutions to degenerate elliptic problems with
  unbounded coefficients}.
\newblock {\em {Ann. Inst. Henri Poincar\'e, Anal. Non Lin\'eaire}},
  26(5):2001--2024, 2009.

\bibitem{Ra2008E}
E.~Radkevich.
\newblock Equations with nonnegative characteristic form. {I}.
\newblock {\em Sovrem. Mat. Prilozh.}, 55, 2008.
\newblock [Russian]; English transl. in: \textit{J. Math. Sci. (N. Y.)},
  158(3):297--452, 2009.

\bibitem{Ra2008aE}
E.~Radkevich.
\newblock Equations with nonnegative characteristic form. {II}.
\newblock {\em Sovrem. Mat. Prilozh.}, 56, 2008.
\newblock [Russian]; English transl. in: \textit{J. Math. Sci. (N. Y.)},
  158(4):453--604, 2009.

\bibitem{Ri1938}
M.~Riesz.
\newblock Int{\'e}grales de {R}iemann--{L}iouville et potentiels.
\newblock {\em {Acta Litt. Sci. Szeged}}, 9:1--42, 1938.
\newblock [French].

\bibitem{R1970}
R.~Rockafellar.
\newblock {\em Convex analysis}, volume~28 of {\em Princeton Math. Ser.}
\newblock Princeton University Press, Princeton, NJ, 1970.

\bibitem{RO2016}
X.~Ros-Oton.
\newblock Nonlocal elliptic equations in bounded domains: a survey.
\newblock {\em Publ. Mat.}, 60(1):3--26, 2016.

\bibitem{RO2018}
X.~Ros-Oton.
\newblock Boundary regularity, {P}ohozaev identities and nonexistence results.
\newblock In {\em Recent developments in nonlocal theory}, pages 335--358. De
  Gruyter, Berlin, 2018.

\bibitem{ROS2014}
X.~Ros-Oton and J.~Serra.
\newblock The {D}irichlet problem for the fractional {L}aplacian: regularity up
  to the boundary.
\newblock {\em J. Math. Pures Appl. (9)}, 101(3):275--302, 2014.

\bibitem{ROS2019}
X.~Ros-Oton and J.~Serra.
\newblock The boundary {H}arnack principle for nonlocal elliptic operators in
  non-divergence form.
\newblock {\em Potential Anal.}, 51(3):315--331, 2019.

\bibitem{SdL2015}
J.~Sabina~de Lis.
\newblock Hopf maximum principle revisited.
\newblock {\em Electron. J. Differ. Equations}, 2015(115):1--9, 2015.

\bibitem{Sa1980E}
M.~Safonov.
\newblock Harnack's inequality for elliptic equations and {H}\"{o}lder property
  of their solutions.
\newblock {\em Zap. Nauchn. Sem. Leningrad. Otdel. Mat. Inst. Steklov. (LOMI)},
  96:272--287, 1980.
\newblock [Russian]; English transl. in: \textit{J. Soviet Math.},
  21(5):851--863, 1983.

\bibitem{Sa2008}
M.~Safonov.
\newblock Boundary estimates for positive solutions to second order elliptic
  equations.
\newblock Preprint, arxiv.org/abs/0810.0522, 2008.

\bibitem{Sa2010}
M.~Safonov.
\newblock Non-divergence elliptic equations of second order with unbounded
  drift.
\newblock In {\em Nonlinear partial differential equations and related topics},
  volume 229 of {\em Amer. Math. Soc. Transl. Ser. 2}, pages 211--232. AMS,
  Providence, RI, 2010.

\bibitem{Sa2015E}
M.~Safonov.
\newblock Narrow domains and the {H}arnack inequality for elliptic equations.
\newblock {\em Algebra Anal.}, 27(3):220--237, 2015.
\newblock \textit{St. Petersburg Math. J.}, 27(3):509--522, 2016.

\bibitem{Sa2018}
M.~Safonov.
\newblock {On the boundary estimates for second-order elliptic equations}.
\newblock {\em {Complex Var. Elliptic Equ.}}, 63(7-8):1123--1141, 2018.

\bibitem{Sa2020E}
M.~Safonov.
\newblock Growth theorems for metric spaces with applications to {PDE}.
\newblock {\em Algebra Anal.}, 32(4):271--284, 2020.
\newblock \textit{St. Petersburg Math. J.}, 32(4):809--818, 2021.

\bibitem{Sch1968}
M.~Schechter.
\newblock On the invariance of the essential spectrum of an arbitrary operator.
  {III}.
\newblock {\em Proc. Cambridge Philos. Soc.}, 64:975--984, 1968.

\bibitem{Sch1971}
M.~Schechter.
\newblock {\em Spectra of partial differential operators}, volume~14 of {\em
  North-Holland Series in Applied Mathematics and Mechanics}.
\newblock North-Holland Publishing Co., Amsterdam-London; American Elsevier
  Publishing Co., Inc., New York, 1971.

\bibitem{Sch1986}
M.~Schechter.
\newblock {\em Spectra of partial differential operators}, volume~14 of {\em
  North-Holland Series in Applied Mathematics and Mechanics}.
\newblock North-Holland Publishing Co., Amsterdam, 2nd edition, 1986.

\bibitem{SSSZ2012}
G.~Seregin, L.~Silvestre, V.~\v{S}ver\'{a}k, and A.~Zlato\v{s}.
\newblock On divergence-free drifts.
\newblock {\em J. Differ. Equ.}, 252(1):505--540, 2012.

\bibitem{S1956}
J.~Serrin.
\newblock On the {H}arnack inequality for linear elliptic equations.
\newblock {\em J. Analyse Math.}, 4:292--308, 1955/56.

\bibitem{S1964}
J.~Serrin.
\newblock {Local behavior of solutions of quasi-linear equations}.
\newblock {\em {Acta Math.}}, 111:247--302, 1964.

\bibitem{S1969}
J.~Serrin.
\newblock On surfaces of constant mean curvature which span a given space
  curve.
\newblock {\em Math. Z.}, 112:77--88, 1969.

\bibitem{S1971}
J.~Serrin.
\newblock A symmetry problem in potential theory.
\newblock {\em Arch. Rat. Mech. Anal.}, 43:304--318, 1971.

\bibitem{Sl2007}
L.~Silvestre.
\newblock Regularity of the obstacle problem for a fractional power of the
  {L}aplace operator.
\newblock {\em Comm. Pure Appl. Math.}, 60(1):67--112, 2007.

\bibitem{Si1990}
C.~Simader.
\newblock An elementary proof of {H}arnack's inequality for {S}chr{\"o}dinger
  operators and related topics.
\newblock {\em {Math. Zeitschr.}}, 203(1):129--152, 1990.

\bibitem{S2018}
B.~Sirakov.
\newblock Boundary {H}arnack estimates and quantitative strong maximum
  principles for uniformly elliptic {PDE}.
\newblock {\em Int. Math. Res. Not. IMRN}, 2018(24):7457--7482, 2018.

\bibitem{S2022}
B.~Sirakov.
\newblock Global integrability and weak {H}arnack estimates for elliptic
  {P}{D}{E}s in divergence form.
\newblock {\em Anal. PDE}, 15(1):197--216, 2022.

\bibitem{SSu2021}
B.~Sirakov and P.~Souplet.
\newblock The {V}\'{a}zquez maximum principle and the {L}andis conjecture for
  elliptic {PDE} with unbounded coefficients.
\newblock {\em Adv. Math.}, 387:107838, 2021.

\bibitem{SkVo2021}
I.~Skrypnik and M.~Voitovych.
\newblock {${\cal{B}}_1$} classes of {D}e
  {G}iorgi--{L}adyzhenskaya--{U}ral'tseva and their applications to elliptic
  and parabolic equations with generalized {O}rlicz growth conditions.
\newblock {\em Nonlinear Anal.}, 202:112135, 30, 2021.

\bibitem{S1958E}
E.~Solomencev.
\newblock Sur les valeurs limites des fonctions sousharmoniques.
\newblock {\em Czech. Math. J.}, 8(83):520--536, 1958.
\newblock [Russian].

\bibitem{Sp1981}
R.~Sperb.
\newblock {\em Maximum principles and their applications}, volume 157 of {\em
  Mathematics in Science and Engineering}.
\newblock Academic Press, Inc. [Harcourt Brace Jovanovich, Publishers], New
  York-London, 1981.

\bibitem{St1965}
G.~Stampacchia.
\newblock Le probl{\`e}me de {D}irichlet pour les {\'e}quations elliptiques du
  second ordre {\`a} coefficients discontinus.
\newblock {\em {Ann. Inst. Fourier}}, 15(1):189--257, 1965.
\newblock [French].

\bibitem{Stu1956}
F.~Stummel.
\newblock Singul{\"a}re elliptische {D}ifferentialoperatoren in {H}ilbertschen
  {R}{\"a}umen.
\newblock {\em {Math. Ann.}}, 132:150--176, 1956.
\newblock [German].

\bibitem{Sw1997}
G.~Sweers.
\newblock Hopf's lemma and two dimensional domains with corners.
\newblock {\em Rend. Istit. Mat. Univ. Trieste}, XXVIII:383--419, 1997.

\bibitem{To1983}
P.~Tolksdorf.
\newblock On the {D}irichlet problem for quasilinear equations in domains with
  conical boundary points.
\newblock {\em Commun. Partial Differ. Equ.}, 8(7):773--817, 1983.

\bibitem{TW1983}
A.~Treibergs and S.~Wei.
\newblock Embedded hyperspheres with prescribed mean curvature.
\newblock {\em J. Differential Geom.}, 18(3):513--521, 1983.

\bibitem{Tb1978}
H.~Triebel.
\newblock {\em Interpolation theory, function spaces, differential operators}.
\newblock VEB Deutscher Verlag der Wissenschaften, Berlin, 1978.

\bibitem{Tr1967}
N.~Trudinger.
\newblock On {H}arnack type inequalities and their application to quasilinear
  elliptic equations.
\newblock {\em {Commun. Pure Appl. Math.}}, 20:721--747, 1967.

\bibitem{Tr1971}
N.~Trudinger.
\newblock On the regularity of generalized solutions of linear, non-uniformly
  elliptic equations.
\newblock {\em Arch. Rat. Mech. Anal.}, 42:50--62, 1971.

\bibitem{Tr1973}
N.~Trudinger.
\newblock {Linear elliptic operators with measurable coefficients}.
\newblock {\em {Ann. Sc. Norm. Super. Pisa, Sci. Fis. Mat., Ser. 3}},
  27:265--308, 1973.

\bibitem{Tr1981}
N.~Trudinger.
\newblock {Harnack inequalities for nonuniformly elliptic divergence structure
  equations}.
\newblock {\em {Invent. Math.}}, 64:517--531, 1981.

\bibitem{Tr2020}
N.~Trudinger.
\newblock Remarks on the {P}ucci conjecture.
\newblock {\em Indiana Univ. Math. J.}, 69(1):109--118, 2020.

\bibitem{Vz1984}
J.~V\'{a}zquez.
\newblock A strong maximum principle for some quasilinear elliptic equations.
\newblock {\em Appl. Math. Optim.}, 12(3):191--202, 1984.

\bibitem{VM1967E}
G.~Verzhbinskij and V.~Maz'ya.
\newblock Asymptotic behavior of solutions of the {Dirichlet} problem near a
  nonregular boundary.
\newblock {\em Dokl. Akad. Nauk SSSR}, 176(3):498--501, 1967.
\newblock [Russian]; English transl. in: \textit{Soviet Math. Dokl.},
  8:1110--1113, 1967.

\bibitem{V1963}
R.~V\'{y}born\'{y}.
\newblock On a certain extension of the maximum principle.
\newblock In {\em Differential {E}quations and {T}heir {A}pplications ({P}roc.
  {C}onf., {P}rague, 1962)}, pages 223--228. Publ. House Czechoslovak Acad.
  Sci., Prague; Academic Press, New York, 1963.

\bibitem{V1964}
R.~V{\'{y}}born{\'{y}}.
\newblock {\"U}ber das erweiterte {M}aximumprinzip.
\newblock {\em Czech. Math. J.}, 14:116--120, 1964.
\newblock [German].

\bibitem{W1967}
K.-O. Widman.
\newblock Inequalities for the {G}reen function and boundary continuity of the
  gradient of solutions of elliptic differential equations.
\newblock {\em Math. Scand.}, 21:17--37, 1967.

\bibitem{Wu1978}
J.~Wu.
\newblock Comparisons of kernel functions, boundary {H}arnack principle and
  relative {F}atou theorem on {L}ipschitz domains.
\newblock {\em Ann. Inst. Fourier (Grenoble)}, 28(4):147--167, 1978.

\bibitem{Za2002}
P.~Zamboni.
\newblock H{\"o}lder continuity for solutions of linear degenerate elliptic
  equations under minimal assumptions.
\newblock {\em {J. Differ. Equ.}}, 182(1):121--140, 2002.

\bibitem{Z1910E}
S.~Zaremba.
\newblock Sur un probl\`eme mixte relatif {\`a} l' {\'e}quation de {L}aplace.
\newblock {\em Bull. Acad. Sci. Cracovie. Cl. Sci. Math. Nat. Ser. A}, pages
  313--344, 1910.
\newblock [French].

\bibitem{Zh1996}
Q.~Zhang.
\newblock A {H}arnack inequality for the equation {$\nabla(a\nabla u)+b\nabla
  u=0$}, when {$|b|\in K_{n+1}$}.
\newblock {\em Manuscripta Math.}, 89(1):61--77, 1996.

\bibitem{Zh2004}
Q.~Zhang.
\newblock {A strong regularity result for parabolic equations}.
\newblock {\em {Commun. Math. Phys.}}, 244(2):245--260, 2004.

\bibitem{ZhY2009}
Q.~Zheng and X.~Yao.
\newblock Higher-order {K}ato class potentials for {S}chr{\"{o}}dinger
  operators.
\newblock {\em Bull. London Math. Soc.}, 41(2):293--301, 2009.

\end{thebibliography}

\end{document}